\documentclass[11pt]{siamart171218}

\usepackage{graphicx}
\usepackage[nameinlink]{cleveref} 
\usepackage{bm}
\usepackage{comment}
\usepackage{pstool}

\Crefname{equation}{}{}
\usepackage{xcolor}
\usepackage{multirow}
\graphicspath{{./Figures/}}
\definecolor{CeruleanRef}{RGB}{12,127,172}
\hypersetup{colorlinks=true,allcolors=CeruleanRef}
\hypersetup{pdfstartview=FitB,pdfpagemode=UseNone}
\newtheorem{assumption}{Assumption}[section]
\newtheorem{condition}{Condition}[section]
\newtheorem{test}{Test}[section]
\crefname{algocfline}{algorithm}{algorithms}
\crefname{assumption}{assumption}{assumptions}
\crefname{condition}{condition}{conditions}
\crefname{test}{test}{tests}

\usepackage[ruled,vlined,algo2e,linesnumbered]{algorithm2e}
\usepackage{algorithm,algorithmic}

\theoremstyle{plain}
\theoremheaderfont{\normalfont\sc}
\theorembodyfont{\normalfont}
\theoremseparator{.}
\theoremsymbol{}
\newtheorem{remark}{Remark}[section]
\crefname{remark}{remark}{remarks}

\usepackage{tikz}
\usepackage[sort]{cite}
\usepackage[british]{babel}
\usepackage{pgf}
\usepackage{pgfplots}
\usepackage{pgfplotstable}
\usepackage{shellesc}
\usepackage[shortlabels]{enumitem}
\pgfplotsset{compat=newest}
\usetikzlibrary{spy}
\usetikzlibrary{shapes}
\usetikzlibrary{backgrounds}
\usetikzlibrary{shadows}
\usetikzlibrary{matrix}
\usetikzlibrary{external}
\usepgfplotslibrary{fillbetween}
\usepackage{soul}

\usepackage{amsmath,amssymb}
\usepackage{mathtools}
\DeclareMathOperator*{\argmin}{arg\,min}
\DeclareMathOperator*{\argmax}{arg\,max}

\DeclareMathOperator{\proj}{proj}
\DeclareMathOperator{\prox}{prox}
\DeclareMathOperator{\dd}{d\!}
\let\div\relax \DeclareMathOperator{\div}{div}

\newcommand{\defeq}{\stackrel{\rm def}{=}}

\def\E{\mathbb{E}}
\def\R{\mathbb{R}}
\def\X{\mathcal{X}}
\def\L{\mathcal{L}}
\def\E{\mathbb{E}}
\def\N{\mathbb{N}}

\newcommand{\ck}[1]{{\color{black}#1}}
\newcommand{\rb}[1]{{\color{black}#1}}

\newcommand{\rev}[1]{{\color{black}#1}}
\newcommand{\ckrev}[1]{{\color{black}#1}}
\newcommand{\rbreview}[1]{{\color{black}#1}}

\begin{document}

\title{
        An Adaptive Sampling Augmented Lagrangian Method for Stochastic Optimization with Deterministic Constraints
                                                \thanks{Submitted to the editors \today.
\funding{This work was performed under the auspices of the U.S.~Department of Energy by Lawrence Livermore National Laboratory under Contract DE-AC52-07NA27344 and the LLNL-LDRD Program under Project tracking No.~22-ERD-009. Release number LLNL-JRNL-848453.}}}

\author{
    Raghu~Bollapragada\thanks{\protect
        Operations Research and Industrial Engineering,
                The University of Texas at Austin,
        Austin, TX 78712
        (\email{raghu.bollapragada@utexas.edu}, \email{cem.karamanli@utexas.edu})
    }
    \and
                Cem~Karamanli\footnotemark[2]
    \and
	Brendan~Keith\thanks{\protect
        Division of Applied Mathematics,
        Brown University,
        Providence, RI 02912 USA
        (\email{brendan\_keith@brown.edu}).
    }
    \and
    Boyan~Lazarov\thanks{\protect
         Center for Design Optimization,
         Lawrence Livermore National Laboratory,
         Livermore, CA 94550
         (\email{lazarov2@llnl.gov})
    }
    \and
    Socratis~Petrides\thanks{\protect
         Center for Applied Scientific Computing,
         Lawrence Livermore National Laboratory,
         Livermore, CA 94550
        (\email{petrides1@llnl.gov}, \email{wang125@llnl.gov})
    }
    \and
    Jingyi~Wang\footnotemark[5]
}

\date{\today}

\maketitle

\begin{center}
    \small
    \medskip
  \emph{Dedicated with respect and admiration to Leszek Demkowicz on the occasion of his 70th birthday anniversary.}
  \medskip
\end{center}

\begin{abstract}
The primary goal of this paper is to provide an efficient solution algorithm based on the augmented Lagrangian framework for optimization problems with a stochastic objective function and deterministic constraints. Our main contribution is combining the augmented Lagrangian framework with adaptive sampling, resulting in an efficient optimization methodology validated with practical examples. To achieve the presented efficiency, we consider inexact solutions for the augmented Lagrangian subproblems, and through an adaptive sampling mechanism, we control the variance in the gradient estimates. Furthermore, we analyze the theoretical performance of the proposed scheme by showing equivalence to a gradient descent algorithm on a Moreau envelope function, and we prove sublinear convergence for convex objectives and linear convergence for strongly convex objectives with affine equality constraints. The worst-case sample complexity of the resulting algorithm, for an arbitrary choice of penalty parameter in the augmented Lagrangian function, is $\mathcal{O}(\epsilon^{-3-\delta})$, where $\epsilon > 0$ is the expected error of the solution and $\delta > 0$ is a user-defined parameter. If the penalty parameter is chosen to be $\mathcal{O}(\epsilon^{-1})$, we demonstrate that the result can be improved to $\mathcal{O}(\epsilon^{-2})$, which is competitive with the other methods employed in the literature.
Moreover, if the objective function is strongly convex with affine equality constraints, we obtain $\mathcal{O}(\epsilon^{-1}\log(1/\epsilon))$ complexity. Finally, we empirically verify the performance of our adaptive sampling augmented Lagrangian framework in machine learning optimization and engineering design problems, including topology optimization of a heat sink with environmental uncertainty.

\end{abstract}

\section{Introduction} \label{sec:introduction}

We consider constrained stochastic optimization problems of the form
\begin{equation}
\label{eq:GeneralFormulation}
	\min_{x \in \X}~ f(x)
	\quad
	\text{subject to~}
	c(x) = 0
     \,,
\end{equation}
where the objective function $f\colon \mathbb{R}^n \to \mathbb{R}$ is the expected value $f(x) = \E_{\zeta}[F(x,\zeta)]$ of smooth random functions $F(\cdot,\zeta)\colon \mathbb{R}^n \to \mathbb{R}$, the constraint set $\X \subset \mathbb{R}^n$ is compact and convex, and the constraint function $c\colon \mathbb{R}^n \to \mathbb{R}^m$, 
\begin{equation}\label{eq:linearconstraint}
c(x) \defeq Ax - b
\,,
\end{equation}
is an affine map with $A\in\R^{m\times n}$ and $b \in \R^m$.
Our primary motivation is to develop feasible strategies for solving \rev{optimal} design problems with manufacturing and operational uncertainties \cite{rockafellar2010buffered, Chen2010a, Jansen2013b, Andreassen2014} (cf.~\Cref{sub:optimal_truss_design,sub:topology_optimization}) by efficiently solving optimization problems of the form \cref{eq:GeneralFormulation}. Due to the inherently high computational cost, current design problems are often limited to low-dimensional sources of uncertainty or involve smoothly-varying random fields, which can be parameterized by a truncated series expansion with a small number of discrete random variables \cite{khristenko2019analysis}. The above limitations restrict the practical applicability of some optimization approaches and lead to simplified heuristic procedures requiring subsequent manual intervention and suboptimal design performance \cite{Lazarov2016}. Therefore, designing an efficient and robust optimization framework to address these challenges is crucial. Moreover, constrained stochastic optimization problems like~\cref{eq:GeneralFormulation} also commonly arise in other applications such as machine learning and finance; see, e.g., \cite{cornuejols2006optimization,barocas2016big,royset2022risk,shapiro2021lectures} and references therein.

One of the well-known techniques
to solve constrained optimization problems is the augmented Lagrangian method \cite{hestenes1969multiplier,powell1969method,bertsekas2014constrained,curtis2015adaptive,curtis2016adaptive}. This method transforms the original constrained optimization problem \cref{eq:GeneralFormulation} into a sequence of subproblems where the constraint violation is penalized in the objective.
The main advantage of this transformation is that it enables using efficient algorithms for solving the subproblems.
On the other hand, the major drawback is that multiple subproblems must be solved sequentially.
To mitigate the cost of solving the subproblems, inexact solution mechanisms are widely used \cite{rockafellar1976augmented,kang2015inexact,sahin2019inexact,lan2016iteration,xu2021iteration,li2021rate,li2021augmented}.
Although these mechanisms are well-understood for deterministic problems, the literature on their usage in stochastic settings is limited \cite{ouyang2013stochastic,li2022stochastic}. Indeed, from our perspective, the main challenge in extending the augmented Lagrangian framework to stochastic approximation techniques  lies in defining inexactness criteria for the stochastic methods used to solve the subproblems. In this work, we propose stochastic inexactness termination conditions that address this gap and guarantee convergence in expectation.

Adaptive sampling is a powerful technique that is used in stochastic optimization to control the accuracy of gradient estimates in a computationally efficient manner. The idea comes from the following observation, which is made mathematically precise later in the text: There is little need for an accurate gradient estimate in a stochastic solver when the iterates are far from the optimal solution. However, stochastic algorithms require increasingly accurate gradient estimates as the iterates get closer to the solution. To maintain accuracy, adaptive sampling methods dynamically increase the batch/sample size in response to an \textit{a posteriori} estimate of the variance of the sampled gradients. Theoretical results from the adaptive sampling literature are promising. Indeed, in \cite{byrd2012sample}, the authors show that this methodology matches the best achievable complexity bound for \emph{unconstrained} stochastic programs. Adaptive sampling is also known to be efficient in practice\cite{bollapragada2018progressive}. Recently, adaptive sampling methods have been used to develop efficient proximal/projected gradient algorithms for constrained optimization problems \cite{beiser2020adaptive,xie2020constrained}.
Nevertheless, projecting gradients at every iteration can be challenging or inefficient, depending on the structure of the constraint set.
Therefore, we go beyond the work in \cite{beiser2020adaptive,xie2020constrained} and consider augmented Lagrangian techniques.
In turn, we address a more general class of algorithms and provide greater flexibility for treating the constraint set.
\subsection{Contributions}
In this paper, we propose an adaptive sampling augmented Lagrangian (ASAL) method by combining the augmented Lagrangian framework with adaptive sampling techniques to solve constrained stochastic optimization problems. We use 
adaptive sampling to control the accuracy of the gradient estimates when solving the subproblems obtained by  penalizing the linear equality constraints. Moreover, we employ inexact solution mechanisms by imposing stochastic inexactness conditions to terminate the inner (i.e., subproblem) iterations. In this way, we maximize the overall computational efficiency of our approach without sacrificing accuracy. Another important aspect of the methodology is that it relies on proximal/projected gradients to achieve feasibility with respect to the constraint $x \in \X$.
Since the method relies only on gradient information, we establish sublinear convergence in the outer iterations for convex objective functions.
Furthermore, given a user-defined algorithm parameter $\delta > 0$ and an arbitrary penalty parameter $\alpha > 0$, we find the total expected number of gradient evaluations to achieve an $\epsilon$-accurate solution to be $\mathcal{O}(\epsilon^{-3-\delta})$.
Moreover, if the penalty parameter is chosen to be sufficiently large, i.e., $\mathcal{O}(\epsilon^{-1})$, then our result improves to $\mathcal{O}(\epsilon^{-2})$.
Finally, the worst-case complexity becomes $\mathcal{O}(\epsilon^{-1} \log(1/\epsilon))$ for strongly convex objective functions and when $\X = \R^n$. \Cref{table:comparisonofpapers} compares our setting and theoretical results with the relevant literature. To evaluate the efficacy of our framework, we compare its performance to baseline algorithms in a collection of model problems from machine learning (\Cref{sub:logistic_regression_with_disparate_impact_constraints}) and engineering (\Cref{sub:optimal_truss_design,sub:topology_optimization}).

\begin{table}[htp]\label{table:comparisonofpapers}
\centering
\def\arraystretch{1.2}
\caption{
Summary of the theoretical convergence rate and sample complexity results in the relevant literature under different problem settings. In all the works mentioned here, the constraints are deterministic.  Here, $K$ denotes the (outer) iteration number, and $\epsilon$ denotes the required accuracy. Convergence rates are deterministic for deterministic problems and are in expectation for stochastic problems. Finally, sample complexity for stochastic solvers denotes the total number of expected stochastic gradient evaluations required to get $\epsilon-$accurate solutions. }
\resizebox{\columnwidth}{!}{
\begin{tabular}{cccccc}
\hline
paper & objective & set ($\X$) & constraints & rate (outer iter) & sample complexity \\
\hline
\multirow{2}{*}{\cite{lan2016iteration}} & convex & \multirow{2}{*}{convex compact} & \multirow{2}{*}{linear} & \multirow{2}{*}{$\mathcal{O}(1/K)$} & \multirow{2}{*}{-} \\
& deterministic & & & & \\ \hline
\multirow{2}{*}{\cite{xu2021iteration}} & convex & \multirow{2}{*}{convex closed} & \multirow{2}{*}{convex} & \multirow{2}{*}{$\mathcal{O}(1/K)$} & \multirow{2}{*}{-} \\
& deterministic & & & & \\ \hline
\multirow{2}{*}{\cite{xie2019si}} & strongly convex & \multirow{2}{*}{$\R^n$} & \multirow{2}{*}{linear} & \multirow{2}{*}{linear} & \multirow{2}{*}{$\mathcal{O}(\epsilon^{-1})$} \\
& stochastic & & & & \\ \hline
\multirow{2}{*}{\cite{xu2020primal}} & convex & \multirow{2}{*}{convex} & \multirow{2}{*}{convex} & \multirow{2}{*}{$\mathcal{O}(1/\sqrt{K})$} & \multirow{2}{*}{$\mathcal{O}(\epsilon^{-2})$} \\
& stochastic nonsmooth & & & & \\ \hline
\multirow{2}{*}{\cite{xu2020primal}} & strongly convex & \multirow{2}{*}{convex} & \multirow{2}{*}{convex} & \multirow{2}{*}{$\mathcal{O}(\log(K)/K)$} & \multirow{2}{*}{$\mathcal{O}(\epsilon^{-1}\log(1/\epsilon))$} \\
& stochastic nonsmooth& & & & \\ \hline
\Cref{thm:TotalWorkComplexity} & convex & \multirow{2}{*}{convex compact} & \multirow{2}{*}{linear} & \multirow{2}{*}{$\mathcal{O}(1/K)$} & \multirow{2}{*}{$\mathcal{O}(\epsilon^{-3-\delta})$} \\
(arbitrary penalty parameter) & stochastic & & & & \\ \hline
\Cref{corr:optimalworkcomp} & convex & \multirow{2}{*}{convex compact} & \multirow{2}{*}{linear} & \multirow{2}{*}{$\mathcal{O}(1/K)$} & \multirow{2}{*}{$\mathcal{O}(\epsilon^{-2})$} \\
 ($\mathcal{O}(\epsilon^{-1})$ penalty parameter) & stochastic & & & & \\ \hline
\multirow{2}{*}{\Cref{thm:TotalWorkComplexity_specialcase}} & strongly convex & \multirow{2}{*}{$\R^n$} & \multirow{2}{*}{linear} & \multirow{2}{*}{linear} & \multirow{2}{*}{$\mathcal{O}(\epsilon^{-1}\log(1/\epsilon))$} \\
& stochastic & & & & \\ \hline
\end{tabular}
}
\label{tbl:values}
\end{table}

\subsection{Literature Review}

The augmented Lagrangian method, also known as the method of multipliers, was first proposed by Hestenes \cite{hestenes1969multiplier} and Powell \cite{powell1969method}. In \cite{bertsekas2014constrained}, its performance is analyzed and compared to other common approaches, such as penalty and Lagrangian methods; see also \cite{eckstein2012augmented,rockafellar1974augmented,rockafellar1976augmented,bertsekas2015convex}.
Although there have been extensive research efforts to enhance the performance of the basic augmented Lagrangian method to solve deterministic optimization problems (see, e.g., \cite{lan2016iteration,xu2021iteration,li2021rate,li2021augmented,curtis2015adaptive,curtis2016adaptive,birgin2012augmented}), the current literature on stochastic optimization problems is limited \cite{jiang2022stochastic,li2022stochastic,xu2020primal}.
In \cite{jiang2022stochastic}, the authors apply a stochastic augmented Lagrangian method to the domain adaptation problem. \rbreview{In \cite{xu2020primal}, Xu developed stochastic primal-dual methods using the augmented Lagrangian function for solving nonsmooth optimization problems with a large number of constraints. In the aforestated approach, a projected stochastic gradient method is employed for the primal updates, while a randomized coordinate method is used for the dual updates.}

For structured optimization problems with linear constraints, the alternating direction method of multipliers (ADMM) framework is often preferred \cite{boyd2011distributed}. There has been significant work on stochastic versions of the ADMM method \cite{ouyang2013stochastic,zhong2014fast,suzuki2014stochastic,zheng2016fast,xie2019si}. In \cite{ouyang2013stochastic}, the authors consider stochastic ADMM and show a $\mathcal{O}(\log(K)/K)$ convergence rate for strongly convex and $\mathcal{O} (1/\sqrt{K})$ for general convex objective functions. In \cite{xie2019si}, the authors design an inexact solution mechanism for the subproblems in stochastic ADMM when $\X = \R^n$. There, the authors employ the stochastic gradient method to solve the subproblems and show a linear convergence rate for strongly convex functions.  Although our approach also involves inexact solutions, we consider adaptive sampling techniques to solve the subproblems and analyze both general convex and strongly convex functions. Moreover, our formulation allows us to consider implicit constraint sets (i.e., $\X \subset \R^n$) and utilizes only projected (or proximal) stochastic gradients. Other works achieve improved convergence rates by introducing stochastic variance reduction techniques (see, e.g., \cite{zhong2014fast,suzuki2014stochastic,zheng2016fast,9110776}).

 There are many articles on stochastic optimization methods with dynamic sample sizes \cite{friedlander2012hybrid,byrd2012sample,cartis2018global,pasupathy2018sampling,roosta2019sub,bollapragada2018adaptive,bollapragada2018progressive,bollapragada2019exact,bottou2018optimization,beiser2020adaptive,xie2020constrained,espath2021equivalence,kodakkal2022risk,ganesh2022gradient}. Most of these works focus on unconstrained problems. Of note is the work by Friedlander and Schmidt \cite{friedlander2012hybrid}, which shows linear convergence for finite-sum problems when the sample size increases at a geometric rate. Our work relates to the approach taken in Byrd et al.~\cite{byrd2012sample}, which shows linear convergence of the expected risk minimization problem and calculates the worst-case complexity bounds for the number of gradient evaluations required to get $\epsilon$-accurate solutions. Byrd et al.~\cite{byrd2012sample} also study the theoretical and practical aspects of the so-called \textit{norm test}, which controls the sample sizes. Finally, in \cite{beiser2020adaptive,xie2020constrained}, the authors consider adaptive sampling mechanisms for constrained stochastic programs. In both works, the constraints are represented by an abstract convex set, and the authors propose generalizations of the norm test that utilize projected (reduced) gradients.

 Another common methodology to approach \cref{eq:GeneralFormulation} is using sample average approximation (SAA) techniques \cite{shapiro2021lectures,phelps2016optimal,kouri2018optimization,kouri2022primal,royset2013optimal} which replace the expected value in the objective function with a fixed sample average or other empirical approximation.
 When it comes to alternative techniques to solve constrained stochastic programs, the sequential quadratic programming (SQP) framework \cite{curtis2021inexact,curtis2021worst,berahas2021sequential,berahas2022adaptive,na2022adaptive,na2023inequality} is also often utilized.

\subsection{Notation}
We denote the set of natural numbers by $\N \defeq \{0,1,2,\dots\}$, and the set of positive natural numbers as $\N_+ \defeq \{1,2,\dots\}$. Throughout this work, $\|\cdot\|$ denotes the $\ell_2$ vector norm or matrix norm and $\langle\cdot,\cdot\rangle$ denotes the $\ell_2$-inner product. Finally, a matrix $A \in \R^{m\times n}$ is indicated to be positive definite by writing $A \succ 0$ and positive semi-definite by writing $A \succeq 0$.
$A^T \in \R^{n\times m}$ denotes the transpose of a matrix $A$.

\subsection{Organization}

This paper is organized as follows. In \Cref{sec:preliminaries}, we introduce the preliminary material and assumptions used throughout the paper. The algorithmic framework and its components are given in \Cref{sec:algorithms}. In \Cref{sec:theory}, we analyze the convergence and complexity properties of our approach. Practical implementation of the algorithmic components is discussed in \Cref{sec:practical}. We demonstrate the numerical performance of our methodology in \Cref{sec:numerical_results}. Finally, in \Cref{sec:final_remarks}, we provide concluding remarks and discuss avenues for future research. 

\section{Preliminaries and Assumptions} \label{sec:preliminaries}

We provide preliminaries regarding the deterministic augmented Lagrangian method and its interpretation as a gradient descent method applied to the Moreau envelope of the dual function.
We also state preliminary assumptions and recall results from the literature that are relied on later in the paper. 

\subsection{Deterministic Augmented Lagrangian Method}
\label{sub:DeterministicAL}
The Lagrangian\linebreak function for the problem \cref{eq:GeneralFormulation} is 
\begin{equation}\label{eq:LagrangianFunction}
	\ell(x,\lambda) = f(x) - \langle \lambda , c(x) \rangle,
\end{equation}
where $\lambda \in \R^m$ is the Lagrangian (dual) parameter associated to the constraint function $c(x)$. 
Using \cref{eq:LagrangianFunction}, we can define the {saddle-point problem}, 
\begin{equation}\label{eq:Saddlepointproblem}
	\min_{x \in \X} \sup_{\lambda \in \R^m} \ell(x,\lambda),
\end{equation}
and note that
\begin{align*}
\sup_{\lambda \in \R^m} \ell(x,\lambda) = 
\begin{cases}
  f(x) & \quad \text{for} \quad c(x) = 0, \\
  \infty & \quad \text{for} \quad c(x) \neq 0.
\end{cases}
\end{align*}
Hence, if there exists $x \in \X \cap \{x \in \R^n \mid c(x) = 0\}$, then \cref{eq:Saddlepointproblem} is equivalent to \cref{eq:GeneralFormulation} in the sense that
\begin{align*}
\min_{x\in \X} \sup_{\lambda \in \R^m} \ell(x,\lambda) = \min_{ \{ x \in \X \mid c(x) = 0 \} } f(x)
\end{align*}
and
\begin{align*}
\argmin_{x\in \X} \sup_{\lambda \in \R^m} \ell(x,\lambda) = \argmin_{ \{ x \in \X \mid c(x) = 0 \} } f(x).
\end{align*}
A primal-dual iterate pair $(\hat x,\hat \lambda)$ is said to be a  stationary point of \cref{eq:Saddlepointproblem} if 
\begin{equation} \label{eq:firstStatPoint}
        (\hat x,\hat \lambda) \in \left\{(x,\lambda) \Bigg| \frac{\proj_{\X} (x - \eta \nabla \ell_x(x,\lambda)) - x}{\eta} = 0 \text{ and } c(x) = 0\right\},
\end{equation}
where $\eta > 0$ and
\begin{equation}\label{eq:proj}
\proj_{\X} (y) = \argmin_{x\in \X} \|x - y\|^2, 
\end{equation}
is the projection of $y \in \R^n$ onto the set $\X$ (see \cite{curtis2016adaptive,lan2016iteration}). We also refer to the conditions in \cref{eq:firstStatPoint} as the
\begin{subequations}
\label{eq:PrimalErrors}
\begin{align}
\label{eq:PrimalFeasibilityError}
\mbox{\emph{feasibility error}:}& \quad \|c(x)\|
\,,
\end{align}
and the
\begin{align}
\label{eq:PrimalOptimalityError}
\mbox{\emph{stationarity error}:}& \quad \left\|\frac{\proj_{\X} (x - \eta \nabla \ell_x(x,\lambda)) - x}{\eta}\right\|
\,.
\end{align}
\end{subequations}

The augmented Lagrangian method is a class of iterative methods that produce stationary points satisfying \cref{eq:firstStatPoint} by solving a sequence of subproblems where the objective function is the sum of the Lagrangian function $\ell(x,\lambda)$ and a quadratic penalty term that penalizes violation of the equality constraint $c(x) = 0$. Specifically, at any given iteration $k\in\mathbb{N}$, the basic primal and dual update rules are given as follows:
\begin{subequations}
\label{eq:ALUpdates}
\begin{align}
    x_k^* &\in \argmin_{x\in \X}~\mathcal{L}(x,\lambda_k;\alpha_k),  \label{eq:primalUpdate} \\
    \lambda_{k+1} &= \lambda_k - \alpha_k c(x_k^*), \label{eq:dualUpdate}
\end{align}
\end{subequations}
where $\alpha_k > 0$ is the penalty parameter and
\begin{equation} \label{eq:AugmentedLagrangian}
	\mathcal{L}(x,\lambda;\alpha) =
	f(x) - \langle \lambda , c(x) \rangle + \frac{\alpha}{2}\|c(x)\|^2,
\end{equation}
is the augmented Lagrangian function. Without restrictions on the objective function $f(x)$, the subproblem in \cref{eq:primalUpdate} may be unbounded.
In this paper, we invoke assumptions that ensure this is not the case (cf.~\Cref{ass:Existenceofsecondorderpoint} or \Cref{ass:StrcvxALwrtx}) as well as some other basic assumptions of additional utility. 

\subsection{Assumptions} We make the following assumptions about the objective function, the constraint function, and the existence of the solution. 

\begin{assumption}
\label{ass:Contdiff}
The objective function $f:\R^n \rightarrow \R$ is a convex continuously differentiable function on $\X$. That is,
$\nabla^2 f(x) \succeq 0$, for all $x \in \X$.
In addition, the gradient of the objective function $\nabla f:\R^n \rightarrow \R^n$ is Lipschitz continuous on $\X$ with Lipschitz constant $L < \infty$. That is,
\begin{align*}
\| \nabla f(x) - \nabla f(y)\| \leq L \| x - y \| \quad \forall x,y \in \X.
\end{align*}
\end{assumption}

\Cref{ass:Contdiff} implies that the augmented Lagrangian function is also a convex function with respect to $x$ on $\X$. That is, for any $\alpha > 0$,
\begin{align*}
\nabla^2_{xx} \L(x, \lambda; \alpha) = \nabla^2 f(x) + \alpha A^TA \succeq 0 \quad \forall x \in \X, \lambda \in \R^m.
\end{align*}

Note that the affine constraint function $c:\R^n \rightarrow \R^m$ is Lipschitz continuous on $\X$ with Lipschitz constant $\|A\|$. That is, for all $x,y \in \X$,
\begin{equation}\label{eq:LipschitzG}
\| c(x) - c(y)\| = \|A(x - y)\|\leq \|A\|\| x - y \| .\end{equation}
Moreover, as a consequence of \Cref{ass:Contdiff,eq:LipschitzG}, we can show that the gradient of the augmented Lagrangian function is Lipschitz continuous with respect to $x$ on $\X$ with Lipschitz constant $L + \alpha \|A\|^2$. That is, due to \cref{eq:AugmentedLagrangian},
\begin{equation*}
    \nabla_x\L(x,\lambda; \alpha) - \nabla_y \L(y,\lambda; \alpha)
    =
    \nabla f(x) - \nabla f(y)
    +
    \alpha
    \langle A, Ax - Ay \rangle
        \,,
\end{equation*}
and so
\begin{equation}
\label{eq:LagLipschitz}
    \|\nabla_x\L(x,\lambda; \alpha) - \nabla_y \L(y,\lambda; \alpha)\|
    \leq
    \left(L + \alpha \|A\|^2\right)\| x - y \|
    \,,
\end{equation}
for all $x, y \in \X$.
\begin{assumption}
\label{ass:Existenceofsecondorderpoint}
The set $\X \subset \R^n$ is nonempty, convex, and compact. Also, there exists an optimal primal-dual pair $(x^*, \lambda^*)$ that satisfies the optimality conditions \cref{eq:firstStatPoint}.
\end{assumption}
The compactness of set $\X$ implies that there exists a positive $D < \infty$ such that
\begin{equation}\label{eq:compactD}
    \|x - y\| \leq D \quad \forall x, y \in \X.
\end{equation}
Also, the existence of an optimal solution $x^*$ implies that the problem in \cref{eq:primalUpdate} is bounded below. That is, for any $x \in \X$, $\lambda \in \R^m$, and $\alpha \geq 0$,   
\begin{align*}
    \mathcal{L}(x,\lambda;\alpha)
    \geq f(x) - \langle \lambda , c(x) \rangle
    &= f(x) - \langle \lambda , c(x) - c(x^*)\rangle \\
    &\geq f(x) - \|\lambda\|\|c(x) - c(x^*)\| \\
        &\geq f(x^*) + \langle \nabla f(x^*), x - x^*\rangle - \|\lambda\|\|c(x) - c(x^*)\| \\
    &\geq f(x^*) - \|\nabla f(x^*)\|D - \|\lambda\|\|A\|D
            \,,
\end{align*}
where the first inequality is due to $\|c(x)\|^2 \geq 0$, the equality is due to $c(x^*) = 0$,  the third inequality is due to convexity of function $f$ (\Cref{ass:Contdiff}) and \cref{eq:compactD}, and the last inequality is due to  \cref{eq:LipschitzG} and \cref{eq:compactD}. Therefore, \cref{eq:primalUpdate} is well-defined.  

We also develop results for the special case where the augmented Lagrangian function is strongly convex with respect to $x\in\X$. 
\begin{assumption}
\label{ass:StrcvxALwrtx}
The augmented Lagrangian is $\mu$-strongly convex with respect to $x\in\X$. That is,
\begin{align*}
\nabla^2_{xx} \L(x, \lambda)  \succeq \mu I \quad \forall x \in \X, \lambda \in \R^m,
\end{align*}
where $I \in \R^{n \times n}$ is an identity matrix.  
\end{assumption}
Note that if the objective function $f(x)$ is $\mu$-strongly convex or $A$ has full column rank, then \Cref{ass:StrcvxALwrtx} is trivially satisfied. Moreover, if \Cref{ass:StrcvxALwrtx} holds, then \cref{eq:primalUpdate} is well-defined for any $\lambda_k \in \R^m$.  

We also make a standard assumption about the stochastic gradient of $f(x) = \E_{\zeta}[F(x,\zeta)]$. \begin{assumption}
\label{ass:variance}
The variance in the stochastic gradient of $f(x)$ is bounded. That is, there exist constants $\omega_1,\omega_2 \geq 0$ such that \begin{align*}
\E_{\zeta}[\|\nabla F(x,\zeta) - \nabla f(x)\|^2] \leq \omega_1 \|\nabla f(x)\|^2 + \omega_2, \quad \forall x\in \X.
\end{align*}
\end{assumption}
Using \Cref{ass:Contdiff,ass:Existenceofsecondorderpoint,ass:variance}, it follows that 
\begin{align*}
\E_{\zeta}[\|\nabla F(x,\zeta) - \nabla f(x)\|^2] &\leq 2\omega_1 \|\nabla f(x) - \nabla f(x^*)\|^2 + 2\omega_1\|\nabla f(x^*)\|^2 + \omega_2 \nonumber \\
&\leq 2\omega_1 L^2D^2 + 2\omega_1\|\nabla f(x^*)\|^2 + \omega_2
\defeq \omega
\,,\nonumber 
\end{align*}
where the first inequality is due to the fact that $\|a + b\|^2 \leq 2\|a\|^2 + 2\|b\|^2$ for any $a,b \in \R^n$. In turn, we note that combining the assumptions above implies the existence of $\omega \geq 0$ such that
\begin{align}\label{eq:bndvar}
    \E_{\zeta}[\|\nabla F(x,\zeta) - \nabla f(x)\|^2] \leq \omega, \quad \forall x\in \X.
\end{align}

\subsection{Gradient Descent and the Moreau envelope} \label{subsec:equivalence}
The convergence properties of the augmented Lagrangian method are often analyzed by showing its equivalence to a method (e.g., proximal point method) applied to dual problem (cf.~\cite{wright2022optimization}). We follow a similar approach in our analysis and show the equivalence of the augmented Lagrangian method and gradient descent method applied to the Moreau envelope \cite{moreau1965proximite} of the (negative) dual function. 
The negative of the dual function of \cref{eq:GeneralFormulation} is denoted
\begin{equation}\label{eq:dualFunction}
    q(\lambda) = -\min_{x\in \X} \ell(x,\lambda)
    \,,
\end{equation}
and is known to be a convex, proper and continuous function from $\R^m$ to $\R$ \cite{boyd2004convex}.
For any given $\alpha > 0$, the Moreau envelope of $q(\lambda)$ is defined as follows \cite{moreau1965proximite}:
\begin{align}\label{eq:Moreauenvelope}
    q_{\alpha}(u) & \defeq \min_{\lambda} \left [ q(\lambda) + \frac{1}{2 \alpha} \| \lambda - u \|^2 \right ]
    \,.
\end{align}
In the following lemma, we summarize the important properties of Moreau envelopes.
\begin{lemma}\label{lemma:MoreauProperties}
The function $q_{\alpha}(u)$ given in \cref{eq:Moreauenvelope} is called the Moreau envelope of $q(\lambda)$ and satisfies the following properties.
\begin{enumerate}[label=(\roman*)]
\item \cite[Equation~3.2]{parikhproximal} The gradient of the Moreau envelope is
\begin{align}\label{eq:moreaugrad}
    \nabla q_\alpha(u) & = \frac{1}{\alpha}(u - \prox_{\alpha q}(u)),
\end{align}
where
\begin{align*}
    \prox_{\alpha q}(u) & = \argmin_{\lambda} \left [ q(\lambda) + \frac{1}{2 \alpha} \| \lambda - u \|^2 \right ].
\end{align*}
\item \cite[Corollary 18.19]{bauschke2011convex} The gradients $\nabla q_\alpha(u)$ are Lipschitz continuous with Lipschitz constant $L_{\alpha} = \alpha^{-1}$. That is, 
\begin{align} \label{eq:lipschitz_moreau}
\|\nabla q_\alpha(u)  - \nabla q_\alpha(v) \| & \leq
\alpha^{-1} \|u - v\|, \quad \forall u,v \in \R^m. 
\end{align}
\item \cite[Page 136]{parikhproximal} The Moreau envelope retains the optimal value and the set of minimizers. That is,
\begin{align} \label{eq:dualMoreaequivalence}
\min_{\lambda} q(\lambda) = \min_{\lambda} q_{\alpha}(\lambda) \quad \mbox{and} \quad   \argmin_{\lambda} q(\lambda) = \argmin_{u} q_{\alpha}(u),
\end{align}
where the unique common minimizer $\lambda^\ast \in \R^m$ satisfies the fixed point equation $\lambda^\ast = {\prox}_{\alpha q}(\lambda^\ast)$.
\item \cite[Lemma 2.23]{planiden2016strongly} $q(\lambda)$ is strongly convex with parameter $\mu_q > 0$ if and only if $ q_\alpha(u)$ is strongly convex with parameter $\mu_\alpha = \frac{\mu_q}{\mu_q \alpha + 1} > 0$.
\end{enumerate}
\end{lemma}

Due to \Cref{ass:Existenceofsecondorderpoint} and weak duality \cite{boyd2004convex}, we have that $q(\lambda)$ is bounded below. That is, the optimal value $q^*$ is finite.
Indeed,
\begin{equation}\label{eq:finitedual}
q^* =\min_{\lambda} q_{\alpha}(\lambda)=  \min_{\lambda} q(\lambda) = - \max_{\lambda} \left [ - q(\lambda) \right ] \geq - \min_{x \in \X, c(x) = 0} f(x) = - f(x^*)
\,.
\end{equation}
Owing to this fact and the properties of $q_\alpha(\lambda)$ in~\Cref{lemma:MoreauProperties}, the dual variable $\lambda \to \lambda^\ast$ will converge by iteratively minimizing $q_\alpha(\lambda)$ as in the gradient descent method.
More explicitly, we may form a convergent sequence of dual variables as follows: 
\begin{align}
    \lambda_{k+1} &= \lambda_k - \alpha \nabla q_{\alpha}(\lambda_k) \label{eq:UpdateruleGDonMoreau}\\
    &= \argmin_{\lambda} \left [ q(\lambda) + \frac{1}{2 \alpha} \| \lambda - \lambda_k \|^2 \right ] \nonumber \\     &= \argmin_{\lambda} \left [ -\min_{x \in \X} [\ell(x,\lambda)] + \frac{1}{2 \alpha} \| \lambda - \lambda_k \|^2 \right ] \nonumber \\
    &=\argmax_{\lambda} \left [ \min_{x \in \X} [\ell(x,\lambda) - \frac{1}{2 \alpha} \| \lambda - \lambda_k \|^2 ] \right ] \nonumber 
    \,,
\end{align}
where the second equality is due to \cref{eq:moreaugrad} and third equality is due to \cref{eq:dualFunction}. The function $\ell(x,\lambda)- \frac{1}{2\alpha}\|\lambda - \lambda_k\|^2$ is convex with respect to $x$ on $\X$ and strongly concave with respect to $\lambda$. By Sion's Minimax Theorem \cite{sion1958general}, we can interchange the min and max operations (cf.~\cite[Section~10.5.2]{wright2022optimization}) and obtain an equivalent characterization. That is, 
\begin{align}
    \max_{\lambda} \left [ \min_{x \in \X} [ \ell(x,\lambda) - \frac{1}{2\alpha}\|\lambda - \lambda_k\|^2 ] \right ]&= 
    \min_{x \in \X} \left [ \max_{\lambda} [ \ell(x,\lambda) - \frac{1}{2\alpha}\|\lambda - \lambda_k\|^2 ] \right ] \\
    	& = \min_{x \in \X} \left [ \max_{\lambda} [ f(x) - \langle \lambda , c(x) \rangle - \frac{1}{2\alpha}\|\lambda - \lambda_k\|^2 ] \right ] \label{eq:maxmintominmax}         \,.
\end{align}
Note that the optimal solution to the max problem (strongly concave in $\lambda$) in the second equality is $\lambda = \lambda_k - \alpha c(x)$.  Substituting this expression into \cref{eq:maxmintominmax}, we find
\begin{align*}
    \max_{\lambda} \left [ \min_{x \in \X} [ \ell(x,\lambda) - \frac{1}{2\alpha}\|\lambda - \lambda_k\|^2 ] \right ] = \min_{x \in \X} \mathcal{L}(x,\lambda_k;\alpha).
\end{align*}
Hence, the dual update $\lambda_{k+1}$ is given as follows:
\begin{subequations}
\begin{align}
    x_k^* &\in \argmin_{x\in \X}~\mathcal{L}(x,\lambda_k;\alpha)  \label{eq:moreauprimalUpdate} \\
    \lambda_{k+1} &= \lambda_k - \alpha c(x_k^*). \label{eq:moreaudualUpdate}
\end{align}
\end{subequations}

We now observe that the primal updates in \cref{eq:primalUpdate} and \cref{eq:moreauprimalUpdate} are both minimizers of the augmented Lagrangian function within the set $\X$. This optimization problem can have multiple optimal solutions when the augmented Lagrangian function $\mathcal{L}(x,\lambda_k;\alpha)$ is only a general convex function (not strongly convex).
Hence, the updates \cref{eq:primalUpdate} and \cref{eq:moreauprimalUpdate} may not be the same. However, the dual updates are equivalent due to the following inequality \cite[Equation~2.16]{lan2016iteration}: For any $x \in \X$ and $x^*_k \in \arg\min_{x\in \X}\mathcal{L}(x,\lambda_k;\alpha)$, 
\begin{equation} \label{eq:FeasToFunc}
\|c(x^*_k) - c(x)\|^2 \leq \frac{2}{\alpha}\left(\L(x, \lambda_k;\alpha) - \L(x^*_k,\lambda_k;\alpha)\right).
\end{equation}
Therefore, all solutions of $\min_{x\in \X}~\mathcal{L}(x,\lambda_k;\alpha)$ have the same constraint function value $c(x)$ and the augmented Lagrangian method is equivalent to the gradient descent method applied to the Moreau envelope \cref{eq:Moreauenvelope}. 
Finally, we conclude this section on preliminary material by noting that  
\begin{align}\label{eq:gradform}
    \nabla q_{\alpha} (\lambda_k) = c(x_k^*)
    \,,
\end{align}
due to \cref{eq:UpdateruleGDonMoreau,eq:moreaudualUpdate}. 

\section{Algorithmic Framework} \label{sec:algorithms}
This section begins with a description of a generic inexact augmented Lagrangian framework for solving \cref{eq:GeneralFormulation}.
We then provide a complete description of our algorithm, which employs the adaptive sampling proximal gradient method \cite{beiser2020adaptive,xie2020constrained} to minimize the augmented Lagrangian function \cref{eq:AugmentedLagrangian} defined at each iteration.

Each primal variable update \cref{eq:primalUpdate} in the augmented Lagrangian method involves solving a computationally expensive optimization problem, namely,
\begin{equation}
\label{eq:augLagProb}
	\min_{x \in \X}~ \mathcal{L}(x,\lambda_k;\alpha),
\end{equation}
where $\alpha > 0$ is the penalty parameter and $\lambda_k$ is the dual variable at iteration $k \in \N$. Owing to the stochastic nature of sampling the objective function $f(x) = \E_{\zeta}[F(x,\zeta)]$, the exact solutions to these subproblems cannot be obtained efficiently. Therefore, we work with the inexact augmented Lagrangian framework outlined in \cref{alg:StochALwInexact}.
At each iteration of this meta-algorithm, the subproblem \cref{eq:augLagProb} is solved (inexactly) by a given subproblem solver $\mathcal{S}$ until certain as yet unspecified inexactness conditions hold (cf.~\Cref{subsec: tolr}).
Of course, the dual variable update incurs errors attributed to the inexact primal solves.
However, if appropriate inexactness conditions are used to terminate the subproblem solver, then \cref{alg:StochALwInexact} will still converge at the same rate as the exact algorithm~\cref{eq:ALUpdates}, albeit in expectation.

\begin{algorithm}[H]
\caption{Inexact Augmented Lagrangian Framework}
\label{alg:StochALwInexact}
\begin{algorithmic}[1]
\REQUIRE $x_{-1} \in \R^n$, $\lambda_0 \in \R^m$, $\alpha>0$, inexactness conditions, solver $\mathcal{S}$. 
\FOR{$k = 0,1,...$ }
    \STATE Set starting point $x_{k,0} \gets x_{k-1}$
    \STATE Find an approximate minimizer $x_k$ of \cref{eq:augLagProb} using solver $\mathcal{S}$, starting with $x_{k,0}$ such that some inexactness conditions are satisfied  
        \STATE Update $\lambda_{k+1} \gets \lambda_k - \alpha c(x_k)$
\ENDFOR
\end{algorithmic}
\end{algorithm}

\begin{remark}
We make the following remarks about \cref{alg:StochALwInexact}.
\begin{itemize}
        \item \textbf{Solver and inexactness conditions:} For the sake of generality, we leave the description of the solver and inexactness conditions arbitrary and specify them in \Cref{subsec: adasample} and \Cref{subsec: tolr} respectively.  We assume that the solver $\mathcal{S}$ can compute an approximate minimizer $x_k$ that satisfies the inexactness conditions. 
    The sequences of primal and dual iterates obtained in the algorithm are random due to the stochastic nature of the objective function $f(x)$.
    Therefore, this assumption is reasonable when the inexactness conditions are also stochastic. 
    
    \item \textbf{Penalty parameter ($\alpha > 0$):} The algorithm employs a constant penalty parameter $\alpha > 0$. In \Cref{sec:theory}, we show that the algorithm converges for any choice of this parameter and does not depend on problem characteristics or other algorithmic parameters.  

    \item \textbf{Starting points ($x_{k,0}$):} At each iteration, the algorithm uses the previous primal iterate as starting point in the solver $\mathcal{S}$ to solve \cref{eq:augLagProb}. This is meant to reduce the computational effort to solve \cref{eq:augLagProb}. Since the successive augmented Lagrangian functions differ only in the dual variable $\lambda_k$, the approximate minimizer of the previous subproblem is an intuitive estimate of the solution to the current problem. In \Cref{sec:theory}, we quantify the efficiency of this starting point rule in terms of total computational work. 
\end{itemize}
\end{remark}
We now describe the unspecified components of this algorithm: the solver $\mathcal{S}$ and the tolerance conditions. 
\subsection{Adaptive Sampling Proximal Gradient Method}\label{subsec: adasample}
Projected or proximal stochastic gradient methods are a popular class of methods for solving \cref{eq:augLagProb} when the projection or proximal operators are easy to compute \cite{parikhproximal}.
The iterate update of a projected stochastic gradient method is given as follows:
\begin{equation}
    x_{k,t+1} = x_{k,t} + \eta R_{S_{k,t}}(x_{k,t},\lambda_{k};\alpha,\eta)     \,,
\end{equation}
where $\eta > 0$ is the step size parameter, $k \in \N$ denotes the \emph{outer} augmented Lagrangian iteration counter, $t\in \N$ denotes the \emph{inner} projected stochastic gradient iteration counter, $S_{k,t}$ is a set consisting of i.i.d.~samples of $\zeta$,
\begin{align}
  R_{S_{k,t}}(x_{k,t},\lambda_{k};\alpha,\eta) &\defeq  \frac{\proj_{\X}(x_{k,t} - \eta \nabla_x \mathcal{L}_{S_{k,t}}(x_{k,t},\lambda_{k};\alpha)) - x_{k,t}}{\eta}, \label{eq:redgradsk}\\
    \nabla_x \mathcal{L}_{S_{k,t}}(x_{k,t},\lambda_{k};\alpha)  &\defeq \frac{1}{|S_{k,t}|} \sum_{\zeta_i \in S_{k,t}} \nabla_x \mathcal{L}(x_{k,t},\lambda_{k}, {\zeta_i};\alpha), \label{eq:gradLagsk}
\end{align}    
and 
\begin{equation*}
    \nabla_x \mathcal{L}(x_{k,t},\lambda_{k}, {\zeta_i};\alpha) = \nabla_x F(x_{k,t},\zeta_i) - \langle \lambda_{k}, \nabla c(x_{k,t}) \rangle + \alpha  \langle c(x_{k,t}),\nabla c(x_{k,t})  \rangle.
\end{equation*}
In what follows, it is helpful to note that $R_{S_{k,t}}(x_{k,t},\lambda_{k};\alpha,\eta)$ denotes a stochastic approximation of the true projected (reduced) gradient
\begin{equation}\label{eq:redgrad}
  R(x_{k,t},\lambda_{k};\alpha,\eta) \defeq  \frac{\proj_{\X}(x_{k,t} - \eta \nabla_x \mathcal{L}(x_{k,t},\lambda_{k};\alpha)) - x_{k,t}}{\eta}. 
\end{equation}

Two adaptive sampling strategies have recently been proposed for the projected\linebreak stochastic gradient method \cite{xie2020constrained,beiser2020adaptive}.
Both strategies employ a mechanism for improving the quality of the stochastic gradient approximation by updating the sample size $|S_{k,t}|$ on the fly at each (subproblem) iteration $t$.
In turn, they overcome a significant limitation of fixed sample size strategies without \rbreview{compromising efficiency, while also maintaining the fast convergence of their deterministic counterparts.}
Indeed, fixed sample size strategies can only guarantee convergence to a neighborhood of the solution or must compromise on the convergence rate.

Adaptive sampling strategies aim to ensure that \rbreview{the variance in the stochastic gradient is controlled by the squared norm of the projected gradient. }In \cite{xie2020constrained}, this is written as follows:
\begin{equation} \label{eq:norm_ideal_auglag}
    \E_{k,t} \left[\|\nabla_x \mathcal{L}_{S_{k,t}}(x_{k,t},\lambda_{k}) -  \nabla_x \mathcal{L}(x_{k,t},\lambda_{k})\|^2\right] \leq \theta_g^2 \|\E_{k,t}[R_{S_{k,t}}(x_{k,t},\lambda_{k};\alpha,\eta)]\|^2
    \,,
\end{equation}
where $\theta_g > 0$ is a given parameter,
and
\begin{equation}
\label{eq:ConditionalExpectation_kt}
    \E_{k,t}[\,\cdot\,] \defeq \E[\,\cdot\,|\lambda_{k}, x_{k,t}]\,,
\end{equation}
denotes the expectation conditioned on the past iterates until $\lambda_{k}, x_{k,t}$. Specifically, $\E_{k,t}$ is the conditional expectation conditioned on the filtration $$\mathbb{T}_{k,t} = \sigma(\lambda_0,x_{-1,0},S_{0,0},\ldots,S_{0, T_{0}},\ldots, S_{k-1, 0},\ldots, S_{k-1, T_{k-1}}, S_{k, T_{0}},\ldots, S_{k, t})\,,$$ where $T_{i}$ denotes the number of inner iterations performed at the outer iteration $i$.

Using the definition of the gradient of the augmented Lagrangian function in \cref{eq:norm_ideal_auglag} results in the following equivalent condition: 
\begin{equation}\label{eq:norm_ideal_var}
    \E_{k,t} \left[\|\nabla F_{S_{k,t}}(x_{k,t}) -  \nabla f (x_{k,t})\|^2\right] \leq \theta_g^2 \|\E_{k,t}[R_{S_{k,t}}(x_{k,t},\lambda_{k};\alpha,\eta)]\|^2
    \,,
\end{equation}
where
\begin{equation}
    \nabla F_{S_{k,t}}(x_{k,t}) \defeq \frac{1}{|S_{k,t}|} \sum_{\zeta_i \in S_{k,t}}\nabla F(x_{k,t},\zeta_i).
\end{equation}
Since the samples of $\zeta$ are i.i.d., Bienaym\'e's identity may be used to simplify the left-hand side of~\cref{eq:norm_ideal_var}. 
This results in the following equivalent condition:
\begin{condition}[Theoretical Sampling Condition]
\label{con:Normcondition1proj}
    For any given $\theta_g > 0$, the variance in the stochastic gradient of the objective function $f$ is controlled by the squared norm of the expected projected gradient $R_{S_{k,t}}$. That is, 	\begin{equation}
	\label{eq:Normcondition1proj}
		\frac{\E_{\zeta}[\|\nabla F(x_{k,t},\zeta) - \nabla f(x_{k,t})\|^2]}{|S_{k,t}|}
		\leq
		\theta_g^2 \|\E_{k,t}[R_{S_{k,t}}(x_{k,t},\lambda_{k};\alpha,\eta)] \|^2.
	\end{equation}
\end{condition}
This condition involves computing true variances and exact projected gradients that are unavailable in practice. Therefore, in \Cref{sec:practical}, we also propose a practical version of this condition to control the sample sizes.

We conclude this subsection with remark and the following well-known result (adapted to our setting) \cite[Corollary 2.3.2]{nesterov2018lectures} that is used in the coming analysis. 
\begin{proposition}
Suppose \Cref{ass:Contdiff,ass:Existenceofsecondorderpoint} hold. Then, for any $0< \eta < \frac{1}{L + \alpha \|A\|^2} $ and for all $x\in \X, \lambda \in \R^m, \alpha > 0$,
\begin{equation}\label{eq:AugLipIneq}
\|R(x,\lambda;\alpha,\eta)\|^2 \leq \frac{2}{\eta} \left(\L(x,\lambda,\alpha) - \L(x^*_\L,\lambda,\alpha)\right),  \end{equation}
where $x^*_\L \in \argmin_{x \in \X} \L(x,\lambda,\alpha)$. 
Moreover, if \Cref{ass:StrcvxALwrtx} also holds, then,
\begin{equation}\label{eq:AugStrnCnvIneq}
\frac{\mu}{2}\|x - x^*_\L\|^2 + \frac{\eta}{2}\|R(x,\lambda;\alpha,\eta)\|^2 \leq \langle R(x,\lambda;\alpha,\eta), x - x^*_\L \rangle.
\end{equation}
\end{proposition}

\begin{remark}[Alternative Sampling Condition]
    An alternative sampling condition is proposed in \cite{beiser2020adaptive} that would replace the right-hand side of~\cref{eq:norm_ideal_auglag} by a constant factor times the squared norm of the projected gradient~\cref{eq:redgrad}.
    Following the procedure above, we would then arrive at a somewhat simpler inequality taking the place of \cref{eq:Normcondition1proj}, namely,
    \begin{align}
    \label{eq:Normcondition1proj_alt}
                                \frac{\E_{\zeta}[\|\nabla F(x_{k,t},\zeta) - \nabla f(x_{k,t})\|^2]}{|S_{k,t}|}
                &\leq \tilde{\theta}_g^2 \|R(x_{k,t},\lambda_{k};\alpha,\eta)]\|^2
        \,,
    \end{align}
    for some $\tilde{\theta}_g^2 > 0$.
    It turns out     that the two conditions \cref{eq:Normcondition1proj,eq:Normcondition1proj_alt} are equivalent in the sense that their right-hand sides bound each other from above and below:
    \begin{equation}
    \label{eq:Equivalence}
        \frac{\|R(x_{k,t},\lambda_k;\alpha,\eta)\|}{1 + \theta_g}
        \leq
        \|\E_{k,t}[R_{S_{k,t}}(x_{k,t},\lambda_{k};\alpha,\eta)\|
        \leq
        \frac{\|R(x_{k,t},\lambda_k;\alpha,\eta)\|}{1 - \theta_g}
        \,.
    \end{equation}
    Indeed, note that
    \begin{align*}
    \|\E_{k,t}[R_{S_{k,t}}(x_{k,t},\lambda_{k};\alpha,\eta) &- R(x_{k,t},\lambda_{k};\alpha,\eta)] \|^2 \\
    &\leq \E_{k,t}[\|R_{S_{k,t}}(x_{k,t},\lambda_{k};\alpha,\eta) - R(x_{k,t},\lambda_{k};\alpha,\eta) \|^2] \\
    &= \eta^{\ck{-2}}\E_{k,t}\Big[\|\proj_{\X}(x_{k,t} - \eta \nabla_x \mathcal{L}_{S_{k,t}}(x_{k,t},\lambda_{k};\alpha)) \\
     &\qquad\qquad\qquad\qquad- \proj_{\X}(x_{k,t} - \eta \nabla_x \mathcal{L}(x_{k,t},\lambda_{k};\alpha))\|^2\Big] \\
    &\leq \E_{k,t}\left[\|\nabla F_{S_{k,t}}(x_{k,t}) - \nabla f(x_{k,t})\|^2\right] \\
    &\leq \theta_g^2 \|\E_{k,t}[R_{S_{k,t}}(x_{k,t},\lambda_{k};\alpha,\eta)]\|^2
    \,,
    \end{align*}
    where the first line follows from Jensen's inequality, the proceeding equality is due to \cref{eq:redgradsk,eq:redgrad}, the second inequality follows from the non-expansiveness property of projections \cite{nesterov2018lectures}, and the last inequality is due to \cref{eq:norm_ideal_var}. Rearranging terms and using the reverse triangle inequality, $\|a\| - \|b\| \leq \|a - b\|$, for all $a,b \in \R^n$, we arrive at~\cref{eq:Equivalence}.
    Moreover, both conditions \cref{eq:Normcondition1proj,eq:Normcondition1proj_alt} lead to identical practical algorithms; cf.~\Cref{sec:practical}.
    We choose to work with~\cref{eq:Normcondition1proj} instead of~\cref{eq:Normcondition1proj_alt} because it leads to a simpler presentation of the complexity theory in \Cref{sec:theory}.
\end{remark}

\subsection{Inexactness Conditions} \label{subsec: tolr}
The efficiency of an inexact augmented Lagrangian framework depends on its inexactness conditions. These conditions must balance the accuracy of the solution computed at each (outer) iteration and the overall computational efficiency.
Due to the stochastic nature of the iterates obtained by our choice of the subproblem solver (cf.~\Cref{subsec: adasample}), these inexactness conditions must also be stochastic.
We now propose inexactness conditions that meet these requirements based on the Moreau envelope perspective developed in \Cref{subsec:equivalence}.
Recall from \Cref{subsec:equivalence} that the exact augmented Lagrangian method can be interpreted as a gradient descent method applied to the Moreau envelope of the (negative) dual function. Therefore, the inexact augmented Lagrangian method leads to  inexact dual variable updates. That is, from the dual update (\texttt{line 3} in \cref{alg:StochALwInexact}), \cref{eq:dualUpdate}, and \cref{eq:gradform}, we have, 
\begin{align}    \lambda_{k+1} &= \lambda_k - \alpha c(x_k) \nonumber \\
    &= \lambda_k - \alpha c(x^*_k) + \alpha c(x^*_k) - \alpha c(x_k)  \nonumber \\
    &= \lambda_k - \alpha \nabla q_{\alpha}(\lambda_k) +\alpha \epsilon_k
    \label{eq:UpdateruleGDonMoreauInexact},
        \end{align}
where $x^*_k$ is an exact minimizer of \cref{eq:augLagProb} and $\epsilon_k \defeq c(x^*_k) - c(x_k)$. \rev{We note that we have not imposed any structure on the subproblems~\cref{eq:augLagProb} that would give the update error $\epsilon_k$ zero mean}; i.e., $\E_k [\epsilon_k] \neq 0$, where
\begin{equation}
\label{eq:ConditionalExpectation_k}
    \E_k[\,\cdot\,] = \E[\,\cdot\,|\lambda_k]
\end{equation}
is the expected value operator conditioned on the iterates up until $\lambda_k$.
Specifically, $\E_{k}$ is the conditional expectation conditioned on the filtration $$\mathbb{T}_{k} = \sigma(\lambda_0,x_{-1,0},S_{0,0},\ldots,S_{0, T_{0}},\ldots, S_{k-1, 0},\ldots, S_{k-1, T_{k-1}})\,,$$ where $T_{i}$ denotes the number of inner iterations performed at the outer iteration $i$.
In turn, we choose to view the additive update rule~\cref{eq:UpdateruleGDonMoreauInexact} as a \emph{biased} stochastic gradient estimator update.

It is natural to consider an additional sampling condition when the sampling error can control the bias; cf.~\cite[Condition 2]{beiser2020adaptive}. \rbreview{Such additional conditions are also common in trust-region methods \cite{bastin2006adaptive,blanchet2019convergence,chen2018stochastic}.}
In the present setting, however, the error is due to the subproblem solver. To address this, we aim to design a tolerance condition for terminating the inner loop.
The following condition allows us to control the inexactness of the Moreau envelope gradient estimates in the dual update:
\begin{condition}[Tolerance Condition I]\label{cond:tolr-I}
	 For any given $\theta_e \in [0,1)$ and $\tau_k \geq 0$ with $\lim_{k \rightarrow \infty} \tau_k = 0$,  
  	\begin{align}\label{eq:TolrNormcondition}
    \E_k\left[\|c(x^*_k) - c(x_k)\|^2\right] \leq \theta_e^2 \|c(x^*_k)\|^2 + \tau_k,
    \end{align}
    where $x^*_k$ is a minimizer of \cref{eq:augLagProb}. \end{condition}
From the Moreau envelope perspective \cref{eq:UpdateruleGDonMoreauInexact}, this condition ensures that the expected squared norm of the error, $\E_k[\|\epsilon_k\|^2]$, is controlled by the squared norm of the gradient of $q_\alpha(\lambda_k)$ and a vanishing positive constant $\tau_k$. That is, from \cref{eq:gradform,eq:TolrNormcondition,eq:UpdateruleGDonMoreauInexact}, it follows that 
\begin{align*}
    \E_k[\|\epsilon_k\|^2] &= \E_k\left[\|c(x^*_k) - c(x_k)\|^2\right] 
    \leq \theta_e^2 \|c(x^*_k)\|^2  + \tau_k     = \theta_e^2 \|\nabla q_{\alpha}(\lambda_k)\|^2 + \tau_k.
\end{align*}
Likewise, this condition ensures that inexact gradient information can be employed far away from the solution (i.e., when the gradient's norm is large).
Meanwhile, it also ensures accurate gradient information near the solution (i.e., when the gradient norm is small). 
Although this condition is derived from controlling the inexactness in the dual update, it directly relates to inexactness in the minimization of \cref{eq:augLagProb} (cf.~\cref{eq:FeasToFunc}).  Therefore, we can replace \Cref{cond:tolr-I} by the following alternative condition:
\begin{condition}[Tolerance Condition II]\label{cond:tolr-II}
	For any given $\theta_e \in [0,1)$ and $\tau_k \geq 0$ with $\lim_{k \rightarrow \infty} \tau_k = 0$,  
	\begin{align}\label{eq:TolrNormconditionII}
    \E_k\left[\L(x_k, \lambda_k;\alpha) - \L(x^*_k,\lambda_k;\alpha)\right] \leq \frac{\alpha \theta_e^2 \|c(x^*_k)\|^2}{2} + \frac{\alpha\tau_k}{2},
    \end{align}
    where $x^*_k$ is a minimizer of \cref{eq:augLagProb}. \end{condition}
\Cref{cond:tolr-II} controls the error in the minimization of \cref{eq:augLagProb} and directly implies \Cref{cond:tolr-I}. Indeed, using \cref{eq:FeasToFunc} and \cref{eq:TolrNormconditionII}, we find
\begin{equation}\label{eq:cond2Tocond1}
    \E_k\left[\|c(x^*_k) - c(x_k)\|^2\right] \leq \frac{2}{\alpha}\E_k\left[\L(x_k, \lambda_k;\alpha) - \L(x^*_k,\lambda_k;\alpha)\right]     \leq \theta_e^2 \|c(x^*_k)\|^2 + \tau_k.
\end{equation}
When the augmented Lagrangian functions are strongly convex, we can also control the norm of the projected gradient \cref{eq:redgrad}. That is, by \Cref{ass:StrcvxALwrtx}, \cref{eq:LipschitzG}, and \cref{eq:AugStrnCnvIneq}, we have 
\begin{equation}\label{eq:FeasToGrad}
    \|c(x^*_k) - c(x_k)\|^2 \leq \|A\|^2 \|x^*_k - x_k\|^2 
    \leq \frac{4 \|A\|^2}{\mu^2}\| R(x_{k},\lambda_{k};\alpha,\eta) \|^2. 
\end{equation}
Therefore, we can impose the following alternate condition when $f(x)$ is strongly convex.
\begin{condition}[Tolerance Condition III]\label{cond:tolr-III}
        For any given $\Tilde \theta_e \in [0,1)$ and $\Tilde \tau_{k} \geq 0$ with $\lim_{k \rightarrow \infty} \Tilde \tau_{k} = 0$,  
	\begin{align}\label{eq:TolrNormconditionIII}
        \E_k\left[\| R(x_{k},\lambda_{k};\alpha,\eta) \|^2\right] \leq \Tilde \theta_e^2 \|c(x^*_k)\|^2 + \Tilde \tau_{k},
    \end{align}
    where $x^*_k$ is a minimizer of \cref{eq:augLagProb}. \end{condition}
\Cref{cond:tolr-III} also controls the error in the subproblem \cref{eq:augLagProb} and implies \Cref{cond:tolr-I}. Indeed, set $\Tilde 
\theta_e \leq \frac{\mu \theta_e}{2\|A\|}$ and $\Tilde \tau_{k} \leq \frac{\mu^2\tau_k}{4\|A\|^2}$.
Then, using \cref{eq:FeasToGrad} and \cref{eq:TolrNormconditionIII}, it holds that
\begin{equation}\label{eq:cond3Tocond1}
    \E_k\left[\|c(x^*_k) - c(x_k)\|^2\right] \leq \frac{4\|A\|^2}{\mu^2}\E_k\left[\|R(x_{k},\lambda_{k};\alpha,\eta)\|^2\right]     \leq \theta_e^2 \|c(x^*_k)\|^2 + \tau_k.
\end{equation}
\begin{remark}
\rbreview{We observe that conditions similar to \Cref{cond:tolr-I,cond:tolr-II,cond:tolr-III} have been proposed in the literature (cf.~\cite{rockafellar1976augmented, lan2016iteration, xu2021iteration, li2021augmented}). The primary advantage of employing our conditions lies in their adaptive control over the subproblem error. Although verifying \Cref{cond:tolr-I,cond:tolr-II,cond:tolr-III} for a stochastic subproblem solver can be challenging because they each require evaluating deterministic quantities, these conditions can still help us gain insight into the errors permitted in the algorithm while retaining desirable convergence properties. Furthermore, these conditions can guide the development of practical algorithms.
}
\end{remark}

\section{Theory} \label{sec:theory}
We now establish theoretical convergence guarantees and total sam-ple complexity results for the proposed inexact augmented Lagrangian algorithmic framework when the inexactness conditions proposed in \Cref{subsec: tolr} are satisfied. We use the following notation for the full expectation:
\begin{equation}
\label{eq:ExpectationIdentity}
    \E[\,\cdot\,] = \E_0[\,\E_1[\,\cdots \E_{k}[\,\cdot\,]]].
\end{equation}

\subsection{Convergence Results}\label{subsec:convergence}

We start by establishing a technical lemma.

\begin{lemma}\label{lemma:descentconditionofmoreauinexact} 
Suppose \Cref{ass:Contdiff,ass:Existenceofsecondorderpoint} hold. For any $x_{-1}, \lambda_0$ and $\alpha > 0$, let $\{x_k, \lambda_k\}$ be the sequence of primal-dual iterates generated by \cref{alg:StochALwInexact}. Then, for all $k \in \N$, 
\begin{align}\label{eq:descentRuleMoreauInexact}
    q_{\alpha}(\lambda_{k+1}) & \leq q_{\alpha}(\lambda_k) - \frac{\alpha}{2} \|\nabla q_{\alpha}(\lambda_k)\|^2 + \frac{\alpha}{2 } \|\epsilon_k\|^2,
\end{align}
where $\epsilon_k = c(x^*_k) - c(x_k)$ and $x^*_k$ is a minimizer of \cref{eq:augLagProb}.
\end{lemma}
\begin{proof}
From the dual update rule (\texttt{line 3} in \Cref{alg:StochALwInexact}), \cref{eq:dualUpdate}, and \cref{eq:gradform}, it follows that
\begin{align*}
    \lambda_{k+1} 
    &= \lambda_k - \alpha \nabla q_{\alpha}(\lambda_k) +\alpha \epsilon_k
    \,.
        \end{align*}
Using the Lipschitz continuity of $\nabla q_{\alpha}(\lambda)$ with Lipschitz constant $L_\alpha = \alpha^{-1}$ (cf.~\cref{eq:lipschitz_moreau}) and the descent lemma \cite{bertsekas2003convex}, we have,
\begin{align*}
    q_{\alpha}(\lambda_{k+1}) &\leq q_{\alpha}(\lambda_k) - \alpha \langle \nabla q_{\alpha}(\lambda_k) - \epsilon_k, \nabla q_{\alpha}(\lambda_k) \rangle +  \frac{\alpha^2 L_\alpha }{2}\|\nabla q_{\alpha}(\lambda_k) - \epsilon_k\|^2\\
    &= q_{\alpha}(\lambda_k) - \frac{\alpha}{2} \|\nabla q_{\alpha}(\lambda_k)\|^2 + \frac{\alpha}{2 } \|\epsilon_k\|^2
    \,,
        \end{align*}
as necessary.
\end{proof}
We are now ready to establish convergence results for the inexactness conditions developed in \Cref{subsec: tolr}. 
\begin{theorem}
\label{thm:Sublinearconv_norm}
Suppose \Cref{ass:Contdiff,ass:Existenceofsecondorderpoint,ass:variance} hold. For any $x_{-1}, \lambda_0$ and $\alpha > 0$, let $\{(x_k, \lambda_k)\}$ be the sequence of primal-dual iterates generated by \cref{alg:StochALwInexact}.
Furthermore, let $\theta_e \in [0,1)$ and $\tau_k \geq 0$ such that $\tau_0^{-1}\sum_{k=0}^{\infty} \tau_k = a_\infty < \infty$.
If any of the following three statements hold at each iteration $k \in N$:
\begin{enumerate}[(a)]
    \item the primal iterates $x_k$ satisfy \Cref{cond:tolr-I};
    \label{item:tolr-I}
    \item the primal iterates $x_k$ satisfy \Cref{cond:tolr-II}; or
    \label{item:tolr-II}
    \item \Cref{ass:StrcvxALwrtx} also holds and the primal iterates $x_k$ satisfy \Cref{cond:tolr-III} with $\Tilde \theta_{e} \leq \frac{\mu \theta_e}{2\|A\|}$ and $\Tilde \tau_{k} \leq \frac{\mu^2\tau_k}{4\|A\|^2}$;
    \label{item:tolr-III}
\end{enumerate}
then 
\begin{align*}
\lim_{k \rightarrow \infty} \E[\|c(x_k)\|^2] = 0.  \end{align*}
Moreover, for any $K \in \N$, we have that,
\begin{align}\label{eq:mincxksublconvresult}
    \min_{0\leq k \leq K-1} \E[\|c(x_k)\|^2] \leq \frac{4(1+\theta_e^2)}{\alpha (1 - \theta_e^2) K}  [q_{\alpha}(\lambda_0) - q^*] + \frac{4}{1-\theta_e^2} \frac{\tau_0 a_\infty}{K},
\end{align}
where $q^* > -\infty$ is defined in \cref{eq:finitedual}.  In addition, if either \ref{item:tolr-II} or \ref{item:tolr-III} is satisfied, then
\begin{equation}
\label{eq:Sublinearconv_norm_stationarity}
\lim_{k \rightarrow \infty} \E\left[\left\|\frac{\proj_{\X} (x_k - \eta \nabla \ell_x(x_k,\lambda_{k+1})) - x_k}{\eta}\right\|^2\right] = 0
\,,
\end{equation}
for every $0< \eta < \frac{1}{L + \alpha \|A\|^2}$.

\end{theorem}
\begin{proof}
If \ref{item:tolr-I}, \ref{item:tolr-II}, or \ref{item:tolr-III} holds, then \cref{eq:TolrNormcondition} holds as well due to \cref{eq:cond2Tocond1} and \cref{eq:cond3Tocond1}.
By \cref{lemma:descentconditionofmoreauinexact}, we have, 
\begin{align*}
q_{\alpha}(\lambda_{k+1}) & \leq q_{\alpha}(\lambda_k) - \frac{\alpha}{2} \|\nabla q_{\alpha}(\lambda_k)\|^2 + \frac{\alpha}{2 } \|\epsilon_k\|^2.
\end{align*}
Taking the conditional expectation~\cref{eq:ConditionalExpectation_k} of both sides and invoking \cref{eq:gradform,eq:TolrNormcondition}, we arrive at
\begin{align}
\E_k\left[q_{\alpha}(\lambda_{k+1})\right] &\leq q_{\alpha}(\lambda_k) - \frac{\alpha}{2} \|\nabla q_{\alpha}(\lambda_k)\|^2 + \frac{\alpha}{2 } \E_k\left[\|\epsilon_k\|^2\right] \nonumber \\
&\leq q_{\alpha}(\lambda_k) - \frac{\alpha(1 - \theta_e^2)}{2}\|c(x^*_k)\|^2 + \frac{\alpha}{2}\tau_k. \label{eq:grad_error}
\end{align}
Rearranging terms, we find
\begin{align}\label{eq:conv_exact_feas_bnd}
    \|c(x^*_k)\|^2 \leq \frac{2}{\alpha(1 - \theta_e^2)}  (q_{\alpha}(\lambda_k) - \E_k[q_{\alpha}(\lambda_{k+1})]) + \frac{1}{1-\theta_e^2}\tau_k. 
\end{align}
Therefore, 
\begin{align}
    \E_k\left[\|c(x_k)\|^2\right] &\leq 2\E_k\left[\|c(x_k) - c(x_k^*)\|^2\right] + 2\|c(x^*_k)\|^2 \nonumber \\
    &\leq 2(1+\theta_e^2)\|c(x^*_k)\|^2 + 2\tau_k \nonumber \\    &\leq \frac{4(1+\theta_e^2)}{\alpha(1 - \theta_e^2)}  (q_{\alpha}(\lambda_k) - \E_k[q_{\alpha}(\lambda_{k+1})]) + \frac{4}{1 - \theta_e^2}\tau_k \nonumber
    \,,
\end{align}
where the first inequality is due to $\|a + b\|^2 \leq 2\|a\|^2 + 2\|b\|^2$ for any $a,b \in \R^n$, the second inequality is due to \cref{eq:TolrNormcondition} and the last inequality follows from \cref{eq:conv_exact_feas_bnd}. 
Taking the full expectation~\cref{eq:ExpectationIdentity}, and summing the above inequality from $k=0$ to $K-1$, delivers
\begin{align*}
    \sum_{k=0}^{K-1}\E\left[\|c(x_k)\|^2\right] &\leq \frac{4(1+\theta_e^2)}{\alpha(1 - \theta_e^2)} \E [q_{\alpha}(\lambda_0) - q_{\alpha}(\lambda_{K})] + \frac{4}{1-\theta_e^2}\sum_{k=0}^{K-1}\tau_k \\
    &\leq \frac{4(1+\theta_e^2)}{\alpha(1 - \theta_e^2)} [q_{\alpha}(\lambda_0) - q^*] + \frac{4}{1-\theta_e^2}\tau_0a_{\infty}     \,,
\end{align*}
where the second inequality follows from~\cref{eq:dualMoreaequivalence} and the assumption $\sum_{k=0}^{\infty} \tau_k = \tau_0a_{\infty} < \infty$. Observe that $q^* > - \infty$ due to \cref{eq:finitedual}, which follows from \Cref{ass:Existenceofsecondorderpoint}. Therefore,  
\begin{align*}
\sum_{k=0}^{K-1}\E\left[\|c(x_k)\|^2\right] < \infty
\,,
\end{align*}
which implies that 
\begin{align*}
\lim_{k \rightarrow \infty}\E\left[\|c(x_k)\|^2\right] = 0.
\end{align*}
Moreover, 
\begin{align*}
 \min_{0\leq k \leq K-1} \E[\|c(x_k)\|^2] & \leq \frac{1}{K}\sum_{k = 0}^{K-1} \E[\|c(x_k)\|^2] \\
 & \leq \frac{4(1+\theta_e^2)}{\alpha (1 - \theta_e^2) K}  [q_{\alpha}(\lambda_0) - q^*] + \frac{4}{1-\theta_e^2} \frac{\tau_0a_\infty}{K}.
\end{align*}
We will now analyze the stationarity error~\cref{eq:Sublinearconv_norm_stationarity}. Using \cref{eq:LagrangianFunction}, $\lambda_{k+1} = \lambda_k - \alpha c(x_k)$, and \cref{eq:AugmentedLagrangian}, it follows that,
\begin{align*}
x_k - \eta \nabla \ell_x(x_k,\lambda_{k+1}) &= x_k - \eta(\nabla f(x_k) - \langle \lambda_{k+1}, \nabla c(x_k)) \rangle) \\
&=  x_k - \eta(\nabla f(x_k) - \langle \lambda_{k} - \alpha c(x_k), \nabla c(x_k) \rangle) \\
&= x_k - \eta \nabla_x \L(x_k, \lambda_k;\alpha). 
\end{align*}
Therefore, 
\begin{align}
\left\|\frac{\proj_{\X} (x_k - \eta \nabla \ell_x(x_k,\lambda_{k+1})) - x_k}{\eta}\right\|^2
&=
\left\|R(x_{k},\lambda_{k};\alpha,\eta)\right\|^2
.
\label{eq:optToAuglag}
\end{align}
If statement \ref{item:tolr-II} holds, then it follows from \cref{eq:AugLipIneq} that
\begin{align*}
\E\left[\left\|R(x_{k},\lambda_{k};\alpha,\eta)\right\|^2\right] &\leq \frac{2}{\eta}\left(\E\left[\L(x_k, \lambda_k;\alpha) - \L(x^*_k,\lambda_k;\alpha)\right]\right) \\
&\leq \frac{\alpha \theta_e^2}{\eta} \E[\|c(x^*_k)\|^2] + \frac{\alpha}{\eta}\tau_k.
\end{align*}
If statement \ref{item:tolr-III} holds, then
\begin{align*}
\E\left[\left\|R(x_{k},\lambda_{k};\alpha,\eta)\right\|^2\right] 
&\leq \Tilde \theta_{e}^2\E[\|c(x^*_k)\|^2] + \Tilde \tau_{k}.
\end{align*}
In turn, if either statements \ref{item:tolr-II} or \ref{item:tolr-III} holds, it follows that
\begin{align}\label{eq:rconvbnd}
\E\left[\left\|R(x_{k},\lambda_{k};\alpha,\eta)\right\|^2\right] 
&\leq \max\bigg\{\ \frac{\alpha \theta_e^2}{\eta},\Tilde \theta_{e}^2\bigg\}\E\big[\|c(x^*_k)\|^2\big] + \max\bigg\{ \frac{\alpha }{\eta} \tau_k, \Tilde \tau_{k}\bigg\}.
\end{align}
Taking the full expectation and summing the inequality \cref{eq:conv_exact_feas_bnd} from $k=0$ to $K-1$, and observing that $q^* > -\infty$, we arrive at
\begin{align}
\sum_{k=0}^{K-1}\E\left[\|c(x_k^*)\|^2\right] \leq \frac{2}{\alpha(1 - \theta_e^2)} [q_{\alpha}(\lambda_0) - q^*] + \frac{1}{1-\theta_e^2}\tau_0a_{\infty} < \infty,
\end{align}
which implies that 
\begin{align}\label{eq:cxstarconv}
\lim_{k \rightarrow \infty}\E\left[\|c(x_k^*)\|^2\right] = 0.
\end{align}
Taking limits on both sides of \cref{eq:rconvbnd} and using \cref{eq:cxstarconv} completes the proof. \end{proof}
\Cref{thm:Sublinearconv_norm} establishes that the \emph{expected} {feasibility error} vanishes as $k\to 0$, and meanwhile, the \emph{smallest} {feasibility error} converges to zero at a sublinear rate.\linebreak Moreover, the {stationarity error} also converges to zero \emph{in expectation} when either \Cref{cond:tolr-II} or \Cref{cond:tolr-III} holds. However, the theorem does not guarantee any rate of convergence of the {stationarity error}. To establish such a result, we can perform one additional update at the iterate at which the expected feasibility error attains a minimum. \begin{corollary}\label{corr:postop}
Suppose \Cref{ass:Contdiff,ass:Existenceofsecondorderpoint,ass:variance} hold. Let $k_*$ be the iteration number at which $\min_{0\leq k \leq K-1} \E[\|c(x_k)\|^2]$ is attained. That is,
\begin{align*}
\E[\|c(x_{k_*})\|^2] = \min_{0\leq k \leq K-1} \E[\|c(x_k)\|^2]. \end{align*}
For any given $\tilde \alpha > 0$, $h \geq 0$, and $0 < \tilde \eta < \frac{1}{L + \tilde \alpha \|A\|^2}$, let $\tilde{x}$ be an approximate minimizer of $\min_{x\in \X} \L(x,\lambda_{k_*},\tilde \alpha)$ obtained with the starting point $x_{k_*}$ that satisfies either of the following two statements:
\begin{enumerate}[(i)]
    \item $\tilde{x}$ satisfies $$\E_{k_*}\left[\L(\tilde{x}, \lambda_{k_*};\tilde \alpha) - \L(x^*_{k_*},\lambda_{k_*};\tilde \alpha)\right] \leq \frac{\tilde \alpha h}{2K}\,; \text{ or }$$
    \label{item:postopI}
    \item \Cref{ass:StrcvxALwrtx} also holds and $\tilde{x}$ satisfies $$\E_{k_*}\left[\| R(\tilde{x},\lambda_{k_*};\tilde \alpha,\tilde \eta) \|^2\right] \leq \frac{\mu^2 h}{4\|A\|^2K}\,.$$     \label{item:postopII}
\end{enumerate} 
Then, for $\tilde{\lambda} = \lambda_{k_*} - \tilde \alpha c(\tilde{x})$, we have
\begin{subequations}
\label{eq:postop}
\begin{align}
\label{eq:postop_a}
\E[\|c(\tilde{x})\|^2] & \leq \frac{2h}{K} + \frac{8(1+\theta_e^2)}{\alpha (1 - \theta_e^2) K}  [q_{\alpha}(\lambda_0) - q^*] + \frac{8}{1-\theta_e^2} \frac{\tau_0a_\infty}{K}
 \end{align}
 and
 \begin{align}
\label{eq:postop_b}
    \E\left[\left\|\frac{\proj_{\X} (\tilde{x} - \rb{\tilde \eta} \nabla \ell_x(\tilde{x},\tilde{\lambda})) - \tilde{x}}{\rb{\tilde\eta}}\right\|^2\right] \leq \frac{b}{K},
\end{align}
\end{subequations}
where $b = \rb{\frac{\tilde \alpha h}{\tilde \eta}}$ if \ref{item:postopI} holds and $b =\frac{\mu^2 h}{4\|A\|^2}$ if \ref{item:postopII} holds. \end{corollary}
\begin{proof}
If either \ref{item:postopI} or \ref{item:postopII} is satisfied, \rb{then using $\tilde \alpha$ as the penalty parameter in  \cref{eq:cond2Tocond1} and \cref{eq:cond3Tocond1}}, it follows that
\begin{align*}
    \E\rb{_{k_*}}[\|c(\tilde{x}) - c(x_{k_*})\|^2] \leq \frac{h}{K}.
\end{align*}
Therefore, \rb{ taking the full expectation, }
\begin{align*}
 \E[\|c(\tilde{x})\|^2] &= \E[\|c(\tilde{x}) - c(x_{k_*}) + c(x_{k_*})\|^2] \\
 &\leq 2\E[\|c(\tilde{x}) - c(x_{k_*})\|^2] + 2\E[\|c(x_{k_*})\|^2] \\
 & \leq \frac{2h}{K} + \frac{8(1+\theta_e^2)}{\alpha (1 - \theta_e^2) K}  [q_{\alpha}(\lambda_0) - q^*] + \frac{8}{1-\theta_e^2} \rb{\frac{\tau_0a_\infty}{K}},
\end{align*}
\rb{where the first inequality is due to $\|a+b\|^2 \leq 2 \|a\|^2 + 2 \|b\|^2$ for any $a,b \in \R^n$, and the last inequality follows from \cref{eq:mincxksublconvresult}}. Now, consider the {stationarity error}. Similar to \cref{eq:optToAuglag}, we can show that
\begin{align}\label{eq:postoptToAuglag}
 \E\left[\left\|\frac{\proj_{\X} (\tilde{x} - \rb{\tilde \eta} \nabla \ell_x(\tilde{x},\tilde{\lambda})) - \tilde{x}}{\rb{\tilde \eta}}\right\|^2\right] = \E\left[\| R(\tilde{x},\lambda_{k_*};\rb{\tilde \alpha},\rb{\tilde \eta}) \|^2\right]
 \,.
\end{align}
Therefore, if \ref{item:postopI} holds, it follows from~\cref{eq:AugLipIneq} \rb{with $\tilde \alpha$} as the penalty parameter that
\begin{align*}
\E\left[\| R(\tilde{x},\lambda_{k_*};\rb{\tilde \alpha},\rb{\tilde \eta}) \|^2\right] &\leq \frac{2}{\rb{\tilde \eta}}\left(\E\left[\L(\tilde{x}, \lambda_{k_*};\rb{\tilde \alpha}) - \L(x^*_{k_*},\lambda_{k_*};\rb{\tilde \alpha}\right] \right)
\leq \frac{\rb{\tilde \alpha} h}{\rb{\tilde \eta} K}
\,.
\end{align*}
Likewise, if \ref{item:postopII} holds, then
$$\E\left[\| R(\tilde{x},\lambda_{k_*};\rb{\tilde \alpha},\rb{\tilde \eta}) \|^2\right] \leq \frac{\mu^2 h}{4\|A\|^2K}.$$ 
Substituting these inequalities into \cref{eq:postoptToAuglag} completes the proof. 
\end{proof}

\subsection{Sample Complexity}
We now establish the
sample complexity for our inexact augmented Lagrangian algorithm, i.e., we estimate the worst-case expected total number of stochastic gradient evaluations to reach an $\epsilon-$accurate solution.
To define accuracy, we specifically consider the following metric:
\begin{align}\label{eq:epsiloncompcond}
\max\left\{\E[\|c(\tilde{x})\|^2],  \E\left[\left\|\frac{\proj_{\X} (\tilde{x} - \rb{\tilde \eta} \nabla \ell_x(\tilde{x},\tilde{\lambda})) - \tilde{x}}{\rb{\tilde \eta}}\right\|^2\right]\right\} \leq \epsilon,
\end{align}
for some $\epsilon \in (0,1)$.
For the sake of brevity in this analysis, we employ \Cref{cond:tolr-II} as the inexactness condition with $\theta_e = 0$. At any outer iteration $k$, $x_{k-1}$ is used as the starting point in the adaptive sampling proximal gradient method to solve the inner subproblem  \cref{eq:augLagProb} until \Cref{cond:tolr-II} is satisfied.
Recall that we define the index for the inner iterations as $t$, and the iterates in the inner loop as $x_{k,t}$.
Since $x_{k-1}$ is used as the starting iterate, we set $x_{k,0} \defeq x_{k-1}$.

The adaptive sampling projected gradient method used to solve the inner subproblems (see \Cref{subsec: adasample}) converges at a sublinear rate (cf.~\cite[Theorem 2.11]{beiser2020adaptive},\cite[Theorem 3.7]{xie2020constrained}). The following theorem reformulates this result for augmented Lagrangian subproblems~\cref{eq:augLagProb}.
\begin{theorem}\label{thm:cvxsublresult}
Suppose \Cref{ass:Contdiff,ass:StrcvxALwrtx} hold. If $\eta = \frac{(1 - 2\theta_g^2)}{L + \alpha \|A\|^2}$ with $\theta_g \in [0,\frac{1}{\sqrt{2}})$ and \Cref{con:Normcondition1proj} is satisfied, then for any outer iteration $k \in \N$ and inner iteration $t \in \rb{\N_+}$, it holds that
\begin{align}\label{eq:SubLinRateAugLag}
\E_k [\L(x_{k,t},\lambda_k;\alpha) - \L(x_k^*,\lambda_k;\alpha)] \leq  \frac{(L + \alpha \|A\|^2)\min_{x^*_k \in \X^*_k}\| x_{k-1} - x^*_k \|^2}{2(1 - \rb{2\theta_g^2})t}
\,,
\end{align}
where $\X^*_k = \arg\min_{x \in \X} \L(x,\lambda_k;\alpha)$. 
\end{theorem}
Recall that $\E_k$ denotes expectation conditioned on the filtration $\mathbb{T}_k$, and note that the initial distance to optimality is in this filtration, i.e., $\min_{x^*_k \in \X^*_k}\| x_{k-1} - x^*_k \|^2 \in \mathbb{T}_k$.
Adaptive sampling methods are more efficient and robust in practice than methods that increase the sample sizes at predetermined rates. However, their sample complexity analysis has proven to be difficult, and establishing an upper bound on the sample sizes at each iteration poses significant challenges. Therefore, we make the following assumption based on the sample size growth rate over inner iterations $t$.
\begin{assumption} \label{ass:arithgrowth}
At any given outer iteration $k \in \N$, the expected sample size required to satisfy \Cref{con:Normcondition1proj} increases at \rb{a polynomial rate} over the inner iterations $t$. More specifically, there exists \rb{$ c_0 \geq 0$} and \rb{$\delta_0 > 0$ arbitrarily close to zero}, such that
\begin{equation}\label{eq:arithgrowth}
\E_{k}[|S_{k,t}|] \rb{=} \frac{c_0\omega}{\theta^2_g\min_{x^*_k \in \X^*_k}\|x_{k-1} - x^*_k\|^2} \rb{(t+1)^{1+\delta_0}} \quad \forall k, t \in \N,
\end{equation}
where $\X^*_k = \arg\min_{x \in \X}\L(x,\lambda_k;\alpha)$ and $\omega$ is defined in \cref{eq:bndvar}. 
\end{assumption}
\rb{Predetermined sample growth rates similar to \Cref{ass:arithgrowth} are employed in unconstrained and constrained stochastic optimization settings \cite{pasupathy2018sampling,berahas2022adaptive}.}
We acknowledge that~\Cref{ass:arithgrowth} pertains to algorithmic quantities and is, therefore, less than ideal.
Nevertheless, while we cannot rigorously prove this statement, we provide the following set of supporting (heuristic) arguments.

Consider rewriting \Cref{con:Normcondition1proj} in the following way:
\begin{align}\label{eq:batchsize_def}
b_{k,t} \defeq \frac{\E_{\zeta}[\|\nabla f(x_{k,t},\zeta) - \nabla F(x_{k,t})\|^2]}{ \theta_g^2 \| \E_{k,t}[R_{S_{k,t}}(x_{k,t},\lambda_{k};\alpha,\eta)] \|^2} \leq |S_{k,t}|
\,.
\end{align}
\rb{This inequality is tight, i.e., $|S_{k,t}| = b_{k,t}$ when \Cref{con:Normcondition1proj} is satisfied with equality. In this case,}
due to \cref{eq:bndvar}, it follows that
\begin{align}\label{eq:boundonbkt}
    b_{k,t} \leq \frac{\omega}{\theta_g^2 \| \E_{k,t}[R_{S_{k,t}}(x_{k,t},\lambda_{k};\alpha,\eta)]\|^2}.
\end{align}
On the other hand, using \cref{eq:SubLinRateAugLag} and taking the expected value of both sides of \cref{eq:AugLipIneq} yields
\begin{align}\label{eq:sublinearR}
\E_k [\| R(x_{k,t},\lambda_{k};\alpha,\eta)\|^2] \leq \frac{(L + \alpha \|A\|^2)\min_{x^*_k \in \X^*_k}\| x_{k-1} - x^*_k \|^2}{\eta(1 - \rb{2\theta_g^2})t}, \rb{\quad \forall t\in \N_+.}
\end{align}
This inequality implies that the expected squared norm of the reduced gradient goes to zero at a sublinear rate.

Now, recall~\cref{eq:Equivalence}. In particular,
\begin{align*}
\|\E_{k,t}&[R_{S_{k,t}}(x_{k,t},\lambda_{k};\alpha,\eta)]\|^2 \leq \frac{1}{(1 - \theta_g)^2}\|R(x_{k,t},\lambda_k;\alpha,\eta)]\|^2
\,.
\end{align*}
Taking the conditional expectation $\E_k$ of both sides and invoking \cref{eq:sublinearR}, it follows that
\begin{align*}
\E_k[\|\E_{k,t}&[R_{S_{k,t}}(x_{k,t},\lambda_{k};\alpha,\eta)]\|^2] \leq \frac{(L + \alpha \|A\|^2)\min_{x^*_k \in \X^*_k}\| x_{k-1} - x^*_k \|^2}{(1 - \theta_g)^2\,\eta(1 - \rb{2\theta_g^2})\,t} , \rb{\quad \forall t\in \N_+.}
\end{align*}
This inequality implies that the expected squared norm of the stochastic reduced gradient goes to zero at a sublinear rate. \rb{Therefore, it is possible to replace the right-hand-side of \Cref{con:Normcondition1proj} with a sublinearly convergent sequence ($(t+1)^{-(1 + \delta_0)}$) for any $t \in \N$ and achieve a similar sublinear convergence result as in \Cref{thm:cvxsublresult}. In such a scenario, the sample sizes satisfy \Cref{ass:arithgrowth}. For the sake of brevity, in the rest of our analysis, we assume $\delta_0 = 0$ since $\delta_0$ is arbitrarily close to zero.}
We now state an equation that is useful to bound finite sum expressions in the complexity analysis. For any $\delta > 0$ and $K \in \N_+$, we have
\begin{equation}\label{eq:Boundonsumofterms}
\sum_{k = 0}^{K} k^{1 + \delta} < \int_{t=0}^{K+1} t^{1 + \delta} dt = \frac{(K+1)^{2 + \delta}}{2 + \delta}.
\end{equation}
We are now ready to prove the main theorem about outer iteration and sample complexity. 
\begin{theorem}
\label{thm:TotalWorkComplexity}
Suppose \Cref{ass:Contdiff,ass:Existenceofsecondorderpoint,ass:arithgrowth} hold \rb{with $\delta_0 = 0$}. For any $x_{-1}, \lambda_0$ and $\alpha > 0$, let $\{(x_k, \lambda_k)\}$ be the sequence of primal-dual iterates generated by \cref{alg:StochALwInexact} where $x_k$ satisfies \Cref{cond:tolr-II} at each outer iteration $k$ with $\theta_e = 0$, $\tau_k = \frac{\tau_0}{(k+1)^{1+\delta/2}}$, \rb{$a_\infty = \sum_{k=0}^{\infty} \frac{1}{(k+1)^{1+\delta/2}}$ }, with $\tau_0 \geq 0$, and $\delta > 0$. Suppose the sample sizes $|S_{k,t}|$ satisfy \Cref{con:Normcondition1proj} with $\theta_g \in [0, \frac{1}{\sqrt{2}})$, $\eta = \frac{(1 - 2\theta_g^2)}{L + \alpha \|A\|^2}$. Under the conditions of \Cref{corr:postop} with $\tilde{x}$ satisfying $(i)$ with $h \geq 0$,  \rb{$\tilde \eta = \frac{(1 - 2\theta_g^2)}{L + \tilde \alpha \|A\|^2}$ and $\tilde \alpha > 0$}, the number of outer iterations to get an $\epsilon-$accurate solution $(\tilde{x}, \tilde{\lambda})$ satisfying \cref{eq:epsiloncompcond} is
\begin{equation}\label{eq:kbound}
K_\epsilon = \left\lceil \frac{1}{\epsilon}\max\left\{ \frac{8}{\alpha } [q_{\alpha}(\lambda_0) - q^*] 
+ 2(4\tau_0 a_{\infty} + h), \frac{\tilde \alpha h}{\tilde \eta}\right\}\right \rceil.
\end{equation}
\rb{Moreover, if the gradients of the Lagrangian function are bounded, i.e., $\\ \|\nabla \L(x_{k-1},\lambda_k;\alpha)\|^2 \leq D_{\L}$ for all $k \in \N$, } then the expected number of stochastic gradient evaluations is 
\begin{equation}\label{eq:workbnd}
\E[\mathcal{W}] \leq \frac{B(K_\epsilon+1)^{3 + \delta}}{\alpha^2\tau_0^2 (3 + \delta)}  + \frac{\tilde{B} K_\epsilon^2}{\rb{\tilde\alpha^2 h^2}}, \end{equation}
where 
\begin{equation*}
B = \frac{2c_0\omega} {\theta_g^2}\left(\frac{(L + \alpha \|A\|^2)^2D\ck{^2}}{(1 - 2\theta_g^2)^2} + 4D_{\L}\right),
\end{equation*}
and
\begin{equation*}
\tilde{B} = \frac{2c_0\omega} {\theta_g^2}\left(\frac{(L + \tilde{\alpha} \|A\|^2)^2D\ck{^2}}{(1 - 2\theta_g^2)^2} + 4D_{\L}\right).
\end{equation*}
\end{theorem}

\begin{proof}
\rb{Substituting $\theta_e = 0$ into \cref{eq:postop}, we obtain
\begin{align*}
    \E[\|c(\tilde{x})\|^2] &\leq  \frac{8}{\alpha K}  [q_{\alpha}(\lambda_0) - q^*] + \frac{2(4\tau_0 a_\infty + h)}{K}
    \,,
 \end{align*}}
 and
 \begin{align*}
    \E\left[\left\|\frac{\proj_{\X} (\tilde{x} - \rb{\tilde \eta} \nabla \ell_x(\tilde{x},\tilde{\lambda})) - \tilde{x}}{\rb{\tilde\eta}}\right\|^2\right] \leq \rb{\frac{\tilde \alpha h}{\tilde \eta K}}
    \,.
\end{align*}
Therefore, for any $K \geq K_\epsilon$  defined in \cref{eq:kbound}, $(\tilde{x},\tilde{\lambda})$ satisfies \cref{eq:epsiloncompcond}. 

Let $T_k$ be the first inner iteration at which \Cref{cond:tolr-II} is satisfied with $\theta_e = 0$.
If $T_k = 0$, then we would have a sufficiently accurate starting point for the algorithm to terminate before the first complete iteration.
Therefore, without loss of generality, we assume that $T_k > 0$.
By \Cref{thm:cvxsublresult}, the inner subproblem termination condition,
\rb{\Cref{cond:tolr-II} with $\theta_e = 0$}, is satisfied at a given inner iteration $t \in \N_+$ if 
\begin{equation*}\frac{(L + \alpha \|A\|^2)\min_{x^*_k \in \X^*_k}\| x_{k-1} - x^*_k \|^2}{2(1 - \rb{2\theta_g^2})t} \leq \frac{\alpha \tau_k}{2}.
\end{equation*}
Thus, we have a deterministic upper bound $\Omega_k > 0$ on the random variable $T_k$; namely,
\begin{equation}\label{eq:stoppingtime}
T_k \leq \ck{\Omega_k \defeq} \left \lceil\frac{(L + \alpha \|A\|^2)}{(1-\rb{2\theta_g^2}) \alpha\tau_k} \min_{x^*_k \in \X^*_k}\| x_{k-1} - x^*_k \|^2 \right \rceil.
\end{equation}
We now analyze the total number of expected stochastic gradient evaluations. First, consider the expected sample complexity at each outer iteration $k$:
\begin{align}
\E_k [\mathcal{W}_k] = \E_k\left[\sum_{t=0}^{\rb{T_k -1}}|S_{k,t}|\right] &\ck{\leq} \frac{c_0 \omega}{\theta^2_g\min_{x^*_k \in \X^*_k}\|x_{k-1} - x^*\|^2}\sum_{t = 0}^{\rb{\ck{\Omega_k - 1}}} \rb{(t+1)} \nonumber \\
&\leq \frac{c_0 \omega \Omega_k^2}{\theta^2_g\min_{x^*_k \in \X^*_k}\|x_{k-1} - x^*\|^2}\,, \label{eq:compperouteriter}
\end{align}
where the first \ck{inequality} is due to \Cref{ass:arithgrowth} \rb{with $\delta_0 = 0$}. Substituting \cref{eq:stoppingtime} into \cref{eq:compperouteriter}, \rb{ using $\lceil x \rceil \leq x + 1$ and $\|a +b\|^2 \leq 2\|a\|^2 + 2\|b\|^2$ for any $a,b\in \R^n$}, we have that
\rb{
\begin{align}\label{eq:workcompperiter}
   \E_k [\mathcal{W}_k] \leq \frac{2c_0\omega(L + \alpha \|A\|^2)^2\min_{x^*_k \in \X^*_k}\|x_{k-1} - x^*_k\|^2}{\theta^2_g(1 - \rb{2\theta_g^2})^2\alpha^2\tau_k^2} + \frac{2c_0 \omega}{\theta^2_g\min_{x^*_k \in \X^*_k}\|x_{k-1} - x^*\|^2} . 
\end{align}}
\rb{Now, $T_k > 0$ implies that \Cref{cond:tolr-II} is violated at $x_{k,0} = x_{k-1}$. That is,  
\begin{align*}
\E_k\left[\L(x_{k-1}, \lambda_k;\alpha) - \L(x^*_k,\lambda_k;\alpha)\right] >  \frac{\alpha\tau_k}{2}
\,.
\end{align*}
Recalling that $x_{k-1}$ is in the filtration $\mathbb{T}_k$, and using convexity of $\L$, it follows that
\begin{align*}
   \frac{\alpha\tau_k}{2} &<  \L(x_{k-1}, \lambda_k;\alpha) - \L(x^*_k,\lambda_k;\alpha) \\
   & \leq \|\nabla \L(x_{k-1},\lambda_k;\alpha)\|\|x_{k-1} - x^*_k\| \\
   & \leq \sqrt{D_{\L}}\|x_{k-1} - x^*_k\|    \,,
\end{align*}
for all $x^*_k \in \X^*_k$.
Therefore, 
\begin{align}\label{eq:minxbnd}
    \min_{x^*_k \in \X^*_k}\|x_{k-1} - x^*_k\|^2 > \frac{\alpha^2 \tau_k^2}{4D_{\L}}
\end{align}
}\rb{Now, summing the inequality \cref{eq:workcompperiter} from $k= 0$ to \rb{$K-1$}, taking full expectation, using \cref{eq:compactD,eq:minxbnd}, $\tau_k = \tau_0 \rb{(k+1)}^{-1-\delta/2}$, and \cref{eq:Boundonsumofterms}}, it follows that
\begin{align}
\E \left[ \sum^{\rb{K-1}}_{k = 0} \mathcal{W}_k\right] &\leq  \frac{2c_0\omega(L + \alpha \|A\|^2)^2}{\theta_g^2(1 - \rb{2\theta_g^2})^2\alpha^2} \sum_{k=0}^{\rb{K-1}}\frac{\E[\min_{x^*_k \in \X^*_k}\|x_{k-1} - x^*_k\|^2]}{\tau_k^2} + \frac{8c_0\omega D_{\L}}{\theta_g^2\alpha^2}\sum_{k=0}^{K-1}\frac{1}{\tau_k^2}\nonumber \\
&\leq \frac{B} {\alpha^2\tau_0^2}\sum_{k=0}^{\rb{K-1}}  \rb{(k+1)}^{2+\delta} \nonumber \\
&\leq \frac{B} {\alpha^2\tau_0^2(3 + \delta)}(K+1)^{3 + \delta}\,.
\end{align}
We now consider the total number of stochastic gradients evaluated in the final step described in \Cref{corr:postop} with $\tilde{x}$ satisfying \ref{item:postopI}. \rb{Following a similar approach to the derivation of \cref{eq:workcompperiter,eq:minxbnd}, and using \cref{eq:compactD}, we have that} \begin{align}
 \E [\tilde{W}]  &\leq \frac{2c_0\omega(L + \rb{\tilde \alpha} \|A\|^2)^2 K^2 D\ck{^2}}{\theta_g^2(1 - \rb{2\theta_g^2})^2\rb{\tilde\alpha^2\tilde h^2}} + \frac{8c_0\omega K^2 D_{\L}}{\theta_g^2\rb{\tilde\alpha^2\tilde h^2}}
 =\frac{\tilde{B} K^2}{\rb{\tilde\alpha^2 h^2}}
  \,.
\end{align}
Finally, we can define the expected total number of gradient evaluations as
\begin{align}
\E [\mathcal{W}] &= \E \left[ \sum^{\rb{K-1}}_{k = 0} \mathcal{W}_k\right] + \E[\mathcal{\tilde{W}}] \label{eq:WorkComplexityStep1}\\
&\leq \frac{B(K+1)^{3 + \delta}}{\alpha^2\tau_0^2 (3 + \delta)}  + \frac{\tilde{B} K^2}{\rb{\tilde\alpha^2 h^2}}\label{eq:WorkComplexityStep2}
\,.
\end{align}
Substituting $K = K_{\epsilon}$ into \cref{eq:WorkComplexityStep2} completes the proof. 
\end{proof}

\rb{
\begin{remark}
In \Cref{thm:TotalWorkComplexity}, we state an additional assumption related to the boundedness of the gradients of the augmented Lagrangian functions at the iterates computed by the algorithm. We note that this is a mild assumption and can be proven using \Cref{ass:Contdiff,ass:Existenceofsecondorderpoint}, and if the dual variables $\lambda_k$ are bounded. Due to the convergence results established in \Cref{subsec:convergence}, it is reasonable to assume that the dual variables are bounded. 
\end{remark}
}

\subsection[Sample Complexity II]{Sample Complexity: $\alpha = \mathcal{O}(\epsilon^{-1})$}

\Cref{thm:TotalWorkComplexity} establishes total outer iteration complexity, $K_\epsilon$, and expected sample complexity, $\E[ \mathcal{W}]$, for any choice of the penalty parameter $\alpha$. If $\alpha$ and the other parameters (e.g., $\tau_0, h$) given in \Cref{thm:TotalWorkComplexity} are chosen to be independent of the accuracy $\epsilon$, then $K_\epsilon = \mathcal{O}(\epsilon^{-1})$ and $\E[ \mathcal{W}] = \mathcal{O}(\epsilon^{-3-\delta})$. However, this sample complexity bound is not tight as the optimal sample complexity for stochastic convex programs is  $\mathcal{O}(\epsilon^{-2})$ \cite{xu2020primal,lan2020algorithms}. The next corollary establishes that this optimal sample complexity can be achieved when $\alpha = \mathcal{O}(\epsilon^{-1})$.   

\begin{corollary}\label{corr:optimalworkcomp}
    Under the conditions of \Cref{thm:TotalWorkComplexity}, if $\alpha = c_{\alpha} \epsilon^{-1}$, $\tau_0 = c_{\tau}\epsilon$, and $h = c_{h}\epsilon$ for some $c_\alpha, c_{\tau}, c_h \in (0,\infty)$. Then $
        K_{\epsilon} =  \mathcal{O}(1)
    $
    and 
    \begin{align}
        \E[ \mathcal{W}] = \mathcal{O}(\epsilon^{-2}).
    \end{align}
\end{corollary}
\begin{proof}
    Substituting $\alpha, \tau_0,$ and $h$ values into \cref{eq:kbound}, it follows that
    \begin{align}
        K_{\epsilon} &\leq 1 + \frac{1}{\epsilon}\max\left\{ \frac{8\epsilon}{c_{\alpha}} [q_{\alpha}(\lambda_0) - q_{\alpha}(\lambda^*)] 
+ 2\epsilon(4c_{\tau}a_{\infty} + c_h), \frac{\tilde \alpha c_h\epsilon}{\tilde \eta}\right\}  \nonumber \\
& = 1 + \max\left\{ \frac{8}{c_{\alpha}} [q_{\alpha}(\lambda_0) - q_{\alpha}(\lambda^*)] 
+ 2(4c_{\tau}a_{\infty} + c_h), \frac{\tilde \alpha c_h}{\tilde \eta}\right\}  \nonumber \\
& = \mathcal{O}(1)\,. \label{eq:itercompo1}
\end{align}
We now analyze the sample complexity.  Using $\alpha, \tau_0$, and $h$ values, and $\epsilon < 1$, it follows that 
\begin{align}
    \frac{B(K_\epsilon+1)^{3 + \delta}}{{\alpha^2\tau_0^2 (3 + \delta)}} &= \frac{2c_0\omega} {\theta_g^2c_{\alpha}^2c_{\tau}^2 \ck{(3 + \delta)}}\left(\frac{(L + c_{\alpha}\epsilon^{-1} \|A\|^2)^2D\ck{^2}}{(1 - 2\theta_g^2)^2} + 4D_{\L}\right)(K_\epsilon+1)^{3 + \delta} \nonumber \\
    &\leq \frac{2c_0\omega} {\epsilon^2\theta_g^2c_{\alpha}^2c_{\tau}^2 \ck{(3 + \delta)}}\left(\frac{(L + \ck{c_{\alpha}} \|A\|^2)^2D\ck{^2}}{(1 - 2\theta_g^2)^2} + 4D_{\L}\right)(K_\epsilon+1)^{3 + \delta} \label{eq:comp_1},\end{align}
and 
\begin{align}\label{eq:comp_2}
\frac{\tilde{B} K_\epsilon^2}{\rb{\tilde\alpha^2 h^2}} &= \frac{2c_0\omega} {\epsilon^2\theta_g^2\rb{\tilde\alpha^2 c_h^2}}\left(\frac{(L + \tilde{\alpha} \|A\|^2)^2D\ck{^2}}{(1 - 2\theta_g^2)^2} + 4D_{\L}\right)K_\epsilon^2. \end{align}
Substituting \cref{eq:comp_1,eq:comp_2} in \cref{eq:workbnd}, we have that,
\begin{align*}
\E[\mathcal{W}] &\leq \frac{2c_0\omega} {\epsilon^2\theta_g^2c_{\alpha}^2c_{\tau}^2\ck{(3 + \delta)}}\left(\frac{(L + \ck{c_\alpha} \|A\|^2)^2D\ck{^2}}{(1 - 2\theta_g^2)^2} + 4D_{\L}\right)(K_\epsilon+1)^{3 + \delta} \\
&~~~~~~ + \frac{2c_0\omega} {\epsilon^2\theta_g^2\rb{\tilde\alpha^2 c_h^2}}\left(\frac{(L + \tilde{\alpha} \|A\|^2)^2D\ck{^2}}{(1 - 2\theta_g^2)^2} + 4D_{\L}\right)K_\epsilon^2 \\
&= \mathcal{O}(\epsilon^{-2}),
\end{align*}
where the last equality is due to the fact that all other constants in the inequality are independent of the choice of $\epsilon$.~
\end{proof}

\begin{remark}
We observe that the complexity results given in \Cref{thm:TotalWorkComplexity} and \Cref{corr:optimalworkcomp} do not exploit the benefits of using the previous iterate $x_{k-1} = x_{k,0}$ as the starting point for solving the current subproblem. That is, the bound  on $\E\left[\min_{x^*_k \in \X^*_k}\|x_{k-1} - x^*_k\|^2\right]$ is not tight. The difficulty in exploiting the benefits of this procedure is due to the fact that the augmented Lagrangian functions are only convex but not necessarily strongly convex. In \Cref{subsec:special}, we consider strongly convex functions and establish the advantages of this procedure. 
\end{remark}

\subsection[Sample Complexity III]{Sample Complexity: $\X = \R^n$}\label{subsec:special}
We provide improved convergence and complexity results when $\X = \R^n$ and the objective function $f$ is $\mu$-strongly convex. \begin{assumption}
\label{ass:Strcvxf}
The objective function $f$ is $\mu$-strongly convex. That is,
\begin{align*}
\nabla^2 f(x)  \succeq \mu I \quad \forall x \in \R^n
\,
\end{align*}
where $I \in \R^{n \times n}$ is an identity matrix.  
\end{assumption}
We should note that \Cref{ass:Strcvxf} implies \Cref{ass:StrcvxALwrtx}. In this case, the inner subproblems are unconstrained and have unique optimal solutions. Moreover, the optimality conditions given in \cref{eq:firstStatPoint} can be written as
\begin{align*}
\nabla \ell_x(x,\lambda) = 0 \quad \mbox{and} \quad c(x) = 0. 
\end{align*}
It \rb{can} also \rb{be shown} that the negative dual function $q(\lambda)$ is strongly convex in this setting, as stated in the following proposition \rb{(cf.~\cite[Propositions 3.1 and 3.3]{guigues2020strong} and the references therein,\cite[Theorem 1]{zhou2018fenchel}, \cite[Proposition 2.5]{guigues2021inexact}}).\begin{proposition}\label{prop:strnconv}
    If \Cref{ass:Contdiff,ass:Strcvxf} hold \rb{with $\X=\R^n$}, then $q(\lambda)$ defined in \cref{eq:dualFunction} is strongly convex with the strong convexity parameter \rb{$\mu_q = \frac{\sigma}{\mu + L}$} where $\sigma = \lambda_{\min}(AA^T)$. \end{proposition}
\rb{For the sake of completeness, we include the proof of this proposition in \Cref{sec:appendix_proof}.}
We also state the following well-known result for strongly convex functions with Lipschitz continuous gradients (cf.~\cite[Theorem 2.1.5 and Theorem 2.1.10]{nesterov2018lectures})
\begin{proposition}\label{prop:inequalities}
    If the function $q_\alpha(\lambda)$ is strongly convex with parameter $\mu_{\alpha}$ and has a Lipschitz continuous gradient with Lipschitz constant $L_{\alpha}$, then for any  $\lambda \in \R^m$, it holds that
    \begin{align}\label{eq:wellknown}
       2 \mu_{\alpha}(q_{\alpha}(\lambda) - q_{\alpha}(\lambda^*)) \leq \|\nabla q_{\alpha}(\lambda)\|^2 \leq 2L_{\alpha}(q_{\alpha}(\lambda) - q_{\alpha}(\lambda^*)),
    \end{align}
    \ck{where $\lambda^* = \argmin_{\lambda} q_\alpha(\lambda)$.}
\end{proposition}
\ck{Note that $L_{\alpha} = \alpha^{-1}$ by \Cref{lemma:MoreauProperties}}. We now establish a linear rate of convergence of both \emph{feasibility error} and \emph{stationarity error}. For the sake of brevity, we only consider \Cref{cond:tolr-III}. 
\begin{theorem}
\label{thm:Linearconv_norm}
Suppose \Cref{ass:Contdiff,ass:Strcvxf} hold and $\X = \R^n$. For any $x_{-1}, \lambda_0$ and $\alpha > 0$, let $\{(x_k, \lambda_k)\}$ be the sequence of primal-dual iterates generated by \cref{alg:StochALwInexact}. If the primal iterates $x_k$ satisfy \Cref{cond:tolr-III} at each iteration $k \in \N$ with $\Tilde \theta_{e} \leq \frac{\mu \theta_e}{2\|A\|}$, $\Tilde \tau_k = \frac{\mu^2\tau_k}{4\|A\|^2}$, $\theta_e \in [0,1)$, and $\tau_k = \tau_0(1/a)^k$ for some $\tau_0 > 0$ and $a > 1$, then
\begin{align} \label{eq:exp_termination}
    \E[\|c(x_k)\|^2] &\leq  A_1\rho^k \quad \mbox{and} \quad 
    \E\left[\left\|\nabla \ell_x(x_k,\lambda_{k+1})) \right\|^2\right] \leq A_2\rho^k,
\end{align}
\ck{
where $A_1 = 4(1 + \theta_e^2)L_{\alpha}A_3 + 2\tau_0$, $A_2 =2L_{\alpha}\Tilde \theta_{e}^2A_3 + \frac{\mu^2\tau_0}{4\|A\|^2}, \\
A_3 = \max\left\{q_\alpha(\lambda_0) - q_\alpha(\lambda^*), \rb{\frac{\tau_0}{ \mu_{\alpha} (1 - \theta_e^2)}}\right\} $, $\rho = \max\left\{1 - \rb{\frac{\alpha \mu_{\alpha} (1 - \theta_e^2) }{2}}, \frac{1}{a}\right\} < 1$, and $\mu_\alpha = \frac{\mu_q}{\mu_q \alpha + 1}$.
}
\begin{proof}
Using \Cref{prop:strnconv} and \Cref{lemma:MoreauProperties}, it follows that the $q_{\alpha}(\lambda)$ is a strongly convex function with strong convexity parameter \rb{$\frac{\mu_q}{\mu_q \alpha + 1}$}. \rb{Therefore,} substituting \rb{\cref{eq:wellknown}} into \cref{eq:grad_error}, 
\ck{using \cref{eq:gradform}, subtracting $q_\alpha(\lambda^*)$ from both sides and taking full expectation we obtain}
\rb{
\begin{align*}
    \E[q_{\alpha}(\lambda_{k+1}) - q_{\alpha}(\lambda^*)] &\leq \left(1 - \alpha\mu_{\alpha}(1 - \theta_e^2)\right)\E[q_{\alpha}(\lambda_k) - q_{\alpha}(\lambda^*)] + \frac{\alpha \tau_k}{2} \\
    &\leq \left(1 - \alpha\mu_{\alpha}(1 - \theta_e^2)\right)\E[q_{\alpha}(\lambda_k) - q_{\alpha}(\lambda^*)] + \frac{\alpha \tau_0}{2a^k},
\end{align*}
where the second inequality is due to $\tau_k = \tau_0(1/a)^k$.
It is now a straightforward exercise in mathematical induction to show that
\begin{equation}\label{eq:linear_func_result}
    \E[q_{\alpha}(\lambda_{k}) - q_{\alpha}(\lambda^*)] \leq A_3 \rho^k \quad \forall k \in \N. 
\end{equation} 
The statement is trivially true for $k=0$. Let's assume it is true for iteration $k$. For iteration $k+1$, it follows that
\begin{align*}\label{eq:linear_func}
    \E[q_{\alpha}(\lambda_{k+1}) - q_{\alpha}(\lambda^*)] &\leq \left(1 - \alpha\mu_{\alpha}(1 - \theta_e^2)\right)\E[q_{\alpha}(\lambda_k) - q_{\alpha}(\lambda^*)] + \frac{\alpha \tau_0}{2a^k} \\
    &\leq A_3\rho^k\left(1 - \alpha\mu_{\alpha}(1 - \theta_e^2) + \frac{\ck{\alpha} \tau_0}{2A_3(a\rho)^k}\right) \\
    &\leq A_3\rho^k\left(1 - \alpha\mu_{\alpha}(1 - \theta_e^2) + \frac{\alpha\mu_{\alpha}(1 - \theta_e^2)}{2}\right) \\
    &\leq A_3\rho^{k+1}
    \,,
\end{align*} }
where the second inequality is due to the statement of the induction, third inequality is due to $\rho \geq 1/a$ and the definition of $A_3$, and the last inequality is due to the definition of $\rho$. Hence, \cref{eq:linear_func_result} is satisfied. Substituting $\lambda = \lambda_{k}$ in \cref{eq:wellknown}, by \cref{eq:gradform} and taking expectation of both sides it follows that, 
\begin{align}
\E[\|c(x_k^*)\|^2] &\leq 2L_{\alpha}\E[q_{\alpha}(\lambda_{k}) - q_{\alpha}(\lambda^*)] \nonumber\\
&\leq 2L_{\alpha}A_3\rho^k. \label{eq:cxresult}
\end{align}
Therefore, using the definitions of $A_1$ and $\rho$, we have that
\begin{align}
\E[\|c(x_k)\|^2] &\leq 2\E[\|c(x_k^*) - c(x_k)\|^2] + 2\E[\|c(x_k^*)\|^2] \nonumber \\
&\leq 2(1 + \theta_e^2)\E[\|c(x_k^*)\|^2] + 2\tau_k \nonumber \\
&\leq \rho^k\left(4(1 + \theta_e^2)L_{\alpha}A_3 + \frac{2\tau_0}{(\rho a)^k}\right) \nonumber\\
&\leq A_1\rho^k. \nonumber 
\end{align}
 Using \cref{eq:cxresult}, \Cref{cond:tolr-III}, and \cref{eq:optToAuglag}, it follows that
 \begin{align*}
    \E\left[\left\|\nabla \ell_x(x_k,\lambda_{k+1})) \right\|^2\right] &\leq \Tilde \theta_{e}^2\E[\|c(x^*_k)\|^2] + \Tilde \tau_k \\
    &\leq 2L_{\alpha}\Tilde \theta_{e}^2A_3\rho^k + \frac{\mu^2\tau_0}{4\|A\|^2a^k} \\
    &\leq A_2\rho^k
    \,,
 \end{align*}
 where the last inequality is due to definitions of $A_2$ and $\rho a \geq 1$. 
\end{proof}
\end{theorem}
We now derive the sample complexity results.
We will use the fact that the metric \cref{eq:epsiloncompcond} can be simplified to 
\begin{align}\label{eq:epsiloncompcond_str}
\max\left\{\E[\|c(x_k)\|^2],  \E\left[\left\|\nabla \ell_x(x_k,\lambda_{k+1})\right\|^2\right]\right\} \leq \epsilon \in (0,1),
\end{align}
since $\X = \R^n$. Moreover, the adaptive sampling projected gradient method employed to solve the inner subproblems converges at a linear rate as stated below (cf.~\cite[Theorem 2.10]{beiser2020adaptive},\cite[Theorem 3.7]{xie2020constrained}).
\begin{theorem}\label{thm:projgradstrcnvx_func}
Suppose \Cref{ass:Contdiff,ass:variance,ass:Strcvxf} hold. If $\eta = \frac{(1 - 2\theta_g^2)}{L + \alpha \|A\|^2}$ with $\theta_g \in [0,\frac{1}{2})$, and \Cref{con:Normcondition1proj} is satisfied. Then, for any outer iteration $k \in \N$ and inner iteration $t \in \N$,  it holds that
\begin{align}\label{eq:LinRateAugLag}
\E_k [\L(x_{k,t},\lambda_k;\alpha) - \L(x_k^*,\lambda_k;\alpha)] \leq \rho_{\L}^t ( \L(x_{k,0},\lambda_k;\alpha) - \L(x_k^*,\lambda_k;\alpha)) 
\,,
\end{align}
where $\rho_{\L} = 1 - \frac{(1 - 2\theta_g^2)\mu}{L + \alpha \|A\|^2} \in [0,1)$. 
\end{theorem}
Using \Cref{prop:inequalities}, it can be shown that the gradient of the augmented Lagrangian function also converges to zero. That is, applying \Cref{prop:inequalities} to the augmented Lagrangian function, we have that
\begin{align}
   2\mu(\L(x_{k,t},\lambda_k;\alpha) - \L(x^*_k,\lambda;\alpha)) &\leq \|\nabla_x \L(x_{k,t},\lambda_k;\alpha)\|^2 \nonumber \\
   &\leq 2(L + \alpha \|A\|^2)(\L(x_{k,t},\lambda_k;\alpha) - \L(x^*_k,\lambda;\alpha)) \label{eq:well-known-Lag}.
\end{align}
Combining \cref{eq:LinRateAugLag} and \cref{eq:well-known-Lag}, it follows that,
\begin{align}\label{eq:normgrad_convergence}
    \E_k[\|\nabla_x \L(x_{k,t},\lambda_k;\alpha)\|^2] \leq \frac{L + \alpha \|A\|^2}{\mu} \rho_{\L}^t\|\nabla_x \L(x_{k,0},\lambda_k;\alpha)\|^2.   
\end{align}
The next theorem establishes pessimistic upper bounds on the sample sizes employed at each outer iteration $k \in \N$ and each inner iteration $t\in \N$, and the number of inner iterations $T_k$ required to satisfy \Cref{cond:tolr-III}. For the sake of brevity, in this complexity analysis, we employ \Cref{cond:tolr-III} with $\Tilde \theta_{e} = 0$, and also
assume that \Cref{ass:variance} holds with $\omega_1 = 1$ and $\omega_2 = \omega$, i.e., \begin{align}\label{eq:boundedvar}
\E_{\zeta}[\|\nabla F(x,\zeta) - \nabla f(x)\|^2] \leq \omega. \end{align}

\begin{theorem}\label{thm:TotalWorkComplexitystrcvx}
Suppose the conditions of \Cref{thm:projgradstrcnvx_func} are satisfied and \cref{eq:boundedvar} holds. Then the number of inner iterations $T_k$ required to satisfy \Cref{cond:tolr-III} with $\Tilde \theta_{e} = 0$ are bounded from above as follows:
\begin{align}\label{eq:Tkbnd_strcnvx}
    T_k \leq \left \lceil\log_{1/\rho_{\L}} \left(\frac{2(L + \alpha \|A\|^2)\left(\alpha^2\|A\|^2 \|c(x_{k-1})\|^2 + \|R(x_{k-1},\lambda_{k-1};\alpha,\eta)\|^2\right)}{\mu \Tilde \tau_k}\right) \right \rceil.
\end{align} 
Moreover, for any inner iteration $t < T_k$, the sample sizes $|S_{k,t}|$ are at most 
\begin{align}\label{eq:sktbnd_strcnvx}
    |S_{k,t}| \leq \frac{\omega}{\theta_g^2\Tilde \tau_k}.
\end{align} 
 \end{theorem}
\begin{proof}
At any outer iteration $k \in N$, let $T_{k}$ denote the first inner iteration $t$ at which the following condition holds:
\begin{align}\label{eq:min_cond_3}
\min\left\{\| R(x_{k,t},\lambda_{k};\alpha,\eta) \|^2, \E_{\ck{k,t}}\left[\| R(x_{k,t},\lambda_{k};\alpha,\eta) \|^2\right]\right\} \leq \Tilde \tau_k.\end{align}
Hence, at $t = T_k$, \Cref{cond:tolr-III} is satisfied with $\Tilde \theta_{e} = 0$ at $x_k = x_{k,t}$.
Therefore, for all $t < T_{k}$, it follows that \begin{align}\label{eq:r_lowvalues}
    \| R(x_{k,t},\lambda_{k};\alpha,\eta) \|^2 > \Tilde \tau_k. \end{align}
Using \cref{eq:r_lowvalues,eq:boundedvar,eq:Equivalence}, and choosing the smallest sample size $|S_{k,t}|$ satisfying \cref{eq:batchsize_def}, it follows that
\begin{align}
    |S_{k,t}| \leq \frac{\omega\rb{(1 + \theta_g)^2}}{\theta_g^2\Tilde \tau_k}.
\end{align}
Now, let us bound the number of inner iterations required to satisfy \cref{eq:min_cond_3}. Using \cref{eq:normgrad_convergence,eq:redgrad} \ck{and $\X = \R^n$} it follows that
\begin{align}\label{eq:Rconvbnd}
\E_k\left[\left\|R(x_{k,t},\lambda_{k};\alpha,\eta)\right\|^2\right] &\leq \frac{L + \alpha \|A\|^2}{\mu}\rho_{\L}^t\left\|R(x_{k,0},\lambda_{k};\alpha,\eta)\right\|^2.
\end{align}
Substituting $x_{k,0} = x_{k-1}$, it follows that
\begin{align}
&\|R(x_{k-1},\lambda_{k};\alpha,\eta)\|^2 \nonumber\\
&~~~~~=\|R(x_{k-1},\lambda_{k};\alpha,\eta) - R(x_{k-1},\lambda_{k-1};\alpha,\eta) + R(x_{k-1},\lambda_{k-1};\alpha,\eta)\|^2 \nonumber \\ 
&~~~~~\leq 2\|R(x_{k-1},\lambda_{k};\alpha,\eta) - R(x_{k-1},\lambda_{k-1};\alpha,\eta)\|^2 + 2 \|R(x_{k-1},\lambda_{k-1};\alpha,\eta)\|^2 \label{eq:normg_2}
\end{align}
where the last inequality is due to $\|a + b\|^2 \leq 2\|a\|^2 + 2\|b\|^2$ for any $a,b \in \R^n$. Consider 
\begin{align}
 \|R(x_{k-1},\lambda_{k};\alpha,\eta) - & R(x_{k-1},\lambda_{k-1};\alpha,\eta)\| \nonumber \\
 &=\|\nabla_x \L(x_{k-1},\lambda_k;\alpha) - \nabla_x \L(x_{k-1},\lambda_{k-1};\alpha)\| \nonumber \\
 &= \|\langle \lambda_k - \lambda_{k-1}, \nabla c(x_{k-1}) \rangle\| \nonumber \\
 &\leq \alpha \|A\| \|c(x_{k-1})\| \label{eq:normg_3}
\end{align}
where the first equality is due to \cref{eq:redgrad} and the inequality is due to $\lambda_k = \lambda_{k-1} - \alpha c(x_{k-1})$. Using \cref{eq:Rconvbnd}, \cref{eq:normg_2}, and \cref{eq:normg_3}, it follows that
\begin{multline}
\E_k\left[\left\|R(x_{k,t},\lambda_{k};\alpha,\eta)\right\|^2\right] \\
\leq \frac{L + \alpha \|A\|^2}{\mu}\rho_{\L}^t\left(2\alpha^2\|A\|^2 \|c(x_{k-1})\|^2 + 2\|R(x_{k-1},\lambda_{k-1};\alpha,\eta)\|^2\right).
\end{multline}
Therefore, for any 
\begin{align*}
t \geq \left \lceil\log_{1/\rho_{\L}} \left(\frac{2(L + \alpha \|A\|^2)\left(\alpha^2\|A\|^2 \|c(x_{k-1})\|^2 + \|R(x_{k-1},\lambda_{k-1};\alpha,\eta)\|^2\right)}{\mu \Tilde \tau_k}\right) \right \rceil
\end{align*}
we have, 
\begin{align*}
\E_k\left[\left\|R(x_{k,t},\lambda_{k};\alpha,\eta)\right\|^2\right] &\leq \Tilde \tau_k.
\end{align*}
Using \cref{eq:min_cond_3}, it follows that,
\begin{align}
    T_k \leq \left \lceil \log_{1/\rho_{\L}} \left(\frac{2(L + \alpha \|A\|^2)\left(\alpha^2\|A\|^2 \|c(x_{k-1})\|^2 + \|R(x_{k-1},\lambda_{k-1};\alpha,\eta)\|^2\right)}{\mu \Tilde \tau_k}\right) \right \rceil
\end{align}
which completes the proof. 
\end{proof}

We are now ready to provide the main complexity theorem for this subsection. 
\begin{theorem}\label{thm:TotalWorkComplexity_specialcase}
Suppose the conditions of \Cref{thm:Linearconv_norm,thm:projgradstrcnvx_func}, and \cref{eq:boundedvar} hold. Then the number of outer iterations to get an $\epsilon-$accurate solution $(x_k,\lambda_{k+1})$ satisfying \cref{eq:epsiloncompcond_str} is 
\begin{align}\label{eq:iter_comp_strcnv}
    K_{\epsilon} = \left\lceil \log_{1/\rho}\left(\frac{\max\{A_1,A_2\}}{\epsilon}\right) \right \rceil=\mathcal{O}\left(\log(1/{\epsilon})\right),
\end{align}
where $A_1,A_2$ are defined in \Cref{thm:Linearconv_norm},  and the expected number of stochastic gradient evaluations is 
\begin{align}
    \E[\mathcal{W}] = \mathcal{O}\left(\epsilon^{-1}\log(1/\epsilon)\right).
\end{align}
\end{theorem}
\begin{proof}
\Cref{eq:iter_comp_strcnv} directly follows from \cref{eq:exp_termination}. Now, consider the sample complexity at each outer iteration $k$
\begin{align}\label{eq:compperouteriter_strcnv}
\mathcal{W}_k \defeq  \sum_{t=0}^{\rb{T_k-1}}|S_{k,t}| &\leq  \frac{\omega\rb{(1 + \theta_g)^2}}{\theta^2_g\Tilde \tau_k}\ck{T_k},
\end{align}
where the inequality is due to \cref{eq:sktbnd_strcnvx}. 
Therefore, the expected total number of gradient evaluations is found to be
\begin{align} \label{eq:comp1}
    \E \left[ \sum^{\rb{K-1}}_{k = 0} \mathcal{W}_k\right] &\leq \E \left[ \sum^{\rb{K-1}}_{k = 0} \frac{\omega\rb{(1 + \theta_g)^2}}{\theta^2_g\Tilde \tau_k}\ck{T_k}\right] \nonumber \\
    &\leq \sum^{\rb{K-1}}_{k = 0} \frac{\omega \rb{(1 + \theta_g)^2}}{\theta^2_g\Tilde \tau_k}\ck{\E[T_k]} \nonumber \\
    &\leq \sum^{\rb{K-1}}_{k = 0} \frac{4\omega \rb{(1 + \theta_g)^2} \|A\|^2 a^k}{\theta^2_g\tau_{0}\mu^2}\ck{\E[T_k]} 
    \,,
\end{align}
where the last inequality is due to $\Tilde \tau_k = \frac{\mu^2 \tau_k}{4\|A\|^2}$ and $\tau_k = \tau_0(1/a)^k$. 
Using \cref{eq:Tkbnd_strcnvx} and taking the full expectation of both sides, it follows that
\begin{align}
&\E[T_k - 1] \nonumber \\ & \leq \E\left [ \log_{1/\rho_{\L}} \left(\frac{2(L + \alpha \|A\|^2)\left(\alpha^2\|A\|^2 \|c(x_{k-1})\|^2 + \|R(x_{k-1},\lambda_{k-1};\alpha,\eta)\|^2\right)}{\mu \Tilde \tau_k}\right) \right ] \nonumber \\
& = \E\left [ \log_{1/\rho_{\L}} \left(\frac{2(L + \alpha \|A\|^2)\left(\alpha^2\|A\|^2 \|c(x_{k-1})\|^2 + \|\nabla_x\ell(x_{k-1},\lambda_{k};\alpha,\eta)\|^2 \right)}{\mu \Tilde \tau_k}\right) \right ] \nonumber \\
& \leq \log_{1/\rho_{\L}} \left(\frac{2(L + \alpha \|A\|^2)\left(\alpha^2\|A\|^2 \E[\|c(x_{k-1})\|^2] + \E[\|\nabla_x\ell(x_{k-1},\lambda_{k};\alpha,\eta)\|^2]\right)}{\mu \Tilde \tau_k}\right) \nonumber \\
& \leq \log_{1/\rho_{\L}} \left(\frac{2(L + \alpha \|A\|^2)\left(\alpha^2\|A\|^2 A_1 + A_2 \right)\rho^{k-1}}{\mu \Tilde \tau_k}\right) \nonumber \\
& = \log_{1/\rho_{\L}} \left(\frac{8\|A\|^2(L + \alpha \|A\|^2) \left(\alpha^2\|A\|^2 A_1 + A_2 \right) \rho^{k-1} a^k}{\mu^3 \tau_0}\right) = \mathcal{O}(k)
\label{eq:tbnd}
\,,\end{align}
where the second line is due to \cref{eq:optToAuglag}, third line due to Jensen's inequality, fourth line is due to \Cref{thm:Linearconv_norm}, and the last line follows from $\Tilde \tau_k = \frac{\mu^2 \tau_k}{4\|A\|^2}$ and $\tau_k = \tau_0(1/a)^k$. Therefore, \cref{eq:tbnd} shows that there exist $s_1 >0$ and $s_2 >0$ such that
\begin{align} \label{eq:tbnd_new}
\E[T_k] \leq s_1 + s_2 k.
\end{align}
Substituting \cref{eq:tbnd_new} into \cref{eq:comp1}, we have that
\begin{align*}
    \E \left[ \sum^{\rb{K-1}}_{k = 0} \mathcal{W}_k\right] &\leq \sum^{\rb{K-1}}_{k = 0} \frac{4\omega \rb{(1 + \theta_g)^2} \|A\|^2 a^k}{\theta^2_g\tau_{0}\mu^2}( s_1 + s_2 k) \\
    &\leq \sum^{\rb{K-1}}_{k = 0} \frac{4\omega \rb{(1 + \theta_g)^2}\|A\|^2 a^k}{\theta^2_g\tau_{0}\mu^2}( s_1 + s_2 K) \\
    &\leq \frac{4\omega \rb{(1 + \theta_g)^2}\|A\|^2 a^\rb{{K}}}{\theta^2_g\tau_{0}\mu^2(a-1)}( s_1 + s_2 K) \\
    &=\mathcal{O}\left(\epsilon^{-1}\log(1/\epsilon)\right)
    ,
\end{align*}
where the last line is due to \cref{eq:iter_comp_strcnv}.
\end{proof}

\begin{remark}
It is important to emphasize that performing sampling complexity analysis for adaptive sampling methods is quite challenging with present optimization techniques. However, these methods fall under a general class of increasing batch size methods where one can establish theoretical sample complexity analysis that shows stochastic gradient and increasing batch size mechanisms have similar total sample complexity results (see, e.g., \cite{byrd2012sample}). We have established pessimistic (i.e., worst-case) complexity bounds where the sample sizes at each inner iteration are bounded above by the largest sample size employed across all inner iterations at any given outer iteration $k$ (cf.~\cref{eq:sktbnd_strcnvx}). Owing to this pessimistic bound on sample sizes, the overall complexity bound $\mathcal{O}\left(\epsilon^{-1}\log(1/\epsilon)\right)$ is slightly worse than the optimal sample complexity $\mathcal{O}\left(\epsilon^{-1}\right)$ for strongly convex stochastic programming problems \cite{xie2019si,bottou2018optimization}.
\end{remark}

\section{Practical Algorithm} \label{sec:practical}

In this section, we present a complete and practical adaptive sampling augmented Lagrangian (ASAL) algorithm that uses an adaptive sampling proximal gradient method to inexactly solve the augmented Lagrangian subproblems. We describe the mechanism by which the sample size is selected at each inner iteration and the mechanism to terminate the subproblem solver.

The sample size selection and inexactness conditions described in \Cref{subsec: adasample,subsec: tolr} respectively are impractical as they require computing exact variances or deterministic quantities such as $\L(x_k, \lambda_k;\alpha)$ and $R(x_{k},\lambda_{k};\alpha,\eta)$. That being said, these quantities can be approximated using sample variances and sampled stochastic counterparts of the deterministic quantities following the ideas proposed in \cite{bollapragada2018adaptive,xie2020constrained,beiser2020adaptive}.

\smallskip
\paragraph{Sample Size Selection} We propose the following practical sampling test to approximate \Cref{con:Normcondition1proj} where the left-hand-side is the sample variance that approximates the exact variance and the right-hand-side is the stochastic projected (reduced) gradient that approximates the expectation of this quantity. 

\begin{test}[Practical Sampling Test]
\label{test:NormTest1proj}
For any given $\theta_g \geq 0$, the sample size $|S_{k,t}|$ satisfies
	    	\begin{equation}
	\label{eq:NormTest1proj}
		\frac{\tfrac{1}{|S_{k,t}|-1}\sum_{\zeta_i \in S_{k,t}}\|\nabla F(x_{k,t},\zeta_i) - \nabla F_{S_{k,t}}(x_{k,t})\|^2}{|S_{k,t}|}
		\leq
		\theta_g^2 \| R_{S_{k,t}}(x_{k,t},\lambda_{k};\alpha,\eta) \|^2.
	\end{equation}
	\end{test}
In our practical \Cref{alg:ASAL}, we aim to satisfy \Cref{test:NormTest1proj} at each inner iteration using the following procedure. \rev{Whenever \cref{eq:NormTest1proj} is \emph{not} satisfied at the current inner iteration $t$, we attempt to ensure \cref{eq:NormTest1proj} will be satisfied at the next inner iteration $t+1$ by using the relative variance,
\begin{align}\label{eq:nu}
    \nu_t \defeq \frac{\tfrac{1}{|S_{k,t}|-1}\sum_{\zeta_i \in S_{k,t}}\|\nabla F(x_{k,t},\zeta_i) - \nabla F_{S_{k,t}}(x_{k,t})\|^2}{ \theta_g^2 |S_{k,t}| \| R_{S_{k,t}}(x_{k,t},\lambda_{k};\alpha,\eta) \|^2}
    ,
\end{align}
to select the next sample size. More specifically, we set $|S_{k,t+1}| = \lceil \nu_t |S_{k,t}|\rceil$ whenever $\nu_t > 1$.}

\rev{
On the other hand, if \cref{eq:NormTest1proj} \emph{is} satisfied at the current inner iteration $t$ (i.e., $\nu_t \leq 1$), then keeping the sample size unchanged, $|S_{k,t+1}| = |S_{k,t}|$, is a simple rule to maintain control over the sample variance. However, if $\nu_t \ll 1$ is sufficiently small and the current sample size $|S_{k,t}| \gg 1$ is sufficiently large, then it may be beneficial to reduce cost by decreasing the sample size. We explore this possibility by providing an opportunity for the sample size to decrease like $|S_{k,t+1}| = \lceil \nu_t |S_{k,t}|\rceil$ until $|S_{k,t}|$ reaches a minimum value.}\footnote{\rev{Although the sample sizes are allowed to decrease, we do not observe sample size decreases in our numerical experiments; cf.~\Cref{rem:SampleSizeDecrease}.}}
\texttt{Lines 8} through \texttt{16} in \Cref{alg:ASAL} encapsulate the sample size selection procedure. 

\smallskip
\paragraph{Inexactness Conditions}
We propose a practical test to terminate the inner subproblem solver. Owing to the difficulty in computing the optimal quantities $c(x^*_k)$ and $\L(x^*_k,\lambda_k;\alpha)$, and the equivalence of \Cref{cond:tolr-I,cond:tolr-II,cond:tolr-III}, we design the practical test based on \Cref{cond:tolr-III}.  Following a similar procedure employed in approximating the sample size test conditions, we approximate the projected (reduced) gradient with its stochastic counterpart and the optimal constraint violation with the current constraint violation. The resulting practical test is as follows:

\begin{test}[Practical Tolerance Test]\label{test:tolr-III}
	For any given $\Tilde \theta_e \in [0,1)$ and $\Tilde \tau_k \geq 0$ with $\lim_{k \rightarrow \infty} \Tilde \tau_k = 0$,  
	\begin{align}\label{eq:TolrNormtest}
        \| R_{S_{k,t}}(x_{k,t},\lambda_{k};\alpha,\eta) \|^2 \leq \Tilde \theta_e^2 \|c(x_{k,t})\|^2 + \Tilde \tau_k.
    \end{align}
\end{test}
We terminate the inner subproblem whenever \cref{eq:TolrNormtest} is violated. \Cref{alg:ASAL} provides a complete description of the ASAL algorithm.

\begin{algorithm}\caption{Adaptive Sampling Augmented Lagrangian (ASAL) Method}
\label{alg:ASAL}
\textbf{Input:} $x_{-1} \in \R^n$, $\lambda_0 \in \R^m$, step size $\eta>0$, penalty parameter $\alpha > 0$, initial sample size $|S_{0,0}|$, sample size test parameters ($\theta_g > 0, \nu_l \in (0,1), s_l > 0, s_{min} > 0$), inexactness tolerance parameters $(\Tilde \theta_e \in [0,1), \Tilde \tau_k \geq 0)$ \\
\textbf{Initialization:} 
Set $k \gets 0$\begin{algorithmic}[1]
\LOOP
        \STATE Set $t \gets 0$
    \STATE Set $x_{k,0} \gets x_{k-1}$
    \REPEAT
		            \STATE Choose a set $S_{k,t}$ consisting of $|S_{k,t}|$ i.i.d. realizations of $\zeta$
	\STATE Compute $R_{S_{k,t}}(x_{k,t},\lambda_{k};\alpha,\eta)$ via \cref{eq:redgradsk,eq:gradLagsk}         	\STATE Update $x_{k,t+1} \gets x_{k,t} + \eta R_{S_{k,t}}(x_{k,t},\lambda_{k};\alpha,\eta)$\;
    	\IF{\Cref{test:NormTest1proj} is not satisfied}
		        \STATE Set $|S_{k,t+1}| \gets \lceil \nu_t |S_{k,t}|\rceil$
	\ELSE
            \IF{$\nu_t < \nu_l$ and $|S_{k,t}| > s_l$}
                                \STATE Set $|S_{k,t+1}| \gets \max\{s_{min},\lceil \nu_t |S_{k,t}|\rceil \}$
            \ELSE
                                \STATE Set $|S_{k,t+1}| \gets |S_{k,t}|$
            \ENDIF
        \ENDIF
	\STATE Set $t \gets t+1$
    \UNTIL{\Cref{test:tolr-III} is satisfied}        \STATE Set $x_{k} \gets x_{k,t}$
        \STATE Update $\lambda_{k+1} \gets \lambda_k - \alpha c(x_{k})$
    \STATE Set $|S_{k+1,0}| \gets |S_{k,t}|$
    \STATE Set $k \gets k+1$
\ENDLOOP

\end{algorithmic}
\end{algorithm}

\section{Numerical results} \label{sec:numerical_results}
In this section, we study the performance of ASAL (\Cref{alg:ASAL}) using model problems from machine learning (\Cref{sub:logistic_regression_with_disparate_impact_constraints}) and engineering (\Cref{sub:optimal_truss_design,sub:topology_optimization}). We implement \Cref{test:tolr-III} with $\tilde \theta_{e} = 0$ and $\tilde \tau_{k} = \tau_{0}/(k+1)$, treating $\tau_0$ as a hyperparameter for this numerical study.

\subsection{Logistic regression with multiple disparate impact constraints} \label{sub:logistic_regression_with_disparate_impact_constraints}
{We first consider a constrained logistic regression problem.}
A decision-making system suffers from \emph{disparate impact} if it provides outputs that affect a group of people sharing a value of a sensitive feature more frequently than other groups \cite{barocas2016big}.
In \cite[Section~4.4]{zafar2019fairness}, it is shown that disparate impact can be controlled in binary classification problems by applying deterministic constraints.
More explicitly, we consider the optimization problem
\begin{equation}\label{eq:log_reg_prob}
	\begin{alignedat}{3}
		&\mathrm{minimize}~
		         &&\frac{1}{N}\sum_{i=1}^N\big[\log( 1 + \exp(-z_i\langle x, y_i \rangle) )\big] + \frac{\gamma}{2}\|x\|^2
		\\
		&\text{subject to }~
		&&\langle a_1, x \rangle = b_1,\quad
		|\langle a_2, x \rangle| \leq b_2,
	\end{alignedat}
\end{equation}
where $x \in \mathbb{R}^n$ is the optimization variable and $(y_i,z_i) \in \mathbb{R}^n\times\{-1,1\}$ are input/output pairs from a classification data set.
Here, $\gamma > 0$ is a fixed Tikhonov regularization parameter.
Meanwhile, $a_1,\,a_2 \in \mathbb{R}^n$ and $b_1,\,b_2\geq 0$ are constraint parameters.
In \cite{barocas2016big}, it is suggested to take, e.g., $a_1 = \mathbb{E}_{y,s} \big[(s-\mathbb{E}_s[s])y\big]$, where $s$ is a secondary observable, in addition to $y$. However, for the purpose of demonstration, we arbitrarily set $a_1$ and $a_2$ from samples drawn for a standard multivariate normal distribution. Likewise, we set $b_1 = 0.1,b_2 = 0.02$. The initial $x_{-1}$ and $\lambda_0$ variables are chosen to be zero vectors, and we set $\gamma = 1/N$.

In this experiment, we use the \texttt{mushroom} classification data set from the LIBSVM collection \cite{10.1145/1961189.1961199}.
The size of this data set is $N = 8124$, and the dimension of the problem is $n=112$.
In order to evaluate the performance of ASAL, we record the feasibility and stationarity errors \cref{eq:PrimalErrors} until 200 training epochs (i.e., 200$N$ cumulative gradient evaluations) have elapsed.
We then compare ASAL to three separately-tuned fixed-batch-size implementations of ASAL using $10\%,~20\%,$ and $50\%$ of the data set size at each iteration, respectively.
In this experiment we use $\theta_g = 0.99$, $\nu_{\mathrm{l}} = 0.5$, $s_{\mathrm{l}} = 0.1N$, and $s_{\min} = 0.1N$.
The value of $\theta_g$ is \emph{not tuned} and is, instead, set at an arbitrary value close to the suggestion for unconstrained problems in \cite{carlon2020multi,espath2021equivalence}.
The values of the other three fixed hyperparameters are also set arbitrarily.
Yet, they appear to have little to no effect on performance; cf.~\Cref{rem:SampleSizeDecrease}.

We treat $\tau_0$, $\alpha$ and the step size $\eta$ as tunable hyperparameters.
All of the hyperparameters are tuned using the following procedure:
We run each augmented Lagrangian algorithm for all possible combinations of $\tau_0 = 10^{4},\, 10^{3},\, 10^{2},\, 10^{1},\, 10^0,\, 10^{-1}$, $\eta = 10^{-1},\, 10^{-2},\, 10^{-3},\, 10^{-4},\, 10^{-5},\, 10^{-6}$, and $\alpha = 10^2,\,10^1,\,10^0,\, 10^{-1},\, 10^{-2}$.
Then, for each algorithm, we select the run with the smallest average objective function value in the final 5 inner iterations among all runs whose minimum feasibility error in the final 30 inner iterations is less than feasibility tolerance $10^{-4}$.

The stationarity and feasibility errors corresponding to the best hyperparameters for each algorithm are overlaid in \Cref{fig:log_reg_plots}.
Because the hyperparameter tuning procedure we have used seeks the best stationarity error among runs reaching a feasibility error threshold, it is no surprise that ASAL and each of the three baseline algorithms achieve a similar \emph{minimal} feasibility error (around feasibility tolerance $10^{-4}$).
Nevertheless, we observe that ASAL outperforms the three baseline algorithms with respect to stationarity error. We also present similar results for \texttt{australian} data set from the LIBSVM collection \cite{10.1145/1961189.1961199} in \Cref{sec:appendix}.

\begin{figure}
\centering
\includegraphics[height=3.05cm]{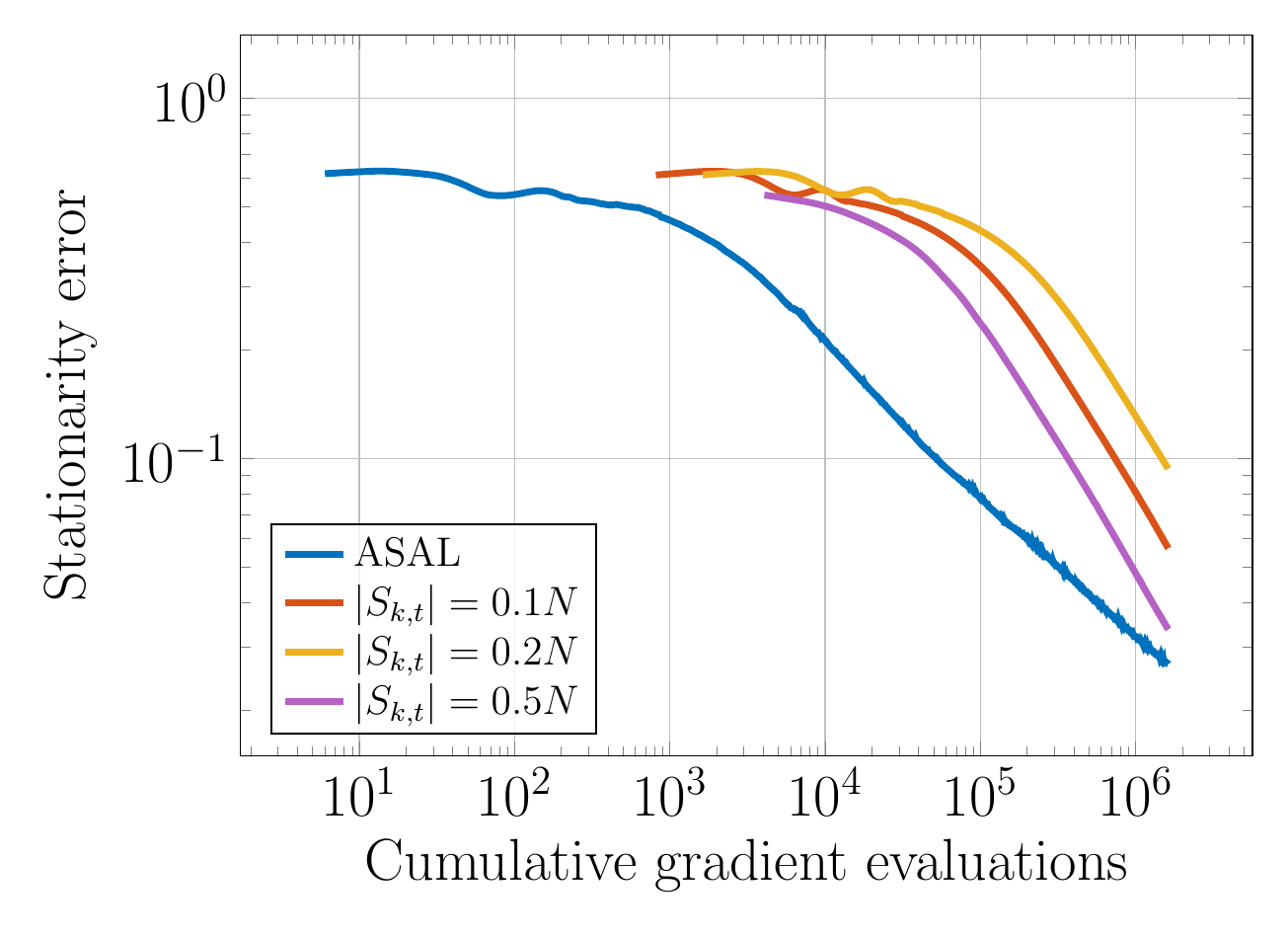}
\includegraphics[height=3.05cm]{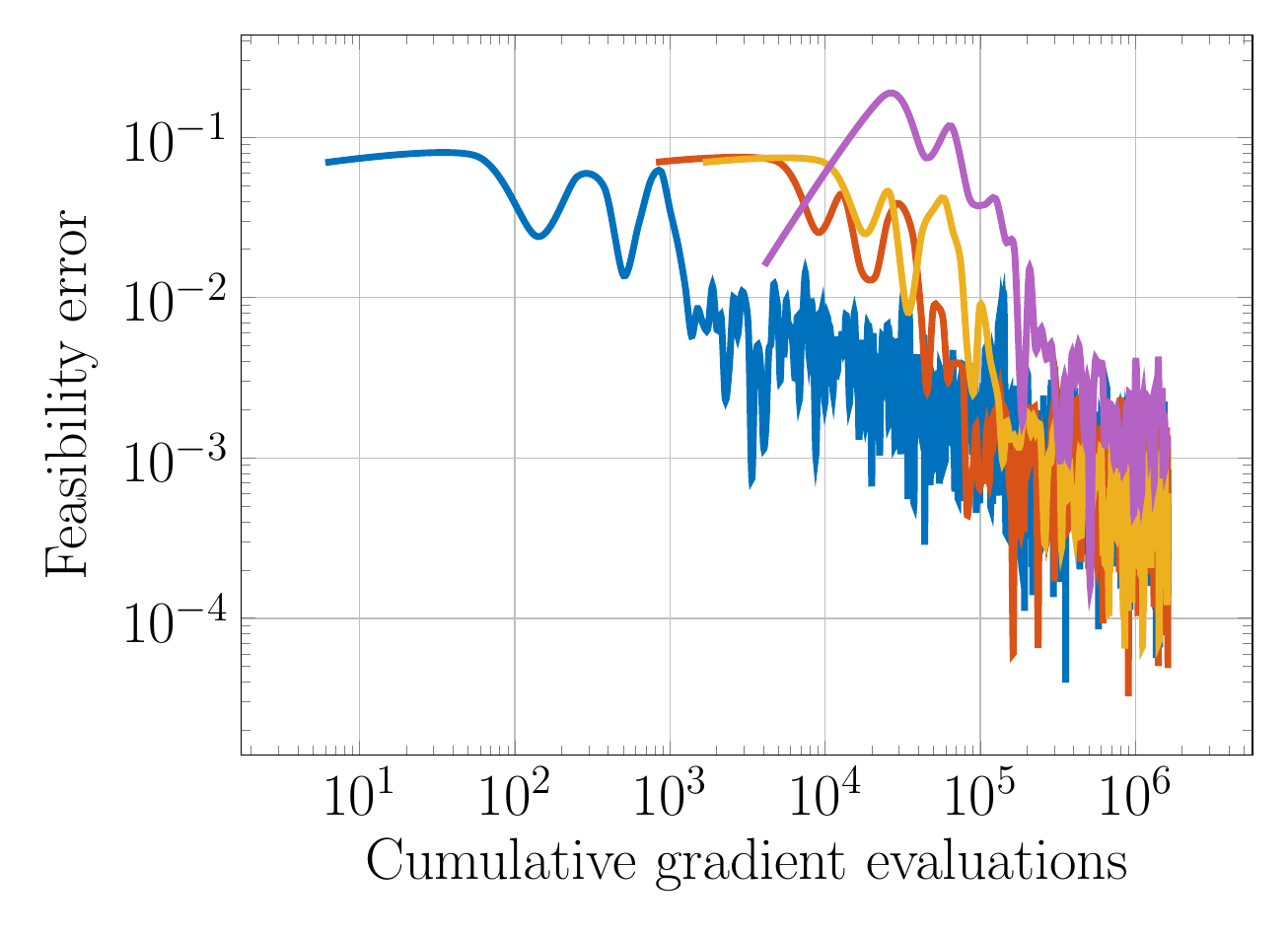}
\includegraphics[height=3.05cm]{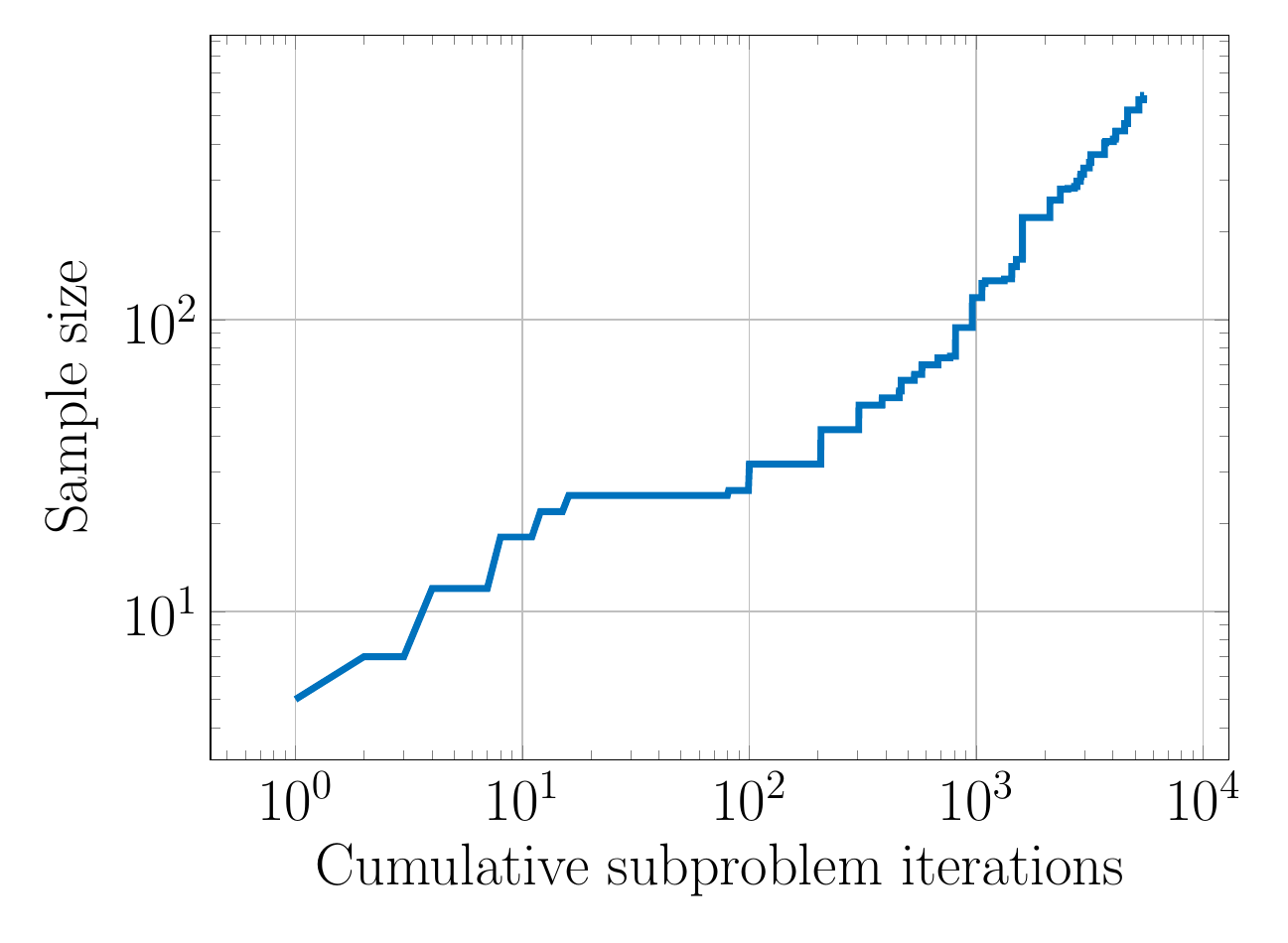}
\caption{
    Results of running \Cref{alg:ASAL} on the constrained logistic regression problem~\cref{eq:log_reg_prob} with the \texttt{mushroom} classification data set.
    Notice that ASAL achieves the lowest average stationarity errors while matching the minimal feasibility error of the three baseline algorithms.
    To generate these results, we used the algorithm parameters  $\theta_g= 0.99$ and individually tuned $\alpha$, $\eta$, and $\tau_0$.
\label{fig:log_reg_plots}}
\end{figure}

\begin{remark}
\label{rem:SampleSizeDecrease}
Notice from \Cref{fig:log_reg_plots} that the ASAL sample size never decreases.
This is despite the safeguarding mechanism in \texttt{line 17} of \Cref{alg:ASAL}.
We have witnessed this non-decreasing sample size property in all of our experiments with ASAL after tuning the hyperparameters $\alpha$, $\eta$, and $\tau_0$.
Thus, we see little justification for allowing sample size decreases in future implementations of ASAL and do not report the hyperparameters $\nu_{\mathrm{l}}$, $s_{\mathrm{l}}$, and $s_{\min}$ in the remaining experiments.
\end{remark}

\begin{remark}
    The starting values for the cumulative gradient evaluations in \Cref{fig:log_reg_plots} represent the fact that we are recording errors only \emph{after} advancing a single optimization step.
    Each algorithm began with the same initial guesses $x_{-1}$ and $\lambda_0$.
\end{remark}

\begin{remark}
    Observe that the expected feasibility error with ASAL steadily decreases in \Cref{fig:log_reg_plots}.
    Meanwhile, the feasibility error in each of the other algorithms plateaus after around $5\times 10^5$ cumulative gradient evaluations.
    This is due to the stable sample size growth provided by our adaptive sampling strategy and the fact that a fixed number of samples are used for each of the baselines; i.e., the baseline algorithms can only converge in expectation to a neighborhood of the solution.
    As a result, even though the slopes of the stationarity errors for the baseline algorithms are higher than ASAL after 200 epochs, we conclude that ASAL would remain the better practical algorithm even if a larger epoch threshold had been used.
\end{remark}

\begin{remark}
The tuning procedure used in this experiment is expensive and impractical for more expensive problems.
Owing to this fact, in the remaining sections, we only compare ASAL to baseline algorithms with a shared set of hyperparameters.
\end{remark}

\subsection{Optimal truss design} \label{sub:optimal_truss_design}

We consider optimizing the simply supported truss structure shown in \Cref{fig:truss} in a problem inspired by an example presented in \cite{rockafellar2010buffered}. The truss elements are numbered as shown in \autoref{fig:truss}, and a random force $F$, pointing downwards, is applied in the middle of the bottom chord. The cross-sections of the truss elements are denoted by $x^i,\,\,i=1,2,\dots,7$, and the yield stress associated with the members with $\sigma_i,\,\,i=1,2,\dots,7$. The first two yield stress limits $\sigma_i,\,\, i=1,2$, are log-normal random variables with mean $100\rm{N}/\rm{mm}^2$ and standard deviation $20\rm{N}/\rm{mm}^2$. The yield stresses for all other members are also log-normal, but with mean $200\rm{N}/\rm{mm}^2$ and standard deviation $40\rm{N}/\rm{mm}^2$. The correlation coefficient between $\sigma_1$ and $\sigma_2$ is 0.8, and between $\sigma_{i}$, $i=1,2$, and $\sigma_j,\,\,j=3,4,5,6,7$, the correlation coefficients are each 0.5. The correlation coefficients between each $\sigma_i\neq\sigma_j$, $i,j\in\left\{3,4,5,6,7\right\}$, are set to 0.8. The applied force $f$ is independent of the yield stresses and is distributed log-normally with mean $1000\rm{kN}$ and standard deviation $400\rm{kN}$. The structure will fail if any member exceeds the associated yield stress, i.e., for each member, we can define the following random limit state function:
\begin{equation}
	g_i(x;f,\bm\sigma)
	=
	\frac{f}{c_i x^i}-\sigma_{i},
	\quad i\in\{1,2,\ldots,7\},
	\end{equation}
where the fixed parameters $c_i$ depend on the geometry and the loads. For this structure, $c_{\{i=1,2\}}=1/\left(2\sqrt{3}\right)$ and $c_{\{i=3,\dots,7\}}=1/\sqrt{3}$.

\begin{figure}
\centering
\includegraphics[height=5cm]{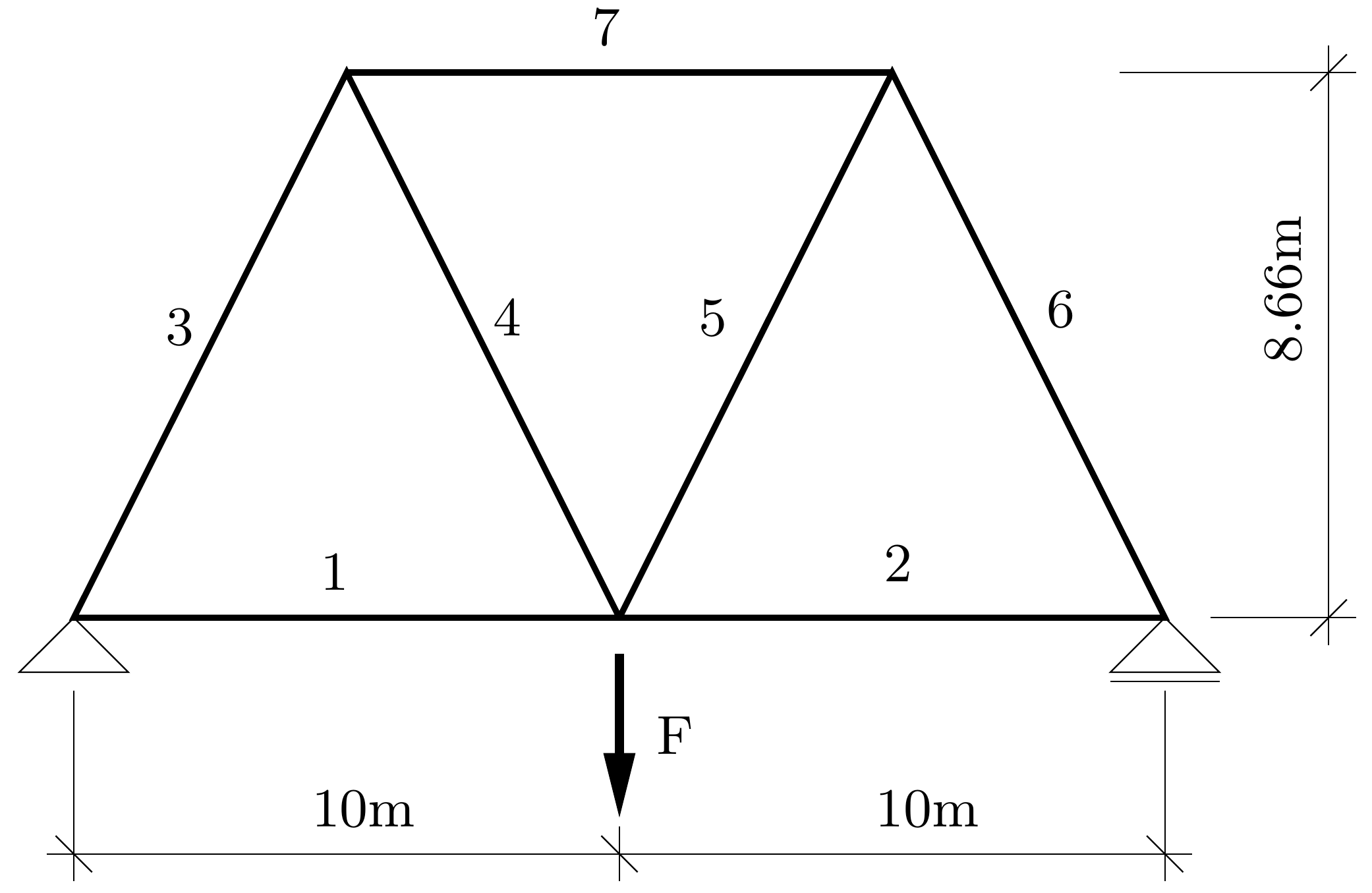}
\caption{Definition of the geometry and the load applied to the truss.}
\label{fig:truss}
\end{figure}

We pose the following stochastic optimization problem:
\begin{equation}
	\begin{alignedat}{3}
		&\mathrm{minimize}\quad
				&&
		\frac{1}{7\alpha}
		\mathbb{E}
				\bigg[
			\ln\bigg(\sum_{i=1}^7 \exp(\alpha g_i(x))\bigg)
		\bigg]
		,
		\\
		&\text{subject to }~
		&&A\leq x\leq B,\quad \langle 1, x \rangle \leq C,
			\end{alignedat}
	\label{eq:truss}
\end{equation}
where $\alpha = 1$, $A = 1\times 10^4 \rm{mm}^2$, $B = 5\times 10^4 \rm{mm}^2$, and $C = 15\times 10^4\rm{mm}^2$ are user-defined parameters.
The components of the optimal solution are estimated to be
\begin{equation}
    x_{\{i=1,2\}} = 4.342 \times 10^4 \rm{mm}^2
    \quad
    \text{and}
    \quad
    x_{\{i=3,\dots,7\}} = 1.263 \times 10^4 \rm{mm}^2
    \,.
\end{equation}
To solve this problem, we use ASAL with $\theta_g = 0.99$ and compare its performance to the stochastic augmented Lagrangian method with fixed sample sizes under a 1 million cumulative sample budget.
In each experiment, we use the penalty and step size values $\alpha = 0.01$ and $\eta = 1.0$.
\Cref{fig:truss_plots} documents our findings.
Notice that, even though it used less than 25\% of the total iterations, the stationarity and feasibility errors from ASAL (248 iterations) are significantly lower after the sample budget expires than the best-performing fixed sample size algorithm (1000 iterations).

\begin{figure}
\centering
\includegraphics[height=3.05cm]{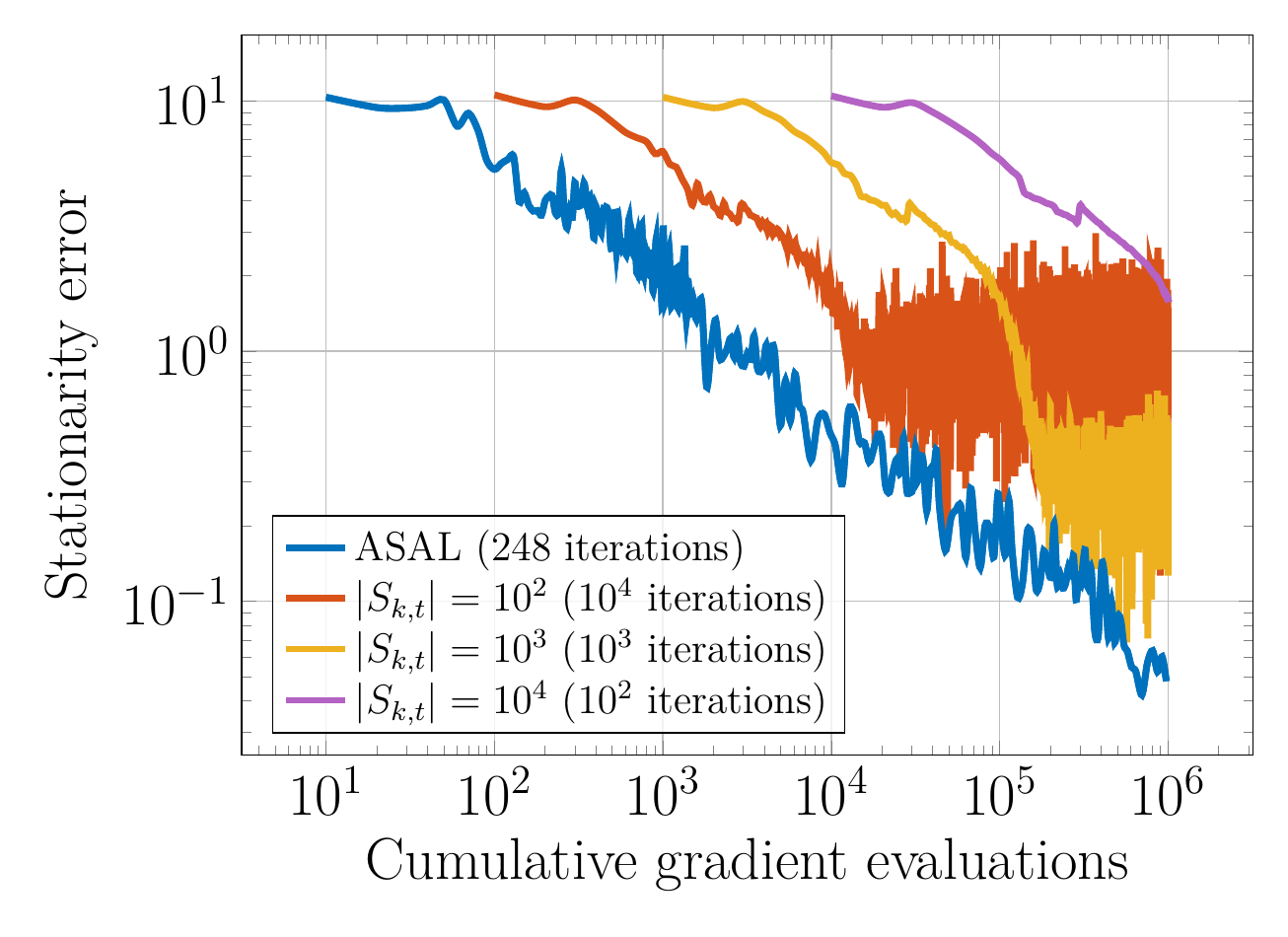}
\includegraphics[height=3.05cm]{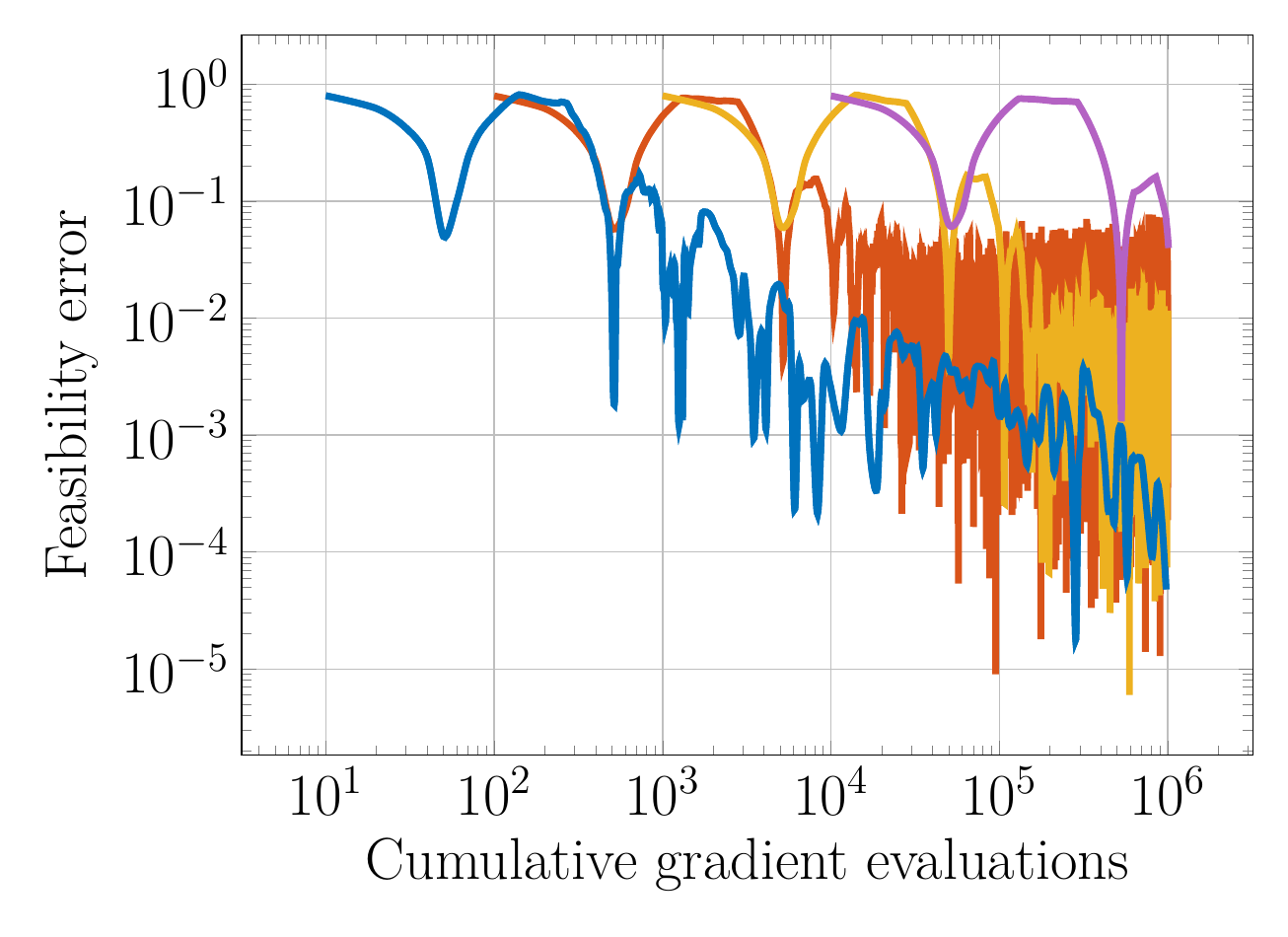}
\includegraphics[height=3.05cm]{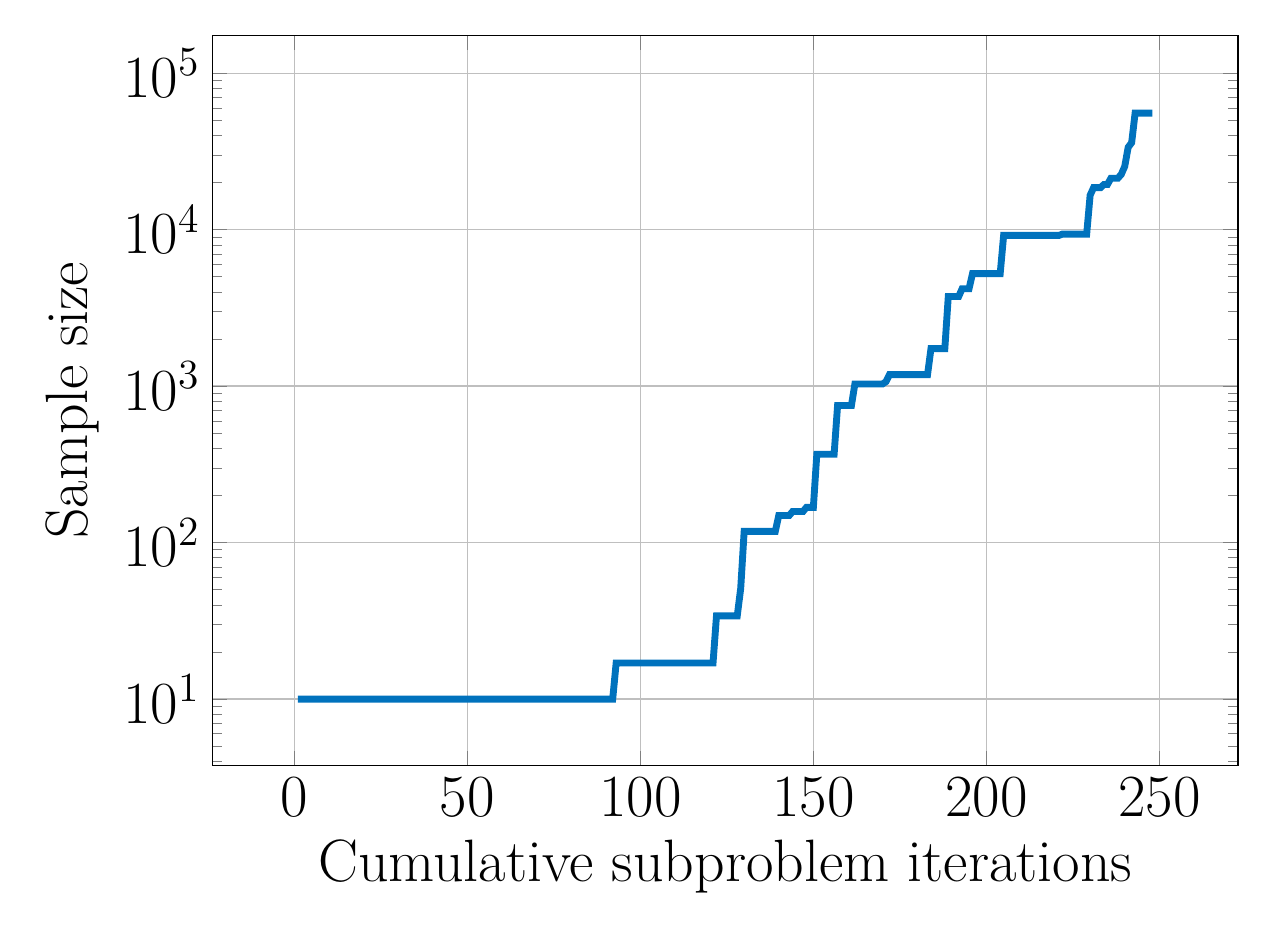}
\caption{
	Results of running \Cref{alg:ASAL} on the truss optimization problem~\cref{eq:truss}.
	Notice that ASAL achieves the lowest average stationary and feasibility errors while simultaneously requiring the smallest number of iterations.
	To generate these results, we used the algorithm parameters $\theta_g = 0.99$, $\alpha = 0.01$, $\eta = 1.0$, and primal tolerance sequence $\tau_k = 10/k$.
	\label{fig:truss_plots}}
\end{figure}

\subsection{Optimal design of a heat sink} \label{sub:topology_optimization}

We close with a non-convex optimization problem of engineering interest.
In this final experiment, we consider the optimal design of a heat sink within a hypothetical square domain $\Omega = (0,1)^2$ with a stochastic heat source $f(\bm{x})$, $\bm{x} \in \Omega$, described by a spatial Gaussian random field with M\'atern covariance.
More specifically, we follow \cite{lindgren2011explicit,lindgren2022spde} and define
\begin{equation}
\label{eq:random_heat_source}
    -\kappa^2\Delta f + f = \mathcal{W}
    ~~\text{in }\Omega
    \,,
    \quad
    \nabla f \cdot \bm{n} = 0
    ~~\text{on }\partial\Omega
    \,,
\end{equation}
where $\mathcal{W}$ is spatial additive white Gaussian noise, $\kappa > 0$ is a correlation length parameter, and $\bm{n}$ denotes the outward-facing unit normal vector field on $\partial\Omega$.
M\'atern random fields can be used to model various random spatial phenomena \cite{khristenko2020statistical,keith2021fractional,lindgren2022spde}, which makes them reasonable for modeling the heat source in this example.
\Cref{fig:random_loads} depicts three representative solutions to~\cref{eq:random_heat_source} for the reader's interest.

\begin{figure}
\centering
    \begin{minipage}[c]{0.30\linewidth}
      \includegraphics[width=\linewidth]{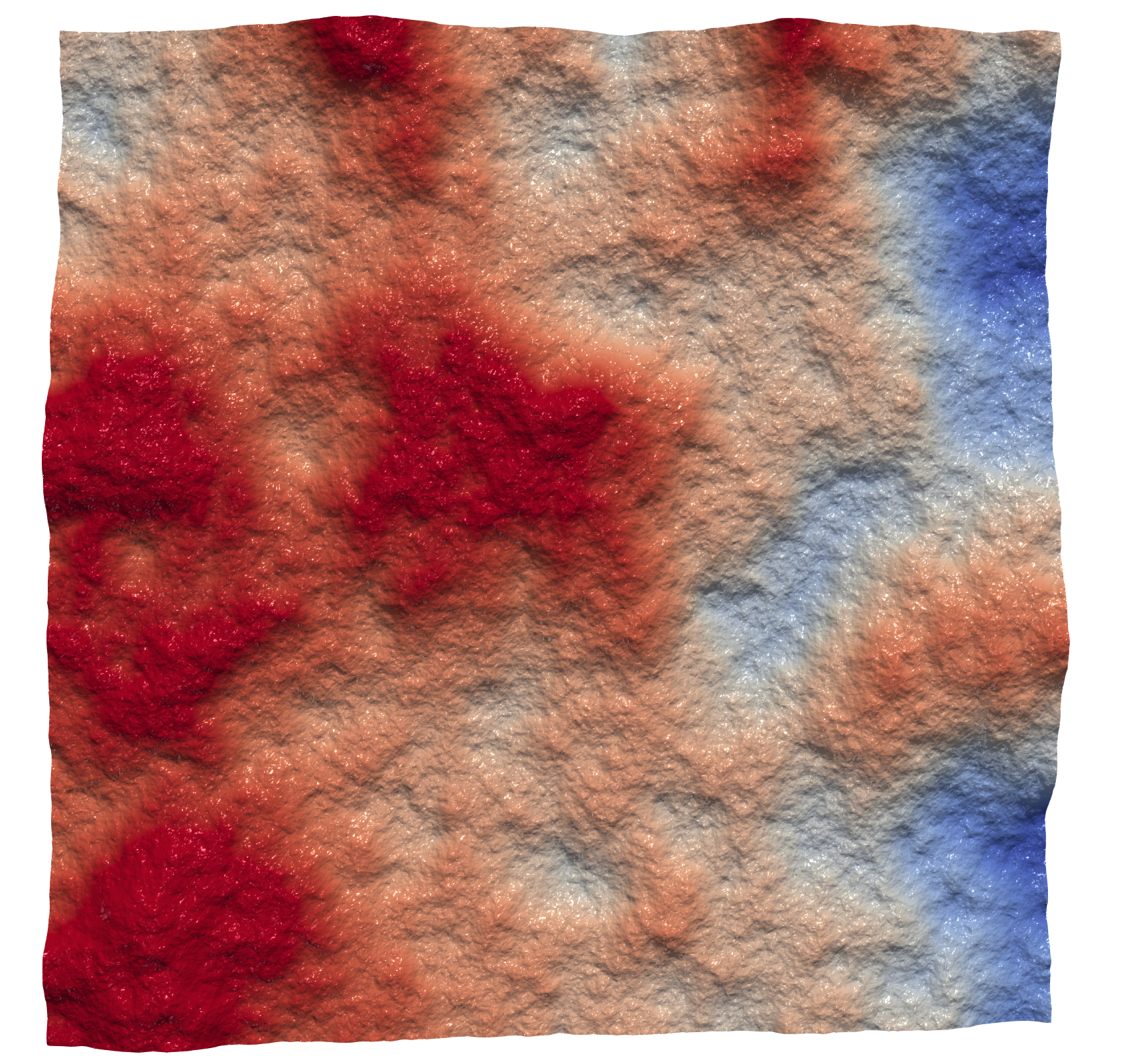}
    \end{minipage}
    \begin{minipage}[c]{0.30\linewidth}
      \includegraphics[width=\linewidth]{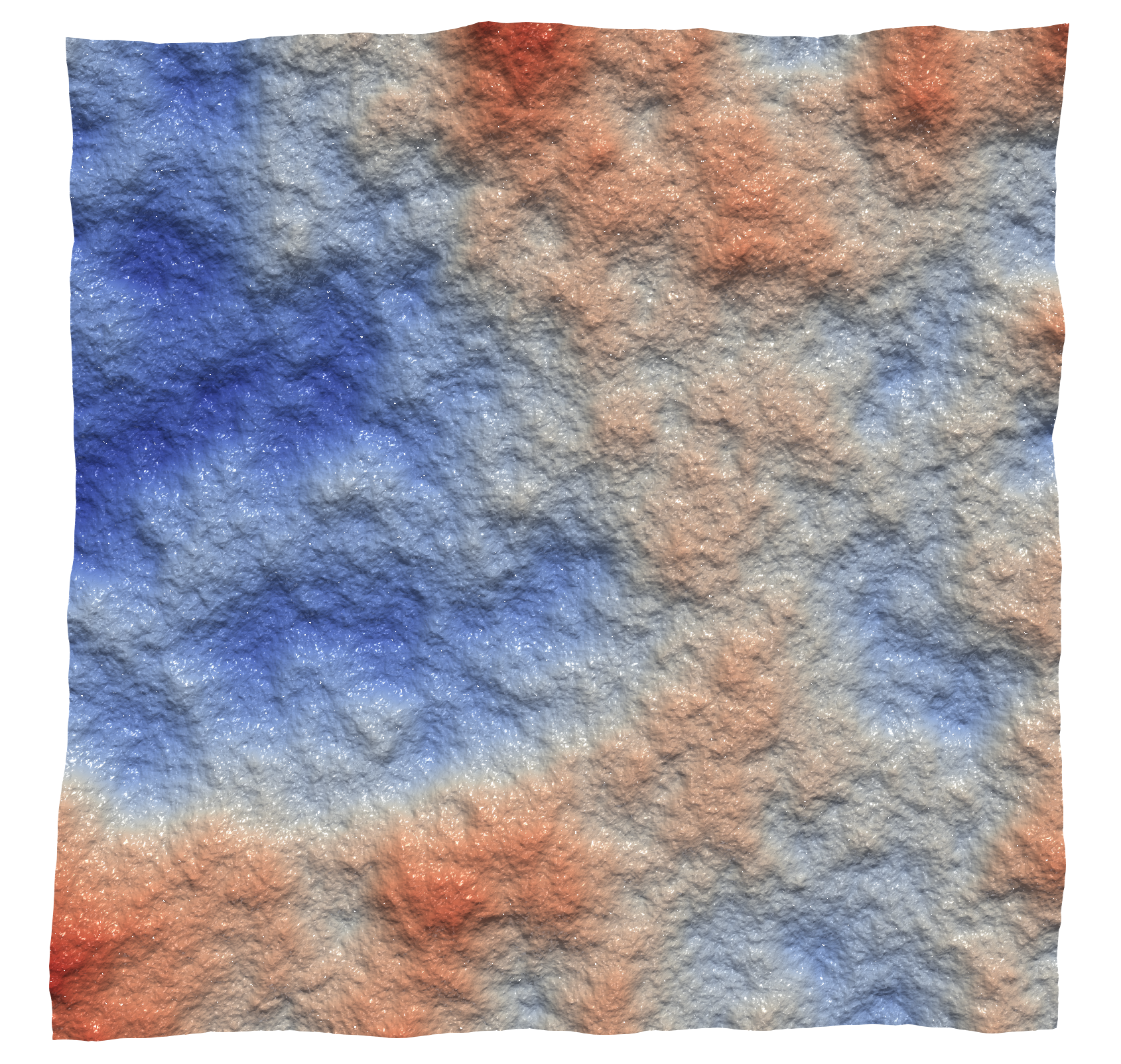}
    \end{minipage}
    \begin{minipage}[c]{0.30\linewidth}
      \includegraphics[width=\linewidth]{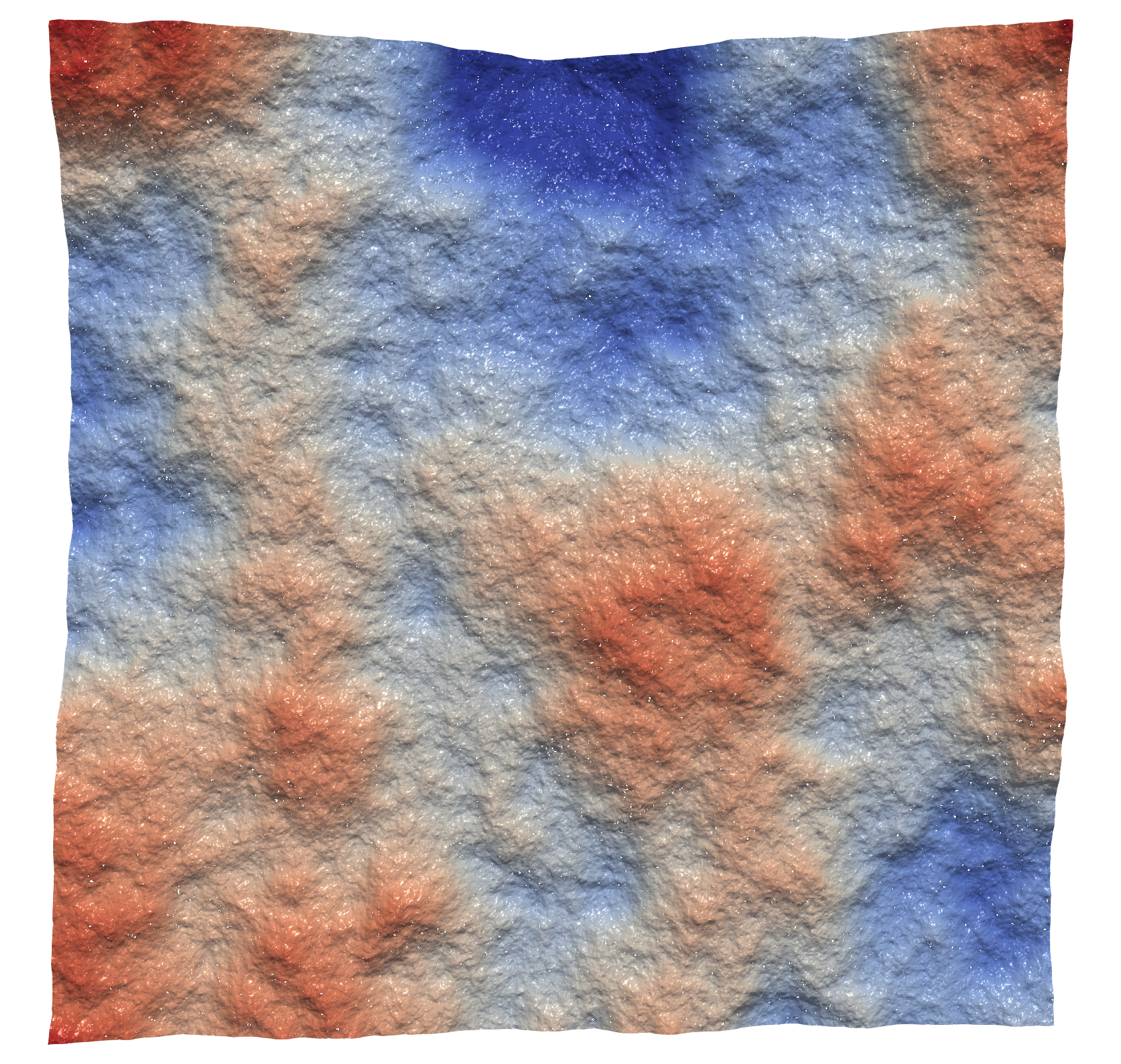}
    \end{minipage}
    \,
    \begin{minipage}[c]{0.052\linewidth}
      \includegraphics[width=\linewidth]{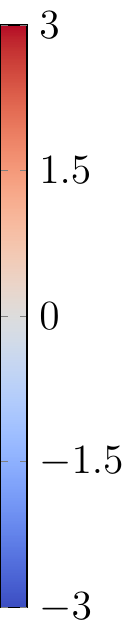}
    \end{minipage}
  \caption{Three independent realizations of the Gaussian random field $f(\bm{x})$ generated by solving~\cref{eq:random_heat_source} on the square domain $\Omega = (0,1)^2$.
  \label{fig:random_loads}
  }
\end{figure}

We use the two-field filtered density approach to topology optimization \cite[Section~3.1.2]{sigmund2013topology} to formulate the optimal heat sink design problem.
The goal is to find a material distribution $0 \leq \rho \leq 1$, where zero indicates no material, and one indicates the complete presence of material, that induces the smallest thermal compliance, $\int_{ \Omega} u f \dd \bm{x}$, in expectation.
In the aforestated expression, the temperature distribution $u$ is determined by $\rho$ and $f$ through the heat diffusion equation $-\div r(\widetilde{\rho}) \nabla u = f$, where $\widetilde{\rho}$ is a regularized (filtered) distribution function \cite{bruns2001topology,lazarov2011filters} and $r(\widetilde{\rho}) > 0$ is a thermal conductivity model.
In this work, we use the well-known (modified) solid isotropic material penalization (SIMP) model $r(\widetilde{\rho}) = \rho_{\_} + \widetilde{\rho}^3 (1-\rho_{\_})$, where $0<\rho_{\_} \ll 1$ is a nominal thermal diffusivity constant assigned to void regions in order to prevent the stiffness matrix from becoming singular \cite{andreassen2011efficient}.

\begin{subequations}
\label{eqs:thermal_compliance}
The full problem formulation is written as follows:
\begin{equation}
\label{eq:thermal_compliance_objective}
  \min_{\rho \in L^2(\Omega),\, u \in H^1(\Omega)} \ \
  \bigg\{
  \,
  \widehat{F}(\rho,u)
  :=
  \mathbb{E} \left[ \int_{ \Omega} u  f \dd \bm{x} \right]
  \,
  \bigg\}
  \, ,
\end{equation}
subject to the constraints
\begin{equation}
\label{eq:thermal_compliance_constraints}
\left\{\,\,
\begin{gathered}
  -\epsilon^2\Delta \widetilde{\rho} + \widetilde{\rho}
  = \rho
  ~~\text{in }\Omega
  \,,~~
  \nabla \widetilde{\rho}\cdot \bm{n} = 0
  ~~\text{on }\partial\Omega
  \,,
  \\
  -\div \bigl( r(\widetilde{\rho}) \nabla u \big)
  = f
  ~~\text{in }\Omega
  \,,~~
  u = 0
  ~~\text{on }\Gamma_0
  \,,~~
    \nabla u\cdot \bm{n} = 0
  ~~\text{on }\partial\Omega \setminus \Gamma_0
  \,,
  \\
  \int_\Omega \rho(\bm{x}) d\bm{x}
  \leq \gamma |\Omega|
  \,,
  ~~\text{and}~~
  0 \leq \rho
  \leq 1
  ~~\text{in }\Omega
        \,,
\end{gathered}
\right.
\end{equation}
\end{subequations}
where $0 < \gamma < 1$ is the \emph{volume fraction}, which constrains the fraction of the domain occupied by design, and $\epsilon > 0$ is a \emph{length scale} for the final design.
The boundary conditions and solution to the optimization problem~\cref{eqs:thermal_compliance} with $\rho_{\_} = 10^{-3}$, $\gamma = 0.5$, $\epsilon = 0.01$, $\kappa = 0.2$ are depicted in \Cref{fig:thermal_compliance}.

\ckrev{
To remove the PDE constraints from the optimization problem, we employ a reduced space formulation, often referred to in the literature as a nested formulation \cite{bendsoe2003topology}, which can be written as
\begin{subequations}
\begin{equation}
\label{eq:thermal_compliance_objective_reduced}
  \min_{\rho \in L^2(\Omega)} \ \
  \bigg\{
  \,
  F(\rho)
  :=
  \mathbb{E} \left[ \int_{ \Omega} u\left(\widetilde{\rho}\left(\rho\right)\right)  f \dd \bm{x} \right]
  \,
  \bigg\}
  \, ,
\end{equation}
subject to the constraints
\begin{equation}
\begin{gathered}
\label{eq:thermal_compliance_convex_constraints}
    \int_\Omega \rho(\bm{x}) d\bm{x}
  \leq \gamma |\Omega|
  \,,
  ~~\text{and}~~
  0 \leq \rho
  \leq 1
  ~~\text{in }\Omega
        \,.
\end{gathered}
\end{equation}
\end{subequations}
In this formulation, it is understood that the temperature field $u = u\left(\widetilde{\rho}\left(\rho\right)\right)$ solves the state equation
\begin{equation}
\begin{gathered}
    -\div \bigl( r(\widetilde{\rho}) \nabla u \big)
  = f
  ~~\text{in }\Omega
  \,,~~
  u = 0
  ~~\text{on }\Gamma_0
  \,,~~
  \nabla u\cdot \bm{n} = 0
  ~~\text{on }\partial\Omega \setminus \Gamma_0
  \,,
  \end{gathered}
\end{equation}
and the filtered density $\widetilde{\rho} = \widetilde{\rho}\left(\rho\right)$ solves the screened Poisson equation,
\begin{equation}
    -\epsilon^2\Delta \widetilde{\rho} + \widetilde{\rho}
  = \rho
  ~~\text{in }\Omega
  \,,~~
  \nabla \widetilde{\rho}\cdot \bm{n} = 0
  ~~\text{on }\partial\Omega
  \,.
\end{equation}
Since the inequality constraint in~\cref{eq:thermal_compliance_convex_constraints} is always active, it is replaced by an equality constraint that our ASAL algorithm can handle. The gradients of the reduced objective function in~\cref{eq:thermal_compliance_objective_reduced} are computed with FEM-discretized representations of the temperature $u$ and filtered density $\widetilde{\rho}$ using standard adjoint analysis techniques \cite{bendsoe2003topology}.
Finally, $L^2(\Omega)$ projections are used to enforce the box constraints found in~\cref{eq:thermal_compliance_convex_constraints}.
}

For comparison, \Cref{fig:thermal_compliance} also depicts a reference solution to~\cref{eqs:thermal_compliance} corresponding to the ({deterministic}) uniform heat field $f \equiv 1$. Close examination reveals significant differences between the designs with deterministic and stochastic inputs. The deterministic case results in an organic tree-like structure that aims to transfer the heat generated at any point in the computational domain using the shortest possible way to the Dirichlet boundary with zero temperature. The design does not depend on the magnitude of the heat source, and any constant input will result in the same material distribution if the initial material distribution is in the vicinity of the local solution. 
On the other hand, due to the oscillatory nature of the stochastic input, the heat source term can take positive and negative values. Such input distribution allows the optimization process to balance the heat transfer locally without linking the local subdomain directly to the boundary with a fixed temperature. Thus, the role of the closed loops of material appearing in the design with stochastic input is to establish a local heat equilibrium. In this case, the global tree-like structure transfers only the excess heat, which cannot be balanced locally. 

\begin{figure}
\centering
\raisebox{3.75pt}[0pt][0pt]{  \includegraphics[width=0.28\linewidth]{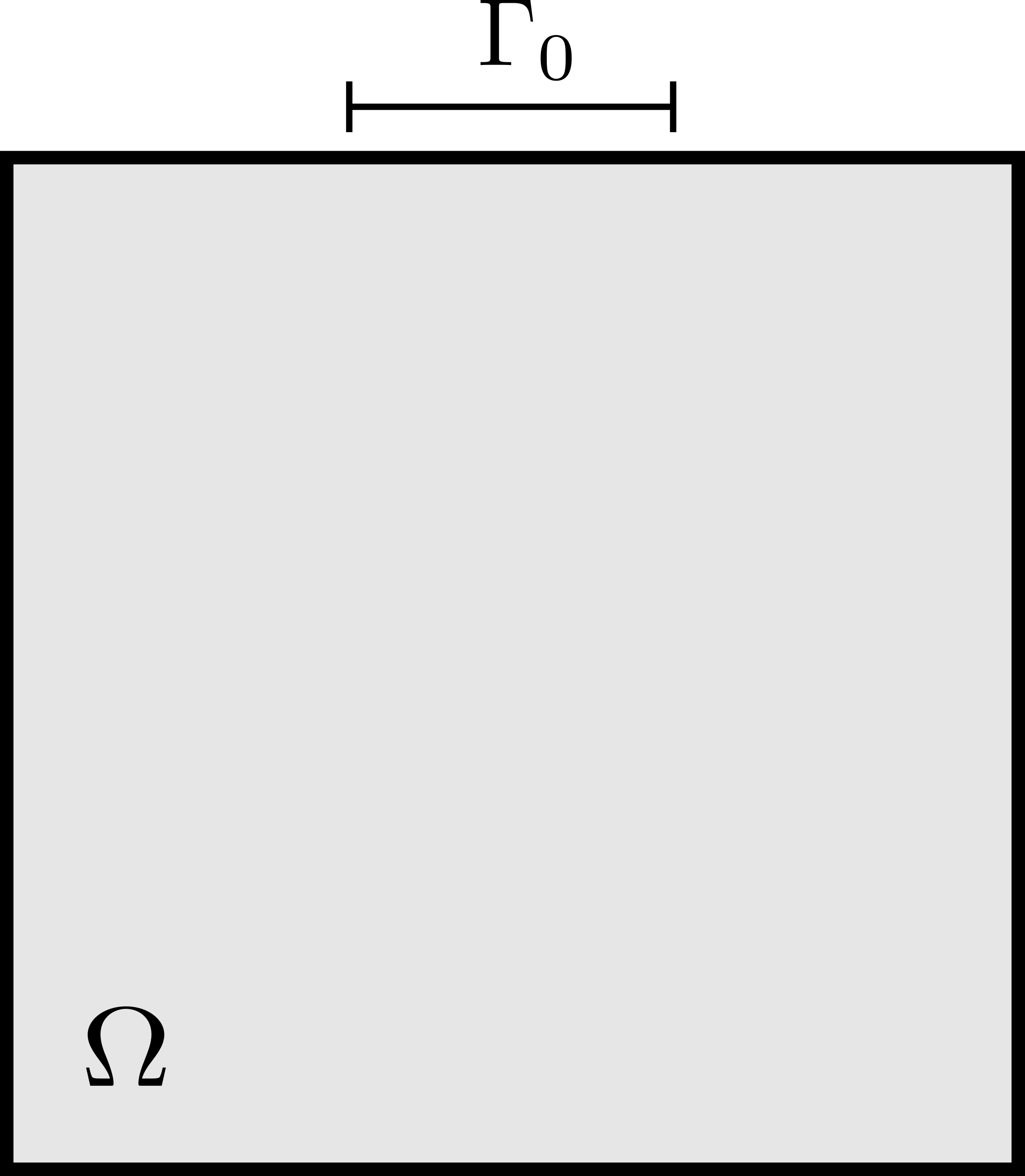}
}
  \quad
  \includegraphics[width=0.3\linewidth]{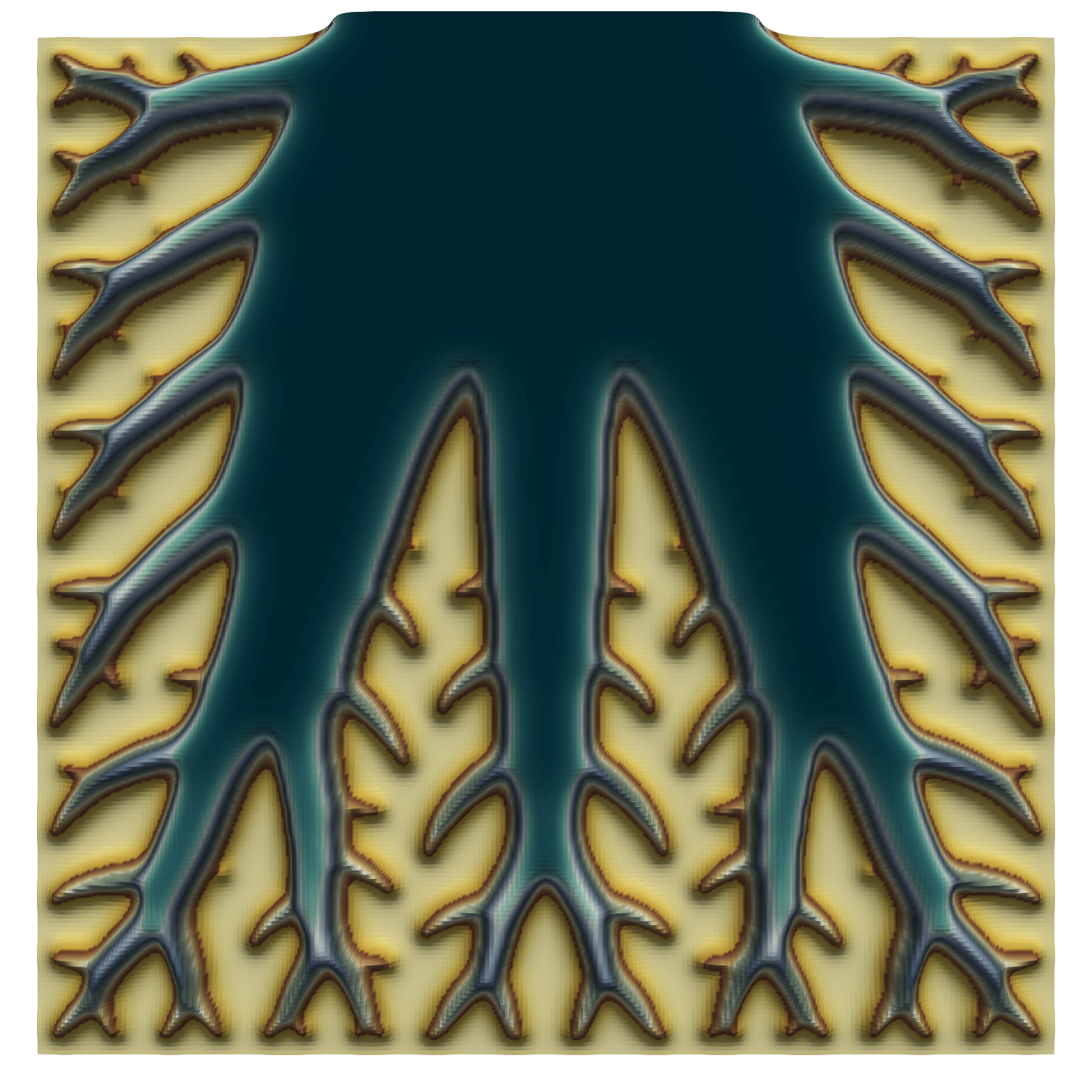}
    \quad
  \includegraphics[width=0.3\linewidth]{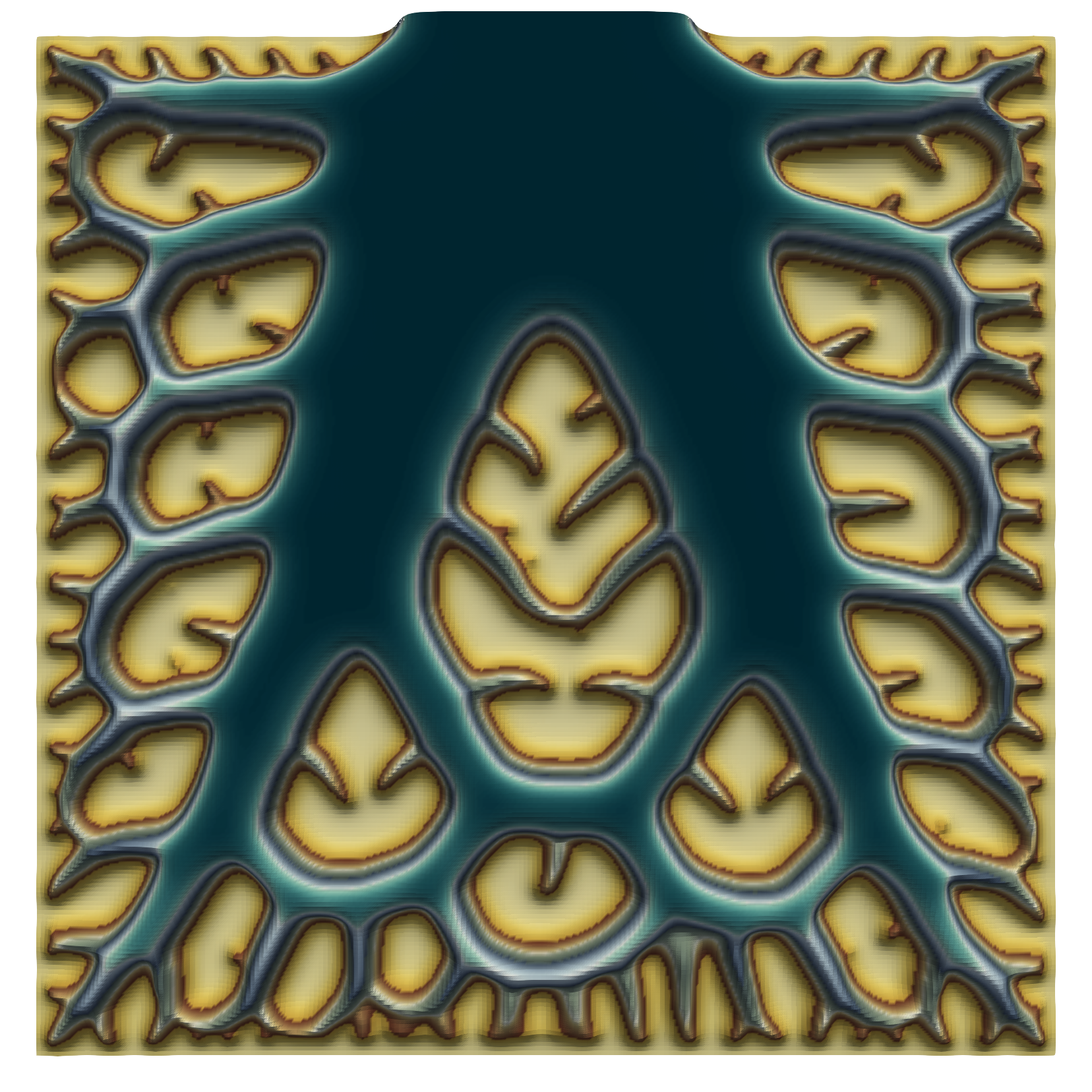}
  \caption{Left: Depiction of the subsets of the domain boundary $\partial\Omega$ for the boundary conditions for the heat field $u$ in~\cref{eq:random_heat_source}. We define $u = 0$ on $\Gamma_0$ and $\nabla u\cdot \bm{n} = 0$ on $\partial\Omega \setminus \Gamma_0$.
  Middle: Reference density field $\tilde{\rho}$ for the solution of the \emph{deterministic} thermal compliance optimization problem~\cref{eqs:thermal_compliance} with $f \equiv 1$ everywhere in $\Omega$.
  Right: The filtered density $\widetilde{\rho}$ for the solution of the \emph{expected value} thermal compliance optimization problem~\cref{eqs:thermal_compliance} with $f$ given by~\cref{eq:random_heat_source}.
  The presence of closed-loop branches in the optimal solution on the right indicates a preference for balancing the heat locally and transferring only the excess unbalanced heat to the external environment through $\Gamma_0$.
  \label{fig:thermal_compliance}
  }
\end{figure}

In this experiment, we use ASAL with $\theta_g = 2$ and compare its performance to stochastic augmented Lagrangian with fixed sample sizes, $|S_{k,t}| = 10^j$, $j=1,2,3$, under a $10^5$ cumulative sample budget.
In each execution, we use the step size values $\alpha = 0.1$ and $\eta = 2.0$.
\Cref{fig:topology_plots} documents our findings.
ASAL achieves the lowest combined average stationary and feasibility errors while requiring less than 20\% of the iterations of the best-performing fixed sample size run ($|S_{k,t}| = 10^2$).
Although the average feasibility errors with ASAL and this fixed sample size run are similar, the variance of the fixed sample size run is much greater.
Finally, the average stationarity error for the best-performing fixed sample size run is significantly larger than the average stationarity error with ASAL.

\begin{figure}
\centering
\includegraphics[height=3.05cm]{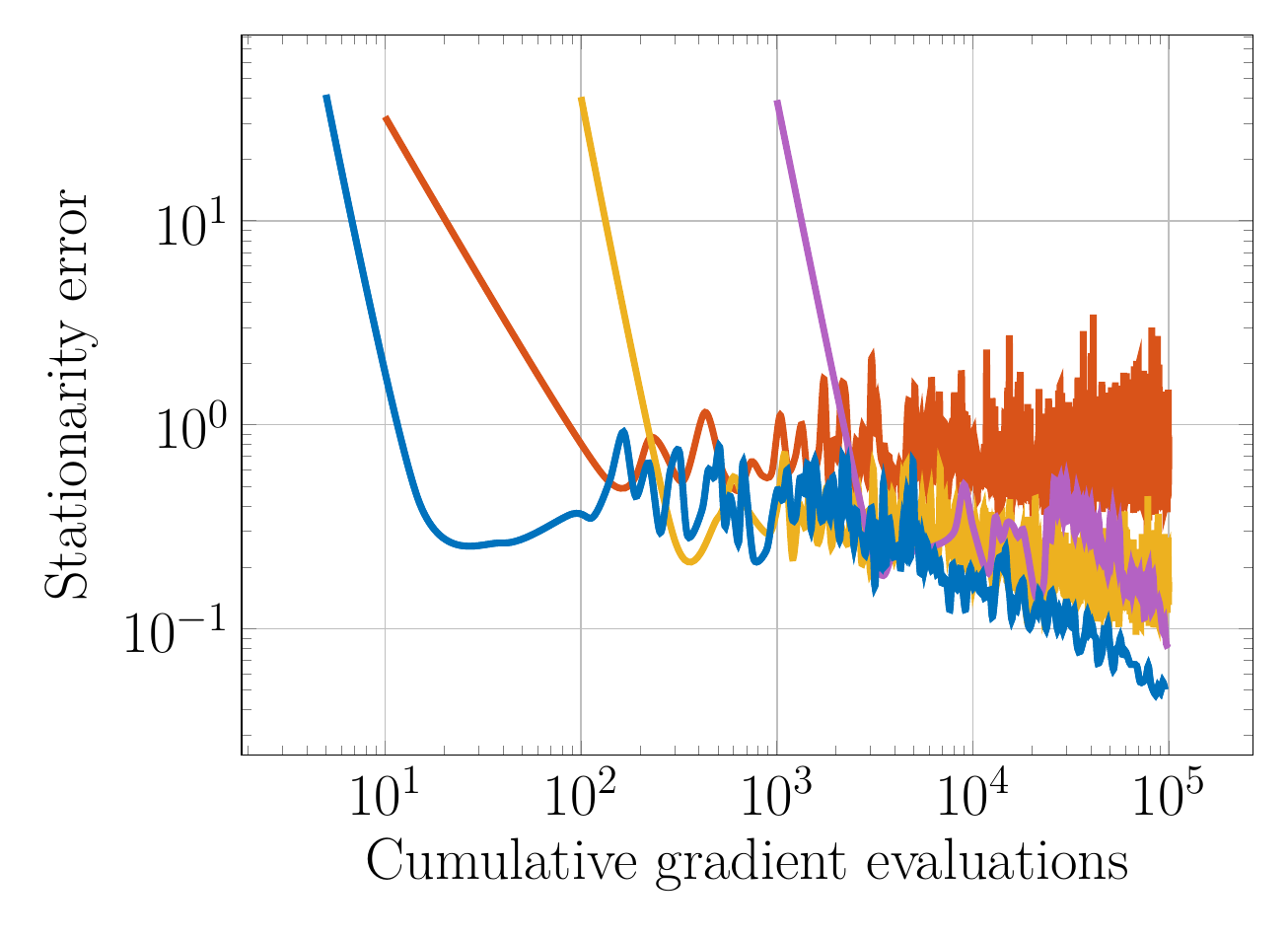}
\includegraphics[height=3.05cm]{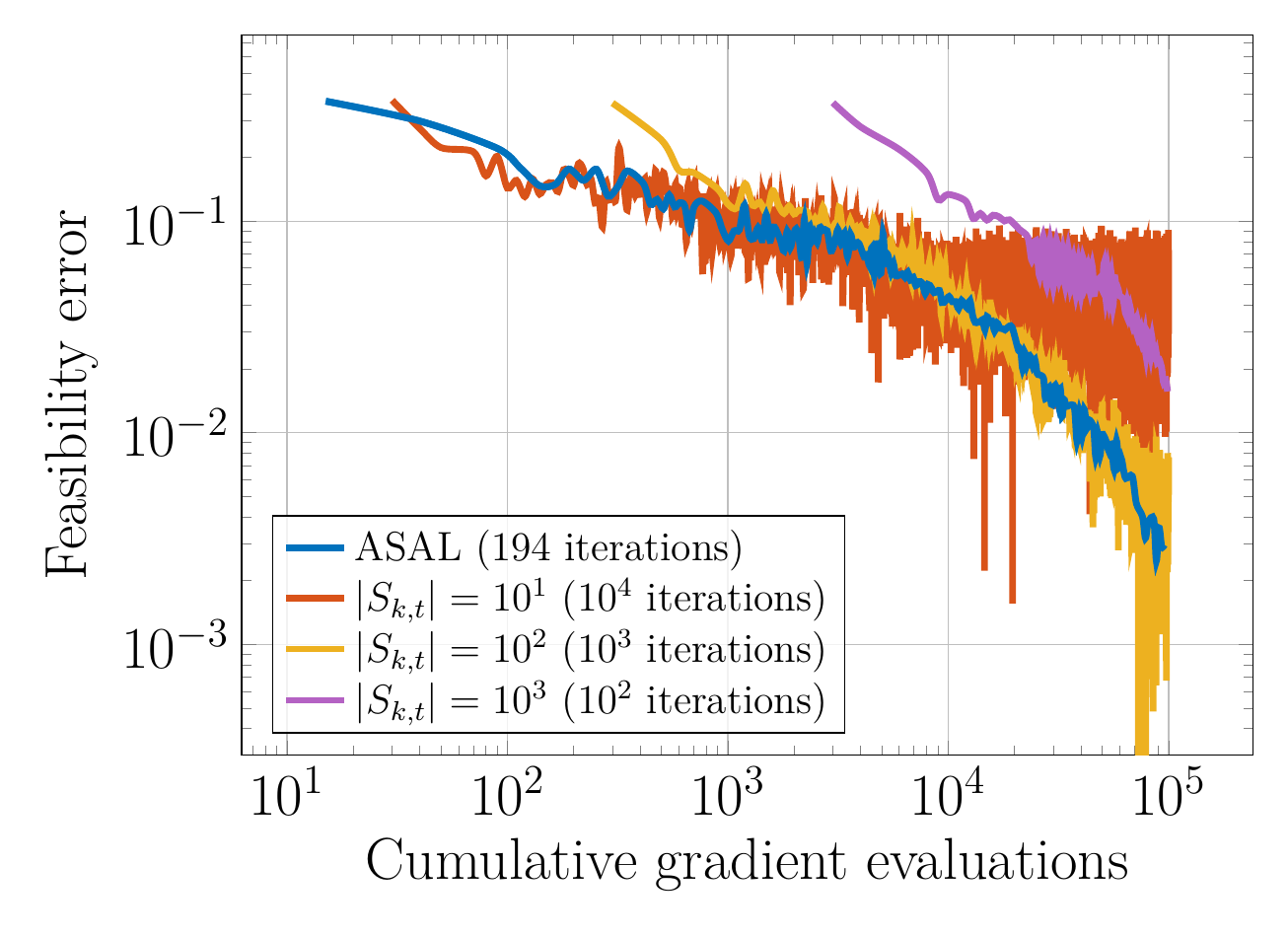}
\includegraphics[height=3.05cm]{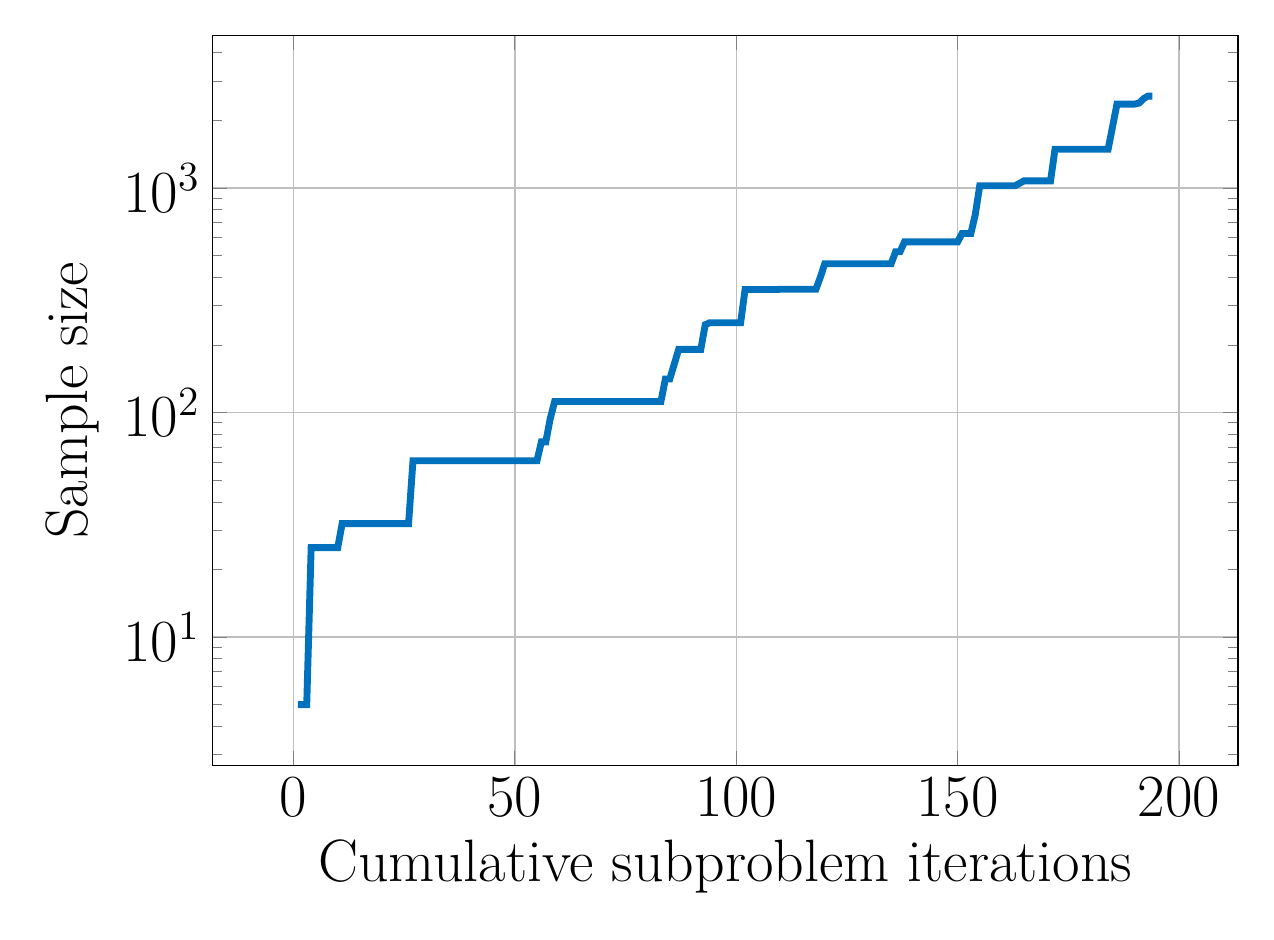}
\caption{
	Results of running ASAL (\Cref{alg:ASAL}) on the topology optimization problem~\cref{eqs:thermal_compliance}.
    Notice that ASAL achieves the lowest combined average stationary and feasibility errors while simultaneously requiring the smallest number of iterations.
    Indeed, the average stationarity error for the best-performing fixed sample size run, $|S_{k,t}| = 10^j$, $j=1,2,3$, ends up larger than the average stationarity error with ASAL.
    Moreover, although the average feasibility errors of ASAL and the best-performing fixed sample size run are similar, the variance for the fixed sample size run is much greater.
	To generate these results, we used the algorithm parameters $\theta_g = 2$, $\alpha = 0.1$, $\eta = 2.0$, and primal tolerance sequence $\tau_k = 1/k$.
	\label{fig:topology_plots}}
\end{figure}

\section{Final Remarks}\label{sec:final_remarks}

Motivated by a growing interest in developing optimization algorithms for constrained stochastic optimization problems, we introduced a framework that combines augmented Lagrangian methods with adaptive sampling techniques. In our framework, we employed stochastic solvers for the subproblems and imposed stochastic tolerance criteria for the inexact solutions. We analyzed various theoretical tolerance conditions and designed a practical test. To establish convergence results, we first showed that our framework is equivalent to an inexact gradient descent algorithm on the Moreau envelope. Second, we showed sublinear convergence in the outer iterations when $f$ is convex and linear convergence when $f$ is strongly convex with $\X = \R^n$. We also analyzed the worst-case expected work complexity of our approach in terms of the number of gradient evaluations required to obtain an $\epsilon$-accurate solution. For convex $f$ and compact $\X$, we showed $\mathcal{O}(\epsilon^{-3-\delta})$ complexity where $\delta > 0$ is a user-defined parameter. This result improves to $\mathcal{O}(\epsilon^{-2})$ when the penalty parameter $\alpha = \mathcal{O}(\epsilon^{-1})$. If $f$ is strongly convex and $\X = \R^n$, we proved $\mathcal{O}(\epsilon^{-1}\log(1/\epsilon))$ complexity.

To evaluate  our framework's practical performance, we tested it on a constrained machine learning problem and in engineering applications. Here, we observed that our method minimizes the objective function more efficiently and reaches a feasible solution in a more stable manner than benchmark stochastic approximation algorithms.

Although our analysis holds for any penalty parameter $\alpha>0$, this parameter should be tuned for optimal performance in practice. The other main hyperparameters are the step size $\eta > 0$ for the inner problems, and the subproblem tolerance values $\tau_k>0$. Since tuning is computationally expensive, it would be helpful to develop methods that adaptively select these hyper-parameters in order to further improve the practical efficacy of our adaptive sampling framework.
Two other natural extensions would be generalizing our methods to include nonlinear constraints and chance constraints.
\appendix
\section{Proof of \protect\Cref{prop:strnconv}}\label{sec:appendix_proof}

\begin{proof}
Due to the strong convexity of $f$, \cref{eq:dualFunction} has a unique optimal solution, denoted as $x(\lambda)$. Using \cite[Corollary 4.5.3]{hiriart2013convex}, we can show that $q(\lambda)$ is differentiable and
\begin{equation}\label{eq:graddual}
\nabla q (\lambda) = A x(\lambda) - b.
\end{equation} Moreover, from the optimality conditions of \cref{eq:dualFunction}, we have 
\begin{equation}\label{eq:optconddualfunc}
\nabla f(x(\lambda)) - A^T \lambda  = 0.
\end{equation}
Let $\lambda_1,\lambda_2 \in \R^m$. Consider, 
\begin{align*}
    \langle \nabla q(\lambda_2) - \nabla q(\lambda_1), \lambda_2 - \lambda_1\rangle &= \langle A(x(\lambda_2) - x(\lambda_1)), \lambda_2 - \lambda_1\rangle  \\
    &= \langle x(\lambda_2) - x(\lambda_1), A^T(\lambda_2 - \lambda_1) \rangle \\
        &=\langle x(\lambda_2) - x(\lambda_1), \nabla f(x(\lambda_2)) - \nabla f(x(\lambda_1)\rangle \\
    &\geq \frac{1}{\mu + L}\|\nabla f(x(\lambda_2)) - \nabla f(x(\lambda_1)\|^2 \\
    &= \frac{1}{\mu + L}\|A^T(\lambda_2 - \lambda_1)\|^2\\
        &\geq \frac{\lambda_{\min}(AA^T)}{\mu + L}\|\lambda_2 - \lambda_1\|^2
    \,,
\end{align*}
where the first  equality is due to \cref{eq:graddual}, the second and the third equalities are due to \cref{eq:optconddualfunc}, and the first inequality is due to  \cite[Theorem 2.1.11]{nesterov2003introductory}. 
Therefore, using \cite[Theorem 2.1.9]{nesterov2003introductory}, we can claim that $q(\lambda)$ is strongly convex with parameter $\frac{\sigma}{L + \mu}$. 
\end{proof}

\section{Logistic regression with multiple disparate impact constraints, \texttt{australian} dataset}\label{sec:appendix}

\begin{figure}
\centering
\includegraphics[height=3.05cm]{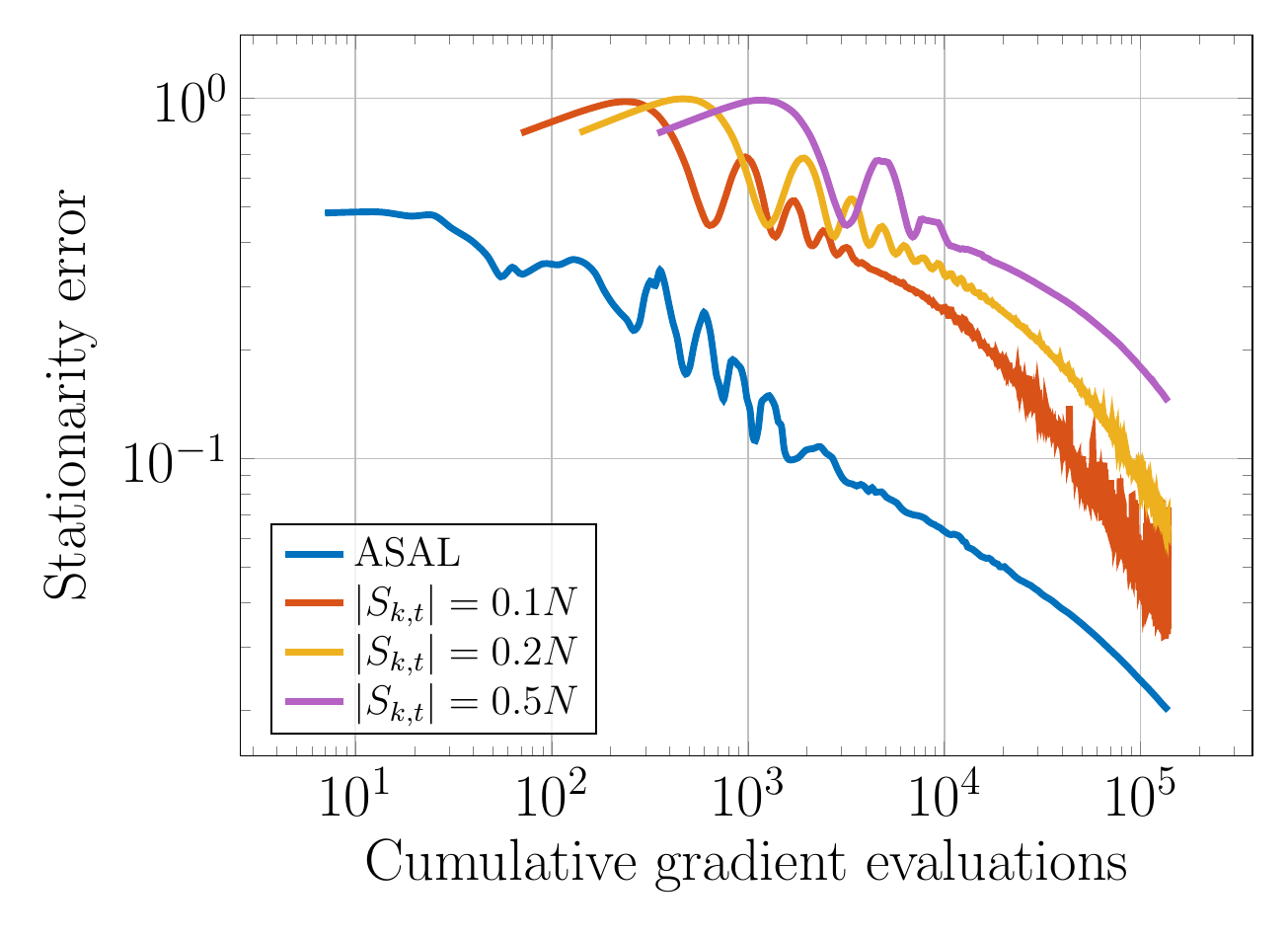}
\includegraphics[height=3.05cm]{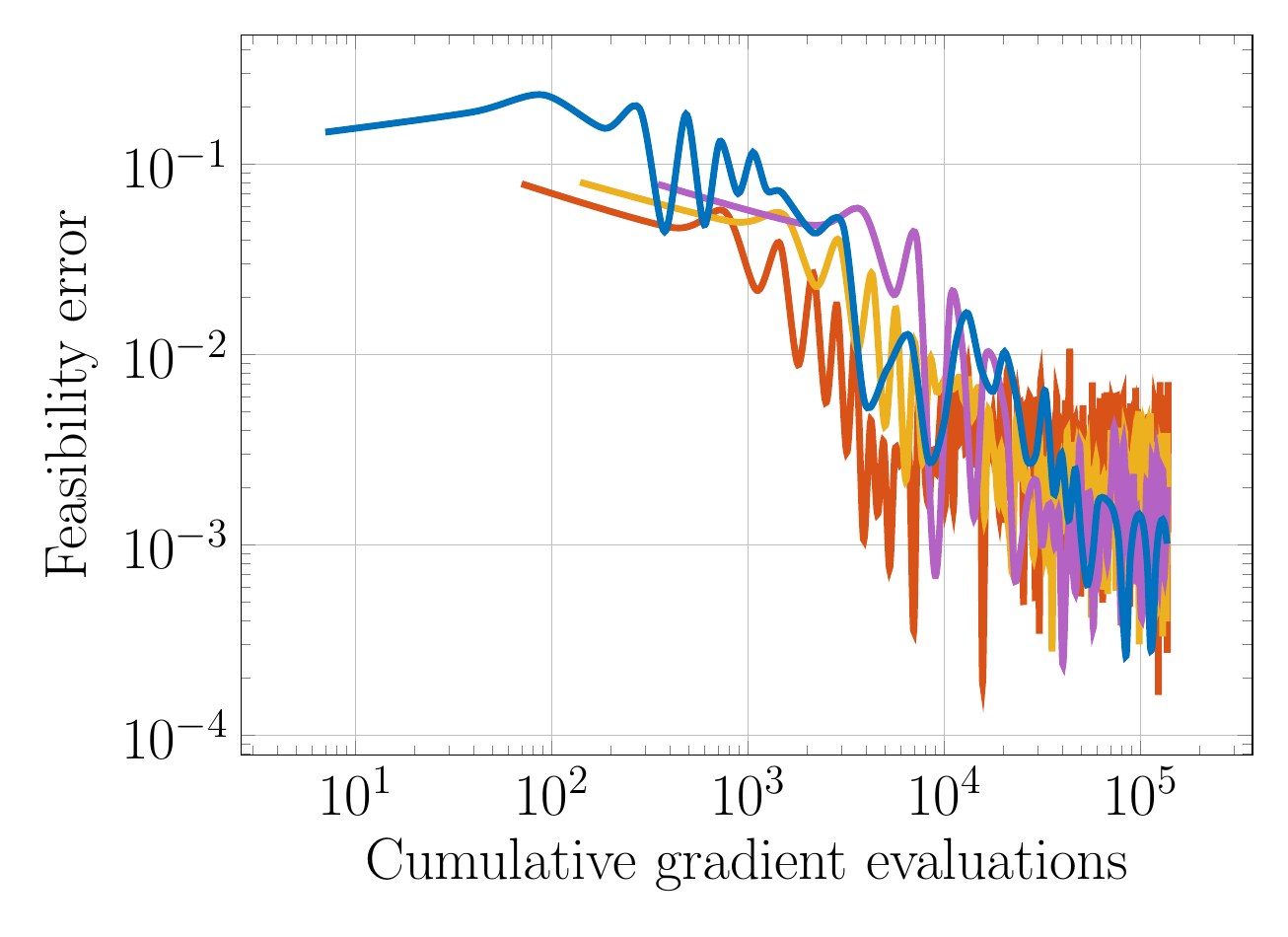}
\includegraphics[height=3.05cm]{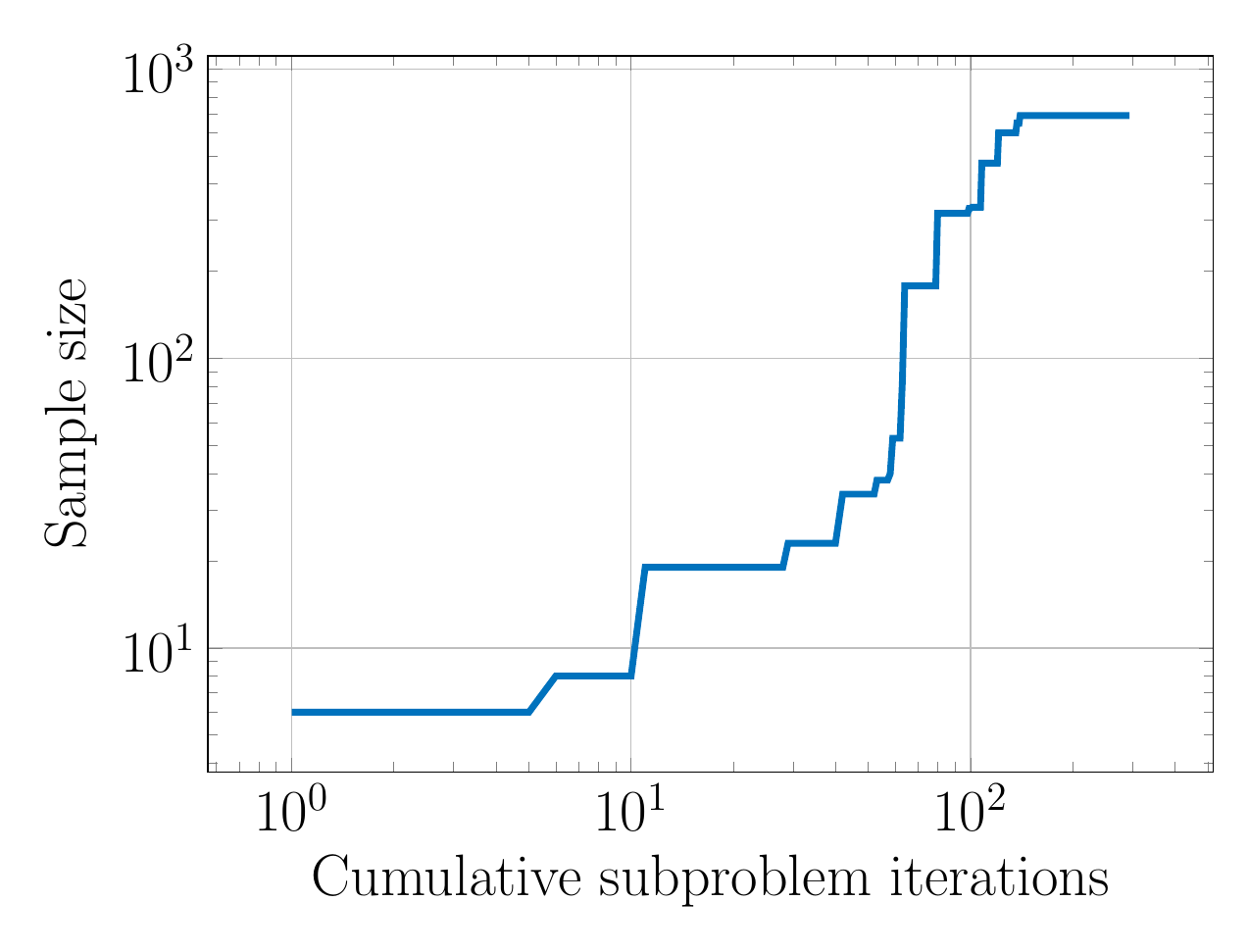}
\caption{
    Results of running \Cref{alg:ASAL} on the constrained logistic regression problem~\cref{eq:log_reg_prob} with the \texttt{australian} classification data set.
    Notice that ASAL achieves the lowest average stationarity errors while matching the minimal feasibility error of the three baseline algorithms.
    To generate these results, we used the algorithm parameters  $\theta_g= 0.99$ and individually tuned $\alpha$, $\eta$, and $\tau_0$.
\label{fig:australian_plots}}
\end{figure}

We consider problem \cref{eq:log_reg_prob} with \texttt{australian} classification data set from the LIBSVM collection \cite{10.1145/1961189.1961199}.
The data set has $N = 690$ rows, and the dimension of the problem is $n=14$.
Considering the budget of cumulative gradient evaluations as 200$N$, and the fixed hyperparameters as $\theta_g = 0.99$, $\nu_{\mathrm{l}} = 0.5$, $s_{\mathrm{l}} = 0.1N, s_{\min} = 0.1N$, we compare three separately-tuned fixed-batch-size implementations of ASAL using $10\%,~20\%,$ and $50\%$ of the data set size. We tune $\tau_0$, $\alpha$ and the step size $\eta$ using the same procedure described in \Cref{sub:logistic_regression_with_disparate_impact_constraints}
with the sets of $\tau_0 = 10^{i-1}$, $\eta = 10^{j-5}$, and $\alpha = 10^{j-4}$, where $i=0,1,2,3,4,5$ and $j=0,1,2,3,4,5,6$.

For each algorithm, we select the run with the smallest average objective function value in the final 10 inner iterations among all runs whose minimum feasibility error in the final 50 inner iterations is less than the feasibility tolerance $10^{-3}$. These values (i.e., $10,\,50$, and $10^{-3}$, respectively) are slightly different than the values given in \Cref{sub:logistic_regression_with_disparate_impact_constraints} to ensure that the best combinations of hyperparameter values correspond to a more stable set of runs. Because of the same reason, we restrict $\alpha = 10^{-1}$ for the ASAL algorithm while tuning, as we observe this value results in choosing runs that show a good balance between stationarity and feasibility errors. The comparison of the algorithms is given in \Cref{fig:australian_plots}. Similar to \Cref{sub:logistic_regression_with_disparate_impact_constraints},
we observe that ASAL and each of the three baseline algorithms achieve a similar \emph{minimal} feasibility error (around feasibility tolerance $10^{-3}$) and that ASAL performs better than the three baseline algorithms when it comes to stationarity error.

\phantomsection\bibliographystyle{siamplain}
\bibliography{main.bib}

\end{document}